\documentclass{article}[12pt]
\usepackage{amsmath}
\usepackage{amssymb}
\usepackage{extarrows}

\newtheorem{theorem}{Theorem}

\newtheorem{lemma}[theorem]{Lemma}

\newtheorem{proposition}[theorem]{Proposition}
\newtheorem{remark}[theorem]{Remark}

\newenvironment{proof}[1][Proof]{\noindent\textbf{#1.} }{\ \rule{0.5em}{0.5em}}
\usepackage{amssymb}
\usepackage{graphicx}
\usepackage[dvipsnames,usenames]{color}

\input{epsf.tex}
\usepackage[affil-it]{authblk}
\author[a]{Ivo J.B.F. Adan\footnote{i.adan@tue.nl}}
\author[b]{Ioannis Dimitriou \footnote{ idimit@uoi.gr}\footnote{Corresponding author.}}
\affil[a]{\small Department of Industrial Engineering and Innovation Sciences, Eindhoven University of Technology, P.O. Box 513, Eindhoven, MB 5600, the Netherlands}
\affil[b]{\small Department of Mathematics, 
	University of Ioannina, 
	45110, Ioannina, Greece.}
\begin{document}
\title{A finite compensation procedure for a class of two-dimensional random walks}

\maketitle
\begin{abstract}
    Motivated by queueing applications, we consider a  class of two-dimensional random walks, the invariant measure of which can be written as a linear combination of a finite number of product-form terms. In this work, we investigate under which conditions such an elegant solution can be derived by applying a finite compensation procedure. The conditions are formulated in terms of relations among the transition probabilities in the inner area, the boundaries as well as the origin. A discussion on the importance of these conditions is also given.
    \end{abstract}
    \vspace{2mm}
	
\section{Introduction}\label{intro}
In this work, motivated by queueing problems modeled as random walks on multi-dimensional grids, we investigate under which conditions the solution to the equilibrium equations of a certain class of two-dimensional random walks can be given as a linear combination of a finite number of product-form terms. An elegant method to construct such a linear combination, essentially consists of first finding the product-form terms satisfying the interior equilibrium equations, by confronting these solutions with the boundary equations, and then building a linear combination that also satisfies the boundary equilibrium equations. 

The most simple application of this method refers to the case where a single product-form satisfies the interior, as well as the boundary equilibrium equations. Specifically, product-form queueing networks belong to this class of multi-dimensional random walks; e.g., see \cite{baskett}, \cite[Chapters 1, 5, 6]{bouch}. See also \cite[Chapter 9]{bouch} for queueing networks that do not have a product-form invariant measure, but can be approximated by perturbing the transition probabilities so as to obtain a product-form invariant measure.

Under certain conditions, the compensation method developed in \cite{ivothesis,adanaplprob}, constructs a linear combination of (infinite or finite) product-form terms that satisfy the equilibrium equations in the interior of the quarter plane, and are chosen such that the equilibrium equations on the boundaries are satisfied as well. In particular, the compensation method implies that the invariant measure of a two-dimensional random walk can be written as an infinite series of product-form terms for all states away of the origin, provided the following fundamental conditions are satisfied \cite{ivothesis,adanaplprob}:
\begin{enumerate}
     \item \textit{Step size:} only transitions to neighboring states are allowed;
\item \textit{Forbidden transitions:} transitions from any interior state to the East, North, North-East are never allowed;
\item \textit{Semi-homogeneity:} for all interior states, the transitions occur at the same rates, and similarly for all states on the horizontal boundary, and for all states on the vertical boundary. 
\end{enumerate}

Our work is strongly motivated by queueing problems which can be solved by a linear combination of product-form terms. In some cases, e.g., $E_{k}/E_{r}/c$  queues \cite{adaner}, or in specific multidimensional queues \cite{sele,semi,lock}, finitely many terms are needed. In other cases, e.g., the shortest queue problem \cite{adan1,adan2}, the multiprogramming queues problem \cite{adanmult}, the $2\times 2$ clocked buffered switch of an interconnection network \cite{boxhout}, etc, infinitely many terms are necessary. For the latter model, the compensation method was further extended to a three-dimensional case \cite{hout}. Moreover, the authors in \cite{boxhout} indicated a link between the compensation approach and the boundary value method \cite{cohbox,fay1}. The boundary value method aims to solve the equilibrium equations by introducing the generating function of the equilibrium distribution, and studying the functional equations that it should satisfy. Typically, in solving these functional equations formidable difficulties may arise. However, for a class of two-dimensional random walks and several queueing problems, the technique developed in \cite{cohbox,fay1} reduces those functional equations to standard Riemann(-Hilbert) boundary value problems, and to singular integral equations for complex-valued functions. A concise exposition of the method, along with several
applications and references, was presented in \cite{coh7}, while a detailed investigation of random walks in the quarter plane was further continued by J.W. Cohen in \cite{coh8}, which among others, provided insight on the ergodicity conditions and the importance of the boundary hitting points. In \cite{coh1,coh2}, J.W. Cohen considerably contributed to the better understanding of the link among the compensation method and the boundary value method. In particular, in \cite{coh1} he studied a class of two-dimensional nearest neighbour random walks without transitions to the North, North-East and East, i.e., the model considered first in \cite{ivothesis}, and showed that the bivariate generating function of the stationary distribution can be represented by a meromorphic function, i.e., an analytic function apart from a finite number of poles in every finite domain. These poles were presented as powers in the product-form terms in \cite{ivothesis}. In \cite{coh2,coh4,coh3}, by studying the two-dimensional shortest queue model, he showed how all poles, and all zeros, of the
meromorphic generating function can be determined from the original functional equation. His work led to simple expressions for the main performance metrics, which can be calculated with
any desired accuracy. Cohen also applied his approach to the symmetric \cite{coh5}, and the asymmetric \cite{coh6} $2\times 2$ clocked buffered switch.

In our work, we cope with a certain class of two-dimensional random walks by allowing transitions from the interior to the North and to the East (thus, violating a fundamental assumption of the \textit{standard} compensation approach), and apply a finite compensation procedure. More precisely, the invariant measure of this class of two-dimensional random walk can be written as a \textit{finite} sum of product-form terms. This work is strongly motivated by the non-work conserving discrete time two-queue system with Bernoulli arrivals that was recently analyzed in \cite{devos}. For this model, by using the generating function technique and complex analytic arguments, the authors solved the functional equation and derived the stationary joint queue-length distribution as a sum of three product-form terms. Although their work was based on a powerful mathematical method, it did not reveal the special features of the corresponding two-dimensional random walk that help to have such an elegant solution. Contrary to the work in \cite{devos}, the analysis in the present paper is based directly on the set of equilibrium equations, which enables a thorough investigation of the conditions that are responsible for having such an elegant solution. Recently, in \cite{chen,chen2}, the authors provided necessary conditions which characterize random walks, the invariant measure of which is a sum of geometric terms. They also developed an approximation scheme and error bounds for two-dimensional random walks, the invariant measure of which is not a sum of geometric terms; see also \cite{Goseling2016} for an approach to approximate by a single product-form the stationary distribution of random walks in the quarter plane. In line with the works in \cite{chen,chen2}, the author in \cite[Chapters 5, 6]{bai} proposed perturbation schemes for two dimensional random walks. In particular, she considered constructing inhomogeneous transition rates for the perturbed random walk such that its stationary distribution is given as a sum of a finite number of geometric terms, and gave an explicit expression for the error bound. We also mention \cite{diek}, which refers to a reflected Brownian motion with constraints on the boundary transition probabilities (i.e., conditions on the covariance and the reflection matrix), and their results are similar to those reported in the present paper. In particular, they showed that for the invariant measure to be a linear combination of finitely many exponential measures, there must be an odd number of terms generated by a mating procedure, which operates in a similar fashion as the compensation procedure. This mating procedure is terminated after a finite number of steps, when at the final step we end up with the initial measure. Our work in the present paper can be seen as a discrete state space analogue of the method used in \cite{diek}.

The goal of the present paper is to investigate whether it is possible to generalize the
compensation approach to obtain the invariant measure with a finite number of product-form terms, when at the same time we violate a fundamental requirement, i.e., by allowing transitions from the interior to the North and East.
\paragraph{Contribution.} We characterize a class of two dimensional random walks for which the invariant measure can be obtained as a sum of exactly three product-form terms by using a finite compensation procedure, i.e., after introducing the first term, new terms are subsequently added to
compensate for the error of the previous term on one of the two boundaries. By revealing structural properties of the transition probabilities, we formulate conditions under which the compensation procedure stops after exactly three steps, and results in the same solution, irrespective whether compensation is started on the vertical or horizontal boundary.  

The rest of the paper is summarized as follows. In Section \ref{general}, we describe the general model in detail, and present our finite compensation approach step-by-step. A modification of the general model, which requires only a single product-form term, is also given. In Section \ref{motivation}, we apply the theoretical results of Section 2, to two queueing models. We also study a queueing model with geometric batch arrivals and using a similar methodological framework we show that its stationary distribution is of product-form. In Section \ref{discussion}, we briefly discuss the importance of the conditions that are used in Section \ref{general}, and yield such an elegant result. Some queueing examples that violate these conditions and do not have such an elegant solution, are also given. We also discussed an example where we show how the conditions are adapted when one of them is violated in order to retain a product-form solution. Finally, some future research directions are discussed.  
\section{The model and the equilibrium equations}\label{general}
Consider a two-dimensional random walk $Q := \{Q_{s}, s=0,1,\ldots\}$ having state space $S=\mathbb{Z}_{+}^{2} = \{(m,n); m,n\in\mathbb{Z}_{+}\}$, where $\mathbb{Z}_{+}$ represents the set of nonnegative integers.  Further associated with $S$ are the subsets $\{H_{n}\}_{n \geq 0}$, $\{V_{m}\}_{m \geq 0}$, where
\begin{displaymath}
\begin{array}{lr}
    V_{m} := \{(m,n):  n \in \mathbb{Z}_{+}\},\,m\in\mathbb{Z}_{+}, &
      H_{n} := \{(m,n):  m \in \mathbb{Z}_{+}\},\,n\in\mathbb{Z}_{+}. 
\end{array}
\end{displaymath}
In particular, we refer to $H_{0}\setminus\{(0,0)\}$ as the horizontal boundary of $S$, $V_{0}\setminus\{(0,0)\}$ as the vertical boundary of $S$, and $I := S \setminus (H_{0} \cup V_{0})$ as the interior of $S$.

We consider {\em{nearest-neighbor}} two dimensional random walks, meaning that from any state $(m,n)$, a transition takes place to state $(i,j) \in S$, such that $\max(|m - i|, |n - j|) \leq 1.$ 
Denote by $p_{k,l}(m,n)$ the one step transition probabilities from $(m,n)$ to $(m+k,n+l)$, and let $\mathbf{P} := [p_{(m,n),(m^{\prime},n^{\prime})}]$ be its transition matrix, i.e., $p_{(m,n),(m+k,n+l)}=p_{k,l}(m,n)$. The process is homogeneous in the sense that for each pair $(m,n)$, $(m^{\prime},n^{\prime})$ in the interior (respectively on the horizontal
 and on the vertical boundary) of $S$,
\begin{displaymath}
p_{k,l}(m,n)=p_{k,l}(m^{\prime},n^{\prime})\text{ and }p_{k,l}(m-k,n-k)=p_{k,l}(m^{\prime}-k,n^{\prime}-k),
\end{displaymath} 
for all $-1\leq k,l\leq 1$. Let $q_{k,l}:=p_{k,l}(m,n)$, $(m,n)\in I$, $q_{k,l}^{(h)}:=p_{k,l}(m,0)$, $(m,0)\in H_{0}\setminus\{(0,0)\}$, with $q_{k,-1}^{(h)}=0$, $-1\leq k\leq 1$, $q_{k,l}^{(v)}:=p_{k,l}(0,n)$, $(0,n)\in V_{0}\setminus\{(0,0)\}$, with $q_{-1,l}^{(v)}=0$, $-1\leq l\leq 1$, and $q_{k,l}^{(0)}:=p_{k,l}(0,0)$, with $q_{k,l}^{(0)}=0$, when $k=-1$ or $l=-1$. All other unspecified transition probabilities within $\mathbf{P}$ are equal to zero. Apart from these assumptions, the random walk $Q$, and also the component chains (i.e., the chains associated with the sets $H_{n}$, $n\in \mathbb{Z}_{+}$, $V_{m}$, $m\in \mathbb{Z}_{+}$, namely the $m-$component and the $n-$component chain, respectively), are assumed to be irreducible and aperiodic.
%

Our objective is to derive -- when it exists -- the stationary distribution $\boldsymbol{\pi} := [\pi_{m,n}]_{m,n \in \mathbb{Z}_{+}}$ of this random walk, and to determine nontrivial sufficient conditions for when each element of $\boldsymbol{\pi}$ can be expressed as a finite linear combination of product-form terms. With that in mind, some structural conditions on the transition probabilities should be imposed. 
First, we focus on random walks for which transitions to the North-East and South-West from any interior state are forbidden; see Figure \ref{p1} for an illustration of the transition structure of $Q$  (see also Subsection \ref{genn} for a discussion on relaxing this condition).\vspace{2mm}\\
\textbf{Condition A:} $q_{1,1}=q_{-1,-1}=0$. \\
Since our main goal is to establish conditions, so that the invariant measure can be written as a sum of a finite number of product-form terms, one may expect that some useful information can be derived from the marginals. More precisely, the following condition (named Condition B), helps to show that the marginal probabilities can be found explicitly, and are of geometric form.\vspace{2mm}\\
\textbf{Condition B:} This condition relates the transition probabilities in the interior with those at the boundaries (Condition B.1), as well as, the transition probabilities at the origin, with those in the interior and the boundaries (Condition B.2). 
In particular, Condition B.1 reads:
\begin{displaymath}
\begin{array}{rlcrl}
     q_{1,1}^{(h)} =& q_{1,0},&&  q_{1,1}^{(v)} = &q_{0,1},\vspace{2mm}\\
     q_{1,0}^{(h)} =& q_{1,-1},&& q_{0,1}^{(v)} = &q_{-1,1},\vspace{2mm}\\
     q_{-1,0}^{(h)} =& q_{-1,0}, &&q_{0,-1}^{(v)} = &q_{0,-1},\vspace{2mm}\\
     q_{-1,1}^{(h)} =& q_{-1,1},&&q_{1,-1}^{(v)} = &q_{1,-1}.
\end{array}
\end{displaymath}
%
%
Note that when $\mathbf{P}$ satisfies Conditions A, B.1, the following observation holds: 
\begin{displaymath}
\begin{array}{rlrl}
     q_{1,1}^{(h)}+q_{1,0}^{(h)} = &q_{1,0} + q_{1,-1},& q_{-1,0}^{(h)}+q_{-1,1}^{(h)}=&q_{-1,0}+q_{-1,1} \vspace{2mm}\\
     q_{1,1}^{(v)}+q_{0,1}^{(v)} =&q_{0,1}+q_{-1,1} ,& q_{0,-1}^{(v)}+q_{1,-1}^{(v)}=&q_{0,-1}+q_{1,-1},\vspace{2mm}\\
    q_{0,0}^{(v)} + q_{1,0}^{(v)} = &q_{0,0} + q_{-1,0} + q_{1,0},& q_{0,0}^{(h)} + q_{0,1}^{(h)} = &q_{0,0} + q_{0,-1} + q_{0,1}
\end{array}
\end{displaymath}
Condition B.2 reads:
\begin{equation}
    \begin{array}{rl}
q_{0,1}^{(0)}+q_{0,1}=&q_{0,1}^{(h)}+q_{0,1}^{(v)},\vspace{2mm}\\
q_{1,0}^{(0)}+q_{1,0}=&q_{1,0}^{(h)}+q_{1,0}^{(v)},\vspace{2mm}\\
q_{1,1}^{(0)}=&q_{1,1}^{(h)}+q_{1,1}^{(v)},
\end{array}\label{cvc}
\end{equation}
Note that, summing equations \eqref{cvc} results, after some algebra, in $q_{0,0}^{(0)}+q_{0,0}=q_{0,0}^{(h)}+q_{0,0}^{(v)}$.

A direct consequence of Conditions A, B.1, B.2, is the following:
\begin{displaymath}
\begin{array}{rl}
  q_{1,0}^{(0)} + q_{1,1}^{(0)} =& q_{1,0}^{(0)} + q_{1,1}^{(h)} + q_{1,1}^{(v)} \\ =& q_{1,0}^{(0)} + q_{1,0} + q_{0,1} \\ =& q_{1,0}^{(h)} + q_{1,0}^{(v)} + q_{0,1} \\= & q_{1,-1} + q_{1,0}^{(v)} + q_{0,1},\vspace{2mm}\\
  q_{0,1}^{(0)} + q_{1,1}^{(0)} =& q_{0,1}^{(0)} + q_{1,1}^{(h)} + q_{1,1}^{(v)} \\= & q_{0,1}^{(0)} + q_{0,1} + q_{1,0} \\ =& q_{0,1}^{(h)} + q_{0,1}^{(v)} + q_{1,0} \\ =& q_{0,1}^{(h)} + q_{-1,1} + q_{1,0}.
\end{array}
\end{displaymath}
\begin{figure}[ht!]
\centering
\includegraphics[scale=1]{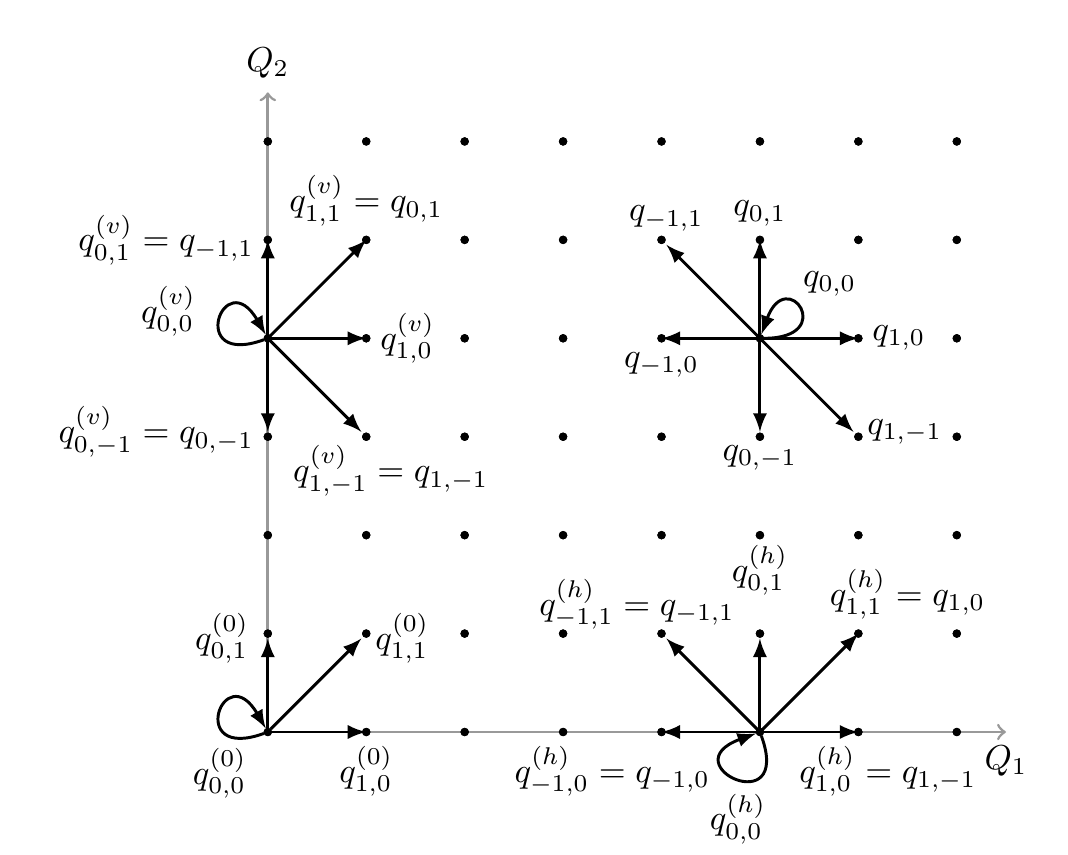}
\caption{The transition diagram.}\label{p1}
\end{figure}

Assume that $Q$ is positive recurrent (see Theorem \ref{stabi} below), having stationary distribution $\boldsymbol{\pi} := [\pi_{m,n}]_{m,n \in \mathbb{Z}_{+}}$, and let its marginal distributions for the $m-$component and $n-$component chain, denoted by $[\pi_{m}^{(1)}]_{m \in \mathbb{Z}_{+}}$, $[\pi_{n}^{(2)}]_{n \in \mathbb{Z}_{+}}$, respectively:
\begin{displaymath}
\pi_{m}^{(1)}=\sum_{n\in H_{n}}\pi_{m,n},\,m\in\mathbb{Z}_{+},\,\,\pi_{n}^{(2)}=\sum_{m\in V_{m}}\pi_{m,n},\,n\in\mathbb{Z}_{+}.
\end{displaymath}
Explicit formulae for the marginal distributions are directly derived by applying the balance principle
\begin{displaymath}
\text{the rate out of set }W=\text{the rate into set }W,
\end{displaymath}
to the sets $H_{n}$, $V_{m}$, $m,n\in \mathbb{Z}_{+}$ (see also Figure \ref{p1}). In particular, for $m > 1$,
\[
(q_{-1,0}+q_{-1,1})\pi_{m}^{(1)} = (q_{1,0}+q_{1,-1})\pi_{m-1}^{(1)},
\]
so
\[
\pi_{m}^{(1)}=\rho_{1}\pi_{m-1}^{(1)} = \cdots = \rho_{1}^{m-1}\pi_{1}^{(1)} ,
\]
where $\rho_{1}:=\frac{q_{1,0}+q_{1,-1}}{q_{-1,0}+q_{-1,1}}$. Moreover, due to Condition B.2,
\begin{displaymath}
\begin{array}{rl}
     (q_{-1,0}+q_{-1,1})\pi_{1}^{(1)}=&(q_{1,0}^{(v)}+q_{1,-1}+q_{0,1})\pi_{0}^{(1)}\Leftrightarrow \pi_{1}^{(1)}=\frac{q_{1,0}^{(v)}+q_{1,-1}+q_{0,1}}{q_{1,0}+q_{1,-1}}\rho_{1}\pi_{0}^{(1)},
\end{array}
\end{displaymath}
thus,
\begin{equation}
    \pi_{m}^{(1)}=\frac{q_{1,0}^{(v)}+q_{1,-1}+q_{0,1}}{q_{1,0}+q_{1,-1}}\rho_{1}^{m}\pi_{0}^{(1)},\quad m\geq 1.\label{mar11}
\end{equation}
Similarly,
\begin{equation}
    \pi_{n}^{(2)}=\frac{q_{0,1}^{(h)}+q_{-1,1}+q_{1,0}}{q_{0,1}+q_{-1,1}}\rho_{2}^{n}\pi_{0}^{(2)},\quad n\geq 1,\label{mar22}
\end{equation}
where $\rho_{2}:=\frac{q_{0,1}+q_{-1,1}}{q_{0,-1}+q_{1,-1}}$.

Having in mind Conditions A, B (i.e., Conditions B.1, B.2), it is readily seen that the component Markov chains are random walks on $\mathbb{Z}_{+}$, with negative drift, and thus ergodic when, $\rho_{1}<1$, and $\rho_{2}<1$. 
In particular, using \eqref{mar11}, \eqref{mar22} we get the following geometric distributions:
\begin{equation}
    \pi_{m}^{(1)}=\left\{\begin{array}{ll}
         \frac{(1-\rho_{1})(q_{1,0}+q_{1,-1})}{(1-\rho_{1})(q_{1,0}+q_{1,-1})+\rho_{1}(q_{1,0}^{(0)}+q_{1,1}^{(0)})}\frac{q_{1,0}^{(v)}+q_{1,-1}+q_{0,1}}{q_{1,0}+q_{1,-1}}\rho_{1}^{m},&m\geq 1,\\
         \frac{(1-\rho_{1})(q_{1,0}+q_{1,-1})}{(1-\rho_{1})(q_{1,0}+q_{1,-1})+\rho_{1}(q_{1,1}^{(0)}+q_{1,0}^{(0)})},&m=0, 
    \end{array}\right.\label{mar1}
\end{equation}
\begin{equation}
    \pi_{n}^{(2)}=\left\{\begin{array}{ll}
         \frac{(1-\rho_{2})(q_{0,1}+q_{-1,1})}{(1-\rho_{2})(q_{0,1}+q_{-1,1})+\rho_{2}(q_{0,1}^{(0)}+q_{1,1}^{(0)})}\frac{q_{0,1}^{(h)}+q_{-1,1}+q_{1,0}}{q_{0,1}+q_{-1,1}}\rho_{2}^{n},&n\geq 1,\\
         \frac{(1-\rho_{2})(q_{0,1}+q_{-1,1})}{(1-\rho_{2})(q_{0,1}+q_{-1,1})+\rho_{2}(q_{0,1}^{(0)}+q_{1,1}^{(0)})},&n=0. 
    \end{array}\right.\label{mar2}
\end{equation}
\begin{remark}
The simple form of the marginal distributions in \eqref{mar1}, \eqref{mar2} is due to the Conditions A, B. For example, for all $m>0$ (see Figure \ref{p1}), the total flow rate to the right equals $q_{1,0}+q_{1,-1}$ and the total flow rate to the left
equals $q_{-1,0}+q_{-1,1}$.
\end{remark}

It can be easily shown that the random walk $Q$ is positive recurrent if and only if both component random walks are positive recurrent, i.e., if and only if the component random walks
have negative drifts \cite{fayo}. The next theorem provides a necessary and sufficient condition for the ergodicity of $Q$.
\begin{theorem}\label{stabi}
Assume that conditions A, B.1 are satisfied. Then, $Q$ is ergodic if and only if
\begin{equation}
\rho_{1}<1,\text{ and }\rho_{2}<1.
\label{stab}
\end{equation}
\end{theorem}
\begin{proof}
The proof is based on \cite[Theorem 1.2.1]{fayo}. Define,
\begin{displaymath}
\begin{array}{rl}
     M=(M_{x},M_{y})=&(\sum_{i,j}iq_{i,j},\sum_{i,j}jq_{i,j})=(q_{1,0}+q_{1,-1}-(q_{-1,1}+q_{-1,0}),q_{0,1}+q_{-1,1}-(q_{1,-1}+q_{0,-1})),  \\
    (M_{x}^{h},M_{y}^{h})=&(\sum_{i,j}iq^{(h)}_{i,j},\sum_{i,j}jq_{i,j}^{(h)})=(M_{x},q_{-1,1}+q_{0,1}^{(h)}+q_{1,0}),\\
    (M_{x}^{v},M_{y}^{v})=&(\sum_{i,j}iq^{(v)}_{i,j},\sum_{i,j}jq_{i,j}^{(v)})=(q_{1,-1}+q_{1,0}^{(v)}+q_{0,1},M_{y}).
\end{array}
\end{displaymath}
Assuming $M\neq 0$, $Q$ is ergodic \cite[Theorem 1.2.1]{fayo} if and only if one of the following conditions holds:
\begin{enumerate}
    \item $M_{x}<0$, $M_{y}<0$, $M_{x}M_{y}^{h}-M_{y}M_{x}^{h}<0$, $M_{y}M_{x}^{v}-M_{x}M_{y}^{v}<0$,
    \item $M_{x}<0$, $M_{y}\geq 0$, $M_{y}M_{x}^{v}-M_{x}M_{y}^{v}<0$,
    \item $M_{x}\geq 0$, $M_{y}<0$, $M_{x}M_{y}^{h}-M_{y}M_{x}^{h}<0$.
\end{enumerate}
Note from Condition 1. that $M_{x}<0$, $M_{y}<0$, is equivalent to $\rho_{1}<1$, $\rho_{2}<1$, respectively, and 
\begin{displaymath}
M_{x}M_{y}^{h}-M_{y}M_{x}^{h}= M_{x}(M_{y}^{h}-M_{y})<0,
\end{displaymath}
since $M_{x}<0$, $M_{y}^{h}>0$, $M_{y}<0$. Similarly, 
$M_{y}M_{x}^{v}-M_{x}M_{y}^{v}<0$, also holds. Thus, $\rho_{1}<1$, $\rho_{2}<1$ is the ergodicity condition.

Note that Conditions 2., 3. do not hold. Indeed, from Condition 2., $M_{x}<0$, $M_{y}\geq 0$, implies that $\rho_{1}<1$, $\rho_{2}\geq 1$. Then, $M_{y}M_{x}^{v}-M_{x}M_{y}^{v}=M_{y}(M_{x}^{v}-M_{x})\geq 0$, since $M_{x}<0$, $M_{x}^{v}>0$, $M_{y}\geq 0$. Similarly, we can show that Condition 3. does not hold. Thus, $Q$ is ergodic if and only if $\rho_{1}<1$, $\rho_{2}<1$.
\end{proof}\vspace{2mm}\\
From here on consider the following assumption.\\
\noindent
\textbf{Assumption:} $\rho_{1}<1$, $\rho_{2}<1$.\vspace{2mm}\\
$\boldsymbol{\pi}$ is the unique normalized solution of the following equilibrium equations: 
\begin{equation}
\begin{array}{rl}
\pi_{m,n}(1-q_{0,0})=&\pi_{m+1,n-1}q_{-1,1}+\pi_{m,n-1}q_{0,1}+\pi_{m-1,n}q_{1,0}+\pi_{m-1,n+1}q_{1,-1}\vspace{2mm}\\
&+\pi_{m,n+1}q_{0,-1}+\pi_{m+1,n}q_{-1,0},\,m>1,n>1,
\end{array}\label{int}
\end{equation}
\begin{equation}
\begin{array}{rl}
\pi_{1,n}(1-q_{0,0})=&\pi_{2,n-1}q_{-1,1}+\pi_{1,n-1}q_{0,1}+\pi_{0,n}q_{1,0}^{(v)}+\pi_{0,n+1}q_{1,-1}^{(v)}\vspace{2mm}\\
&+\pi_{1,n+1}q_{0,-1}+\pi_{2,n}q_{-1,0}+\pi_{0,n-1}q_{1,1}^{(v)},\,n>1,
\end{array}\label{v1}
\end{equation}
\begin{equation}
\begin{array}{rl}
\pi_{0,n}(1-q_{0,0}^{(v)})=&\pi_{1,n-1}q_{-1,1}+\pi_{0,n-1}q_{0,1}^{(v)}+\pi_{0,n+1}q_{0,-1}^{(v)}+\pi_{1,n}q_{-1,0},\,n>1,
\end{array}\label{v2}
\end{equation}
\begin{equation}
\begin{array}{rl}
\pi_{m,1}(1-q_{0,0})=&\pi_{m+1,0}q_{-1,1}^{(h)}+\pi_{m,0}q_{0,1}^{(h)}+\pi_{m-1,1}q_{1,0}+\pi_{m-1,2}q_{1,-1}\vspace{2mm}\\
&+\pi_{m,2}q_{0,-1}+\pi_{m+1,1}q_{-1,0}+\pi_{m-1,0}q_{1,1}^{(h)},\,m>1,
\end{array}\label{h1}
\end{equation}
\begin{equation}
\begin{array}{rl}
\pi_{m,0}(1-q_{0,0}^{(h)})=&\pi_{m+1,0}q_{-1,0}^{(h)}+\pi_{m,1}q_{0,-1}+\pi_{m-1,1}q_{1,-1}+\pi_{m-1,0}q_{1,0}^{(h)},\,m>1.
\end{array}\label{h2}
\end{equation}
\begin{equation}
\pi_{0,0}(1-q_{0,0}^{(0)})=\pi_{1,0}q_{-1,0}^{(h)}+\pi_{0,1}q_{0,-1}^{(v)},
\label{00}
\end{equation}
\begin{equation}
\pi_{0,1}(1-q_{0,0}^{(v)})=\pi_{0,0}q_{0,1}^{(0)}+\pi_{1,0}q_{-1,1}^{(h)}+\pi_{1,1}q_{-1,0}+\pi_{0,2}q_{0,-1}^{(v)},\label{01}
\end{equation}
\begin{equation}
\pi_{1,0}(1-q_{0,0}^{(h)})=\pi_{0,0}q_{1,0}^{(0)}+\pi_{0,1}q_{1,-1}^{(v)}+\pi_{1,1}q_{0,-1}+\pi_{2,0}q_{-1,0}^{(h)},\label{10}
\end{equation}
\begin{equation}
\pi_{1,1}(1-q_{0,0})=\pi_{0,0}q_{1,1}^{(0)}+\pi_{0,1}q_{1,0}^{(v)}+\pi_{1,2}q_{0,-1}+\pi_{2,0}q_{-1,1}^{(h)}+\pi_{0,2}q_{1,-1}^{(v)}+\pi_{2,1}q_{-1,0}+\pi_{1,0}q_{0,1}^{(h)},\label{11}
\end{equation}
\subsection{The \textit{finite} compensation procedure}\label{compe}
Our aim is to establish conditions under which the compensation approach can be used to obtain the invariant measure of two-dimensional random walks with a transition diagram as given in Figure \ref{p1}. In general, the compensation approach yields an explicit expression by directly exploiting the equilibrium equations, without using any transforms, when the conditions mentioned in Section \ref{intro} are satisfied \cite{ivothesis,adanaplprob}.

The compensation approach aims to solve the equilibrium equations by a linear combination of product-form terms. After introducing an initial product-form term that satisfies the equilibrium equations at the interior state space, additional product-form terms are added so as to alternately compensate for the errors that occur on the horizontal and the vertical boundary equilibrium equations. As a first step, we have to characterize a sufficiently rich basis of product-form solutions satisfying the equilibrium equations in the interior of the state space. This rich basis serves as a pool from which the compensation procedure chooses the appropriate product-form terms so that the boundary equations are also satisfied.

In this work, we establish a finite compensation procedure for a class of two-dimensional random walks defined after violating a fundamental condition of applying the compensation method \cite{adanaplprob}, i.e., we now allow from an interior point, transitions to the East and North (see also the discussion in Section \ref{intro}), and show that their invariant measure, $\boldsymbol{\pi}$, can be written as a mixture of exactly three geometric terms. Our approach is summarized in the following steps.
\begin{enumerate}
\item Compensation starts with an initial term $\gamma^{m}\delta^{n}$ that satisfies the interior and one of the boundary conditions (if it satisfies both boundary conditions, then we are done immediately as in Section \ref{special}); see Prop. \ref{prop1} when inner and horizontal boundary condition are satisfied (a symmetric one for the vertical boundary can be obtained thanks to Conditions A, B). Assume that Conditions A, B are satisfied. Insert the geometric term $\gamma^{m}\delta^{n}$ in the interior equilibrium equations to determine the set of basic solutions that satisfy the interior equations; see Lemma \ref{lemma1}.

 The form of the marginal distributions (thanks to conditions A, B) allows to choose an initial geometric term for the invariant measure of $Q$. In particular, starting with a solution that satisfies the inner and the horizontal boundary condition, the initial factor $\gamma$ of the initial geometric term is known explicitly due to the elegant form of the marginals (symmetrically for the initial factor $\delta$ when we start with a solution that satisfies the vertical boundary); see also Lemma \ref{lem1}.
\item Next, we start compensating for the error on the vertical boundary; see Lemma \ref{lemma} (symmetrically on the horizontal when we start with a solution that satisfies the vertical boundary; see Lemma \ref{hcs}); here we do not yet to require that the $\delta$'s and $\gamma$'s are less than 1 in absolute value. The updated solution does not satisfy the horizontal boundary, and then we have to add a new term; see Lemma \ref{hcs} (symmetrically on the vertical when we start with a solution that satisfies the horizontal boundary; see Lemma \ref{lemma}). So up to now, a linear combination of three geometric terms (chosen from the set of basic solutions that satisfy the interior equilibrium equations) is needed in order to satisfy the inner equations. In any compensation step we added an additional geometric term multiplied by a coefficient, and we solve the boundary equilibrium equations with this coefficient as unknown parameter.
\item The updated solution does not satisfy the vertical boundary (symmetrically on the horizontal when we start with a solution that satisfies the vertical boundary), and we apply again Lemma \ref{lemma} (symmetrically Lemma \ref{hcs}). However, the coefficient of the additional product-form term is found equal to zero, which means that no compensation is needed anymore so that the obtained solution to satisfy the inner equilibrium equations. The new updated solution satisfies the inner, the horizontal and the vertical boundary equations. A specific condition, named Condition C (discussed in the following) is responsible for terminating the compensation procedure. The derived formal solution in given in \eqref{solgf}; see also \eqref{solgf1}. At that point we have obtained a solution to the equilibrium equations \eqref{int}-\eqref{h2}, up to a multiplicative constant.
\item To show that the formal solution satisfying \eqref{int}-\eqref{h2} satisfies also the equilibrium equations \eqref{00}-\eqref{11}, we have to introduce an additional condition, named Condition D. This condition relates the transition probabilities $q_{0,1}^{(h)}$ and $q_{1,0}^{(v)}$ with those in the interior. Hence, we have a solution of interior and boundaries, with all $\delta$'s and $\gamma$'s of the compensating terms to be less than 1 in absolute value (so that we can normalize the solution). The normalization equation is used to obtain the unique solution to \eqref{int}-\eqref{11}.
\end{enumerate}

To conclude, we show in the following that under specific conditions, the solution of the equilibrium equations \eqref{int}-\eqref{11} can be written as a finite linear combination of products of the form $\gamma^{m}\delta^{n}$, by using the compensation approach. 
It consists of an initial term that satisfies the equilibrium equations \eqref{int} in the interior of the state space and at one of the two boundaries, and a finite number of compensation
terms. Each compensation term corrects the error made by the previous term at one of the two boundaries. 

The following lemma characterizes a continuum of product-forms satisfying the inner equations, i.e., a set of basic solutions satisfying the interior equilibrium equations.
\begin{lemma}\label{lemma1}
The product $\gamma^{m}\delta^{n}$ is a solution of \eqref{int} if and only if $\gamma$ and $\delta$ satisfy
\begin{equation}
\gamma\delta(1-q_{0,0})=q_{-1,1}\gamma^{2}+q_{0,1}\gamma+q_{1,0}\delta+q_{1,-1}\delta^{2}+q_{0,-1}\gamma\delta^{2}+q_{-1,0}\gamma^{2}\delta.\label{ker}
\end{equation}
\end{lemma}
\begin{proof}
The proof is straightforward by substituting the product $\gamma^{m}\delta^{n}$ in \eqref{int}.
\end{proof}

Let,
\begin{displaymath}
K(\gamma,\delta)=\gamma\delta(1-q_{0,0})-q_{-1,1}\gamma^{2}-q_{0,1}\gamma-q_{1,0}\delta-q_{1,-1}\delta^{2}-q_{0,-1}\gamma\delta^{2}-q_{-1,0}\gamma^{2}\delta.
\end{displaymath}
Moreover, substitute the product $\gamma^{m}\delta^{n}$ in \eqref{v1}-\eqref{h2}. Then, we will have the following equations:
\begin{equation}
\begin{array}{rl}
V_{1}(\gamma,\delta)=&\gamma\delta(1-q_{0,0})-q_{-1,1}\gamma^{2}-q_{0,1}\gamma-q_{1,0}^{(v)}\delta-q_{1,-1}^{(v)}\delta^{2}-q_{0,-1}\gamma\delta^{2}-q_{-1,0}\gamma^{2}\delta-q_{1,1}^{(v)}=0,\vspace{2mm}\\
V_{0}(\gamma,\delta)=&\delta(1-q_{0,0}^{(v)})-q_{-1,1}\gamma-q_{0,-1}^{(v)}\delta^{2}-q_{-1,0}\gamma\delta-q_{0,1}^{(v)}=0,
\end{array}\label{vn}
\end{equation}
\begin{equation}
\begin{array}{rl}
H_{1}(\gamma,\delta)=&\gamma\delta(1-q_{0,0})-q_{-1,1}^{(h)}\gamma^{2}-q_{0,1}^{(h)}\gamma-q_{1,0}\delta-q_{1,-1}\delta^{2}-q_{0,-1}\gamma\delta^{2}-q_{-1,0}\gamma^{2}\delta-q_{1,1}^{(h)}=0,\vspace{2mm}\\
H_{0}(\gamma,\delta)=&\gamma(1-q_{0,0}^{(h)})-q_{1,-1}\delta-q_{-1,0}^{(h)}\gamma^{2}-q_{0,-1}\gamma\delta-q_{1,0}^{(h)}=0,
\end{array}\label{hn}
\end{equation}
Set $H(\gamma,\delta):=H_{1}(\gamma,\delta)+H_{0}(\gamma,\delta)$, and $V(\gamma,\delta):=V_{1}(\gamma,\delta)+V_{0}(\gamma,\delta)$. 
%

As shown in the following, for our problem, we can construct the same formal solution (in the form of a linear combination of a finite number of product-form terms), either by starting with an initial term satisfying the inner and the horizontal boundary equations, or with an initial term satisfying the inner and the vertical boundary equations; see Remarks \ref{rem8}, \ref{remark}. The form of the marginal distributions \eqref{mar1}, \eqref{mar2} provides information regarding one of the terms in the initial product. The next lemma provides in closed form the initial terms, and shows that these initial terms are intersection points of $K(\gamma,\delta)=0$, $H(\gamma,\delta)=0$, and  $K(\gamma,\delta)=0$, $V(\gamma,\delta)=0$, respectively (see also Figure \ref{p2}); for more details see Remark \ref{rem9}.  
\begin{lemma}\label{lem1}
\begin{enumerate}
\item There exists one product-form $\gamma^{m}\delta^{n}$, $0<|\gamma|,|\delta|<1$, which satisfies $K(\gamma,\delta)=0$, $H(\gamma,\delta)=0$. The factors of this product form  are equal to:
\begin{eqnarray}
\gamma=\frac{q_{1,0}+q_{1,-1}}{q_{-1,0}+q_{-1,1}}=\rho_{1},\label{gamma}\\
\delta=\gamma\frac{q_{0,1}+q_{-1,1}\gamma}{q_{1,-1}+q_{0,-1}\gamma}:=\gamma f(\gamma).\label{delta}
\end{eqnarray}
\item There exists one product-form $\gamma^{m}\delta^{n}$, $0<|\gamma|,|\delta|<1$, which satisfies $K(\gamma,\delta)=0$, $V(\gamma,\delta)=0$. The factors of this product-form  are equal to:
\begin{eqnarray}
\delta=\frac{q_{0,1}+q_{-1,1}}{q_{0,-1}+q_{1,-1}}=\rho_{2},\label{ddelta}\\
\gamma=\delta\frac{q_{1,0}+q_{1,-1}\delta}{q_{-1,1}+q_{-1,0}\delta}:=\delta \phi(\delta).\label{ggamma}
\end{eqnarray}
\end{enumerate}
\end{lemma}
\begin{proof}
We only prove the first part. Part 2. can be proved similarly. Let $\gamma^{m}\delta^{n}$, $0<|\gamma|<1$, $0<|\delta|<1$ be a solution of $K(\gamma,\delta)=0$, $H(\gamma,\delta)=0$. Subtracting these equations, and having in mind condition B.1, we obtain,
\begin{equation}
\delta(q_{1,-1}+q_{0,-1}\gamma)=\gamma(q_{1,-1}+q_{1,0}+q_{-1,1}+q_{-1,0}+q_{0,1})-\gamma^{2}q_{-1,0}-(q_{1,-1}+q_{1,0}).\label{ded}
\end{equation}
Now, rearrange the terms in \eqref{ker} to obtain
\begin{equation}
\begin{array}{c}
\delta^{2}(q_{1,-1}+q_{0,-1}\gamma)+\delta(\gamma^{2}q_{-1,0}+q_{1,0}-\gamma(1-q_{0,0}))+\gamma(q_{0,1}+q_{-1,1}\gamma)=0.
\end{array}\label{vz}
\end{equation}
Divide \eqref{vz} with $\delta$, and substitute \eqref{ded} in the resulting equation to obtain after simple calculations \eqref{delta}. To find $\gamma$, use \eqref{delta} in \eqref{ded} to obtain after some algebra
\begin{equation}
\gamma^{2}-\gamma(1+\frac{q_{1,-1}+q_{1,0}}{q_{-1,1}+q_{-1,0}})+\frac{q_{1,-1}+q_{1,0}}{q_{-1,1}+q_{-1,0}}=0.
\label{cc1}
\end{equation}
Note that \eqref{cc1} has two roots, $\gamma=1$, which is rejected, and $\gamma=\frac{q_{1,-1}+q_{1,0}}{q_{-1,1}+q_{-1,0}}=\rho_{1}$, as given in \eqref{gamma}.
\end{proof}
\begin{figure}[ht!]
\centering
\includegraphics[scale=0.5]{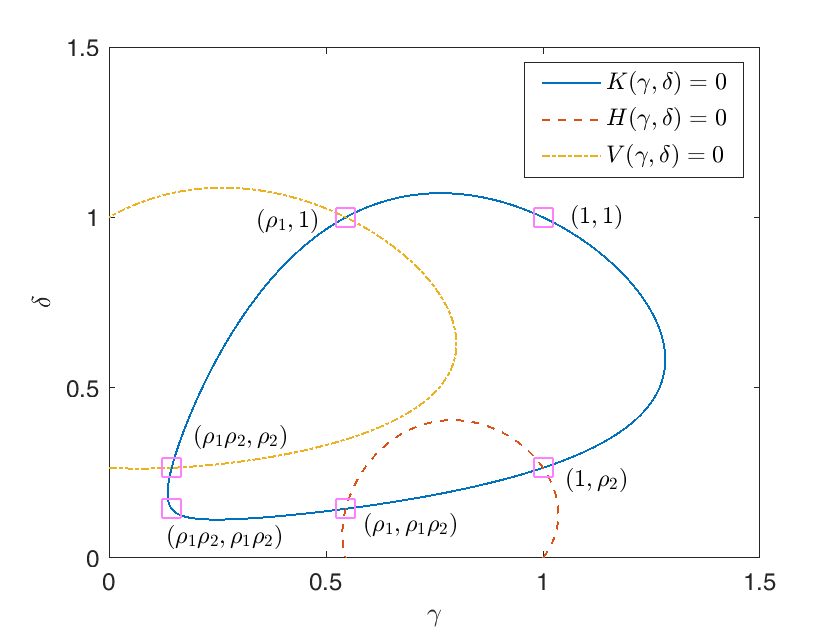}
\caption{The curves $K(\gamma,\delta)=0$, $H(\gamma,\delta)=0$, $V(\gamma,\delta)=0$ for $q_{0,1}=0.0405$, $q_{1,0}=0.027$, $q_{-1,1}=0.0495$, $q_{1,-1}=0.153$, $q_{-1,0}=0.2805$, $q_{0,-1}=0.187$, satisfying Conditions A, B, C.}\label{p2}
\end{figure}

Note that in the proof of Lemma \ref{lem1}, we make use only of condition B.1. Moreover, Lemma \ref{lemma1} characterizes a rich basis of product-form terms satisfying the equilibrium equations in the interior of the state space. This basis is used to construct a linear combination that also satisfies the horizontal and the vertical boundary equations. 

In general, this basis contains uncountably many terms, and the compensation procedure \cite{adanaplprob} chooses appropriately countably many terms so as to alternately compensate for the error on one of the two boundaries. In our work, we focus on conditions that ensure that we need finitely many product-form terms in order to construct a linear combination that also satisfies the horizontal and the vertical boundary equations. Thus, we need a stopping criterion, so that starting from an initial term, we will stop compensating after a finite number of compensation steps. 

The next condition allows to stop compensating after introducing exactly three product-form terms, and refers to a condition that is satisfied by the transition probabilities in the interior state space. \\

\noindent
\textbf{Condition C:} 
    $
    q_{1,0}q_{-1,0}=q_{0,1}q_{0,-1}=q_{-1,1}q_{1,-1}.
    $\\

When we additionally use Condition C, then, simple computations shows that for $\gamma=\rho_{1}$,
\begin{displaymath}
\frac{q_{0,1}+q_{-1,1}\gamma}{q_{1,-1}+q_{0,-1}\gamma}=\frac{q_{0,1}+q_{-1,1}}{q_{1,-1}+q_{0,-1}}=\rho_{2},
\end{displaymath}
so that 
\begin{equation}
\delta=\frac{q_{1,0}+q_{1,-1}}{q_{-1,0}+q_{-1,1}}\times \frac{q_{0,1}+q_{-1,1}}{q_{1,-1}+q_{0,-1}}=\rho_{1}\rho_{2}.\label{delta0}
\end{equation}

In particular, by assuming Condition C, we introduce a stopping criterion (see Lemma \ref{hcs} below), under which after derivation of the third product-form term (i.e., after the second compensation step), no new product-form term is found, i.e., the procedure selects either the former product-form term, or a term that produces an un-normalized solution. Equivalently, by applying another vertical compensation step, the coefficient of the additional term for $m,n>0$ equals zero. Such a situation terminates the compensation procedure and the derived solution satisfies the inner, the horizontal, and the vertical boundary equilibrium equations.

Denote by $\gamma_{0}:=\rho_{1}$ as given in \eqref{gamma}, and $\delta_{0}:=\rho_{1}\rho_{2}$, as given in \eqref{delta}. From here on we assume that Conditions A, B and C are satisfied. 
\begin{proposition} (Initial solution) \label{prop1}
For $\gamma_{0}$, $\delta_{0}$ in \eqref{gamma}, \eqref{delta0}, respectively, the solution 
\begin{equation}
x(m,n)=\left\{\begin{array}{ll}
c_{0}\gamma_{0}^{m}\delta_{0}^{n},&m,n>0,\\
e_{0}\gamma_{0}^{m},&m>0,n=0.
\end{array}\right.
\label{inig}
\end{equation}
satisfies the balance equations \eqref{int}, \eqref{h1}, \eqref{h2}, where 
\begin{equation}
e_{0}=c_{0}\frac{Q_{N}(\gamma_{0})}{H_{N}(\gamma_{0})},
  \label{e00}  
\end{equation}
with
\begin{equation}
\begin{array}{rl}
Q_{N}(\gamma)=&\gamma(q_{0,1}+q_{-1,1}\gamma),\\
H_{N}(\gamma)=&q_{0,1}^{(h)}\gamma+\gamma^{2}q_{-1,1}^{(h)}+q_{1,1}^{(h)}.
\end{array}\label{e0}
\end{equation}
\end{proposition}
\begin{proof}
Substituting \eqref{inig} in \eqref{h1}, \eqref{h2}, yields two equations that must be satisfied by $e_{0}$. In particular, substituting in \eqref{h1}, \eqref{h2} yields respectively,
\begin{displaymath}
\begin{array}{rl}
e_{0}=&c_{0}\delta_{0} w_{1}(\gamma_{0}),\\
e_{0}=&c_{0}\gamma_{0} w_{2}(\gamma_{0}),
\end{array}
\end{displaymath}
where 
\begin{equation}
\begin{array}{rl}
      w_{1}(\gamma)=&\frac{q_{1,-1}+q_{0,-1}\gamma}{\gamma(1-q_{0,0}^{(h)})-q_{-1,0}^{(h)}\gamma^{2}-q_{1,0}^{(h)}},  \\
    w_{2}(\gamma)=& \frac{q_{0,1}+q_{-1,1}\gamma}{\gamma(q_{0,1}^{(h)}+q_{-1,1}^{(h)}\gamma)+q_{1,1}^{(h)}}.
\end{array}\label{mko}
\end{equation} 
It is seen that for $\gamma_{0}$, $\delta_{0}$ given in \eqref{gamma}, \eqref{delta0}, respectively,
\begin{displaymath}
\begin{array}{rl}
\frac{\delta_{0}}{\gamma_{0}}=&f(\gamma_{0})=\frac{w_{2}(\gamma_{0})}{w_{1}(\gamma_{0})},
\end{array}
\end{displaymath}
since the denominators of $w_{1}(\gamma_{0})$, $w_{2}(\gamma_{0})$ coincide. Indeed,
\begin{displaymath}
\begin{array}{rl}
\gamma(1-q_{0,0}^{(h)})-q_{-1,0}^{(h)}\gamma^{2}-q_{1,0}^{(h)}=&\gamma(q_{0,1}^{(h)}+q_{-1,1}^{(h)}\gamma)+q_{1,1}^{(h)}\stackrel{Cond. B}{\Longleftrightarrow}\vspace{2mm}\\
\gamma(q_{0,1}^{(h)}+q_{1,0}+q_{1,-1}+q_{-1,1}+q_{-1,0})-q_{-1,0}\gamma^{2}-q_{1,-1}=&\gamma q_{0,1}^{(h)}+q_{-1,1}\gamma^{2}+q_{1,0}\stackrel{Cond. B}{\Longleftrightarrow}\vspace{2mm}\\
\gamma^{2}-\gamma(1+\frac{q_{1,0}+q_{1,-1}}{q_{-1,0}+q_{-1,1}})+\frac{q_{1,0}+q_{1,-1}}{q_{-1,0}+q_{-1,1}}=0,
\end{array}
\end{displaymath}
which is satisfied by $\gamma=\gamma_{0}=\rho_{1}$; see \eqref{cc1}. Thus, $\delta_{0} w_{1}(\gamma_{0})=\gamma_{0} w_{2}(\gamma_{0})$. After simple calculations, we realize that $\gamma_{0} w_{2}(\gamma_{0})=\frac{Q_{N}(\gamma_{0})}{H_{N}(\gamma_{0})}$, where $Q_{N}(\gamma)$, $H_{N}(\gamma)$ as given in \eqref{e0}. As a consequence, we have a unique coefficient $e_{0}$ as given in \eqref{e00}. 
\end{proof}

It is readily seen that \eqref{inig} does not satisfy the vertical boundary equations \eqref{v1}, \eqref{v2}. The idea behind the compensation approach is to add a new term $c_{1}\tilde{\gamma}^{m}\tilde{\delta}^{n}$, such that $c_{0}\gamma_{0}^{m}\delta_{0}^{n}+c_{1}\tilde{\gamma}^{m}\tilde{\delta}^{n}$ satisfies \eqref{int}-\eqref{v2}. Since it should hold for all $n\geq 2$, we must have $\tilde{\delta}=\delta_{0}$, and since we also want to satisfy \eqref{int}, we must have $\tilde{\gamma}$, to be the zero of $K(\tilde{\gamma},\delta)=0$, such that $\tilde{\gamma}\neq \gamma_{0}$ (note that we cannot choose $\gamma_{0}$, since we will arrive at the previous product-form). It is readily seen that this choice does not provide enough freedom to appropriately choose $c_{1}$, since it needs to fulfil two requirements, i.e., \eqref{v1}, \eqref{v2}.
So we need to update carefully the solution as given in the following lemma.
\begin{lemma}\label{lemma} (Vertical compensation step)
For $\delta=\delta_{0}=\rho_{1}\rho_{2}\in(0,1)$, let $\gamma_{0}$, $\gamma_{1}$ be the roots of \eqref{ker} with $\gamma_{1}=\frac{\delta_{0}}{\gamma_{0}}\phi(\delta_{0})=\delta_{0}<\gamma_{0}$ (thanks to Condition C). Then there exists coefficients $c_{1}$ and $z_{1}$ such that
\begin{equation}
  x(m,n)=\left\{\begin{array}{ll}
c_{0}\gamma_{0}^{m}\delta_{0}^{n}+c_{1}\gamma_{1}^{m}\delta_{0}^{n},&m>0,n>0,\\
z_{1}\delta_{0}^{n},&m=0,n>0.
\end{array}\right. \label{g1} 
\end{equation}
satisfies \eqref{int}, \eqref{v1}, \eqref{v2}. These coefficients are:
\begin{eqnarray}
c_{1}=-c_{0}\frac{\frac{V_{E}(\delta_{0})}{\gamma_{1}}+L(\delta_{0})}{\frac{V_{E}(\delta_{0})}{\gamma_{0}}+L(\delta_{0})},\label{coeg1}\\
z_{1}=-c_{0}\frac{Q_{E}(\delta_{0})\left(\frac{1}{\gamma_{1}}-\frac{1}{\gamma_{0}}\right)}{\frac{V_{E}(\delta_{0})}{\gamma_{1}}+L(\delta_{0})},\label{coeg2}
\end{eqnarray}
where
\begin{equation}
\begin{array}{rl}
Q_{E}(\delta_{0})=&q_{1,-1}\delta_{0}^{2}+q_{1,0}\delta_{0},\\
V_{E}(\delta_{0})=&q_{1,-1}^{(v)}\delta_{0}^{2}+q_{1,0}^{(v)}\delta_{0}+q_{1,1}^{(v)},\\
L(\delta_{0})=&q_{0,-1}^{(v)}\delta_{0}^{2}-(1-q_{0,0}^{(v)})\delta_{0}+q_{0,1}^{(v)}.
\end{array}\label{prege}
\end{equation}
\end{lemma}
\begin{proof}
Thanks to Condition C, we know that $\delta_{0}$ is given in \eqref{delta0}. Then, $K(\gamma,\delta_{0})=0$ results in a quadratic polynomial with two roots, say $\tilde{\gamma}_{0}$, $\tilde{\gamma}_{1}$, such that
\begin{displaymath}
\tilde{\gamma}_{0}\tilde{\gamma}_{1}=\delta_{0}\phi(\delta_{0}).
\end{displaymath}
It is readily seen that $\tilde{\gamma}_{0}=\gamma_{0}$, so that $\tilde{\gamma}_{1}:=\gamma_{1}=\frac{\delta_{0}}{\gamma_{0}}\phi(\delta_{0})$. Condition C ensures that $\gamma_{0}=\phi(\delta_{0})$, so that $\gamma_{1}=\delta_{0}=\rho_{1}\rho_{2}$.

Thus, insertion of the solution to \eqref{v1}, \eqref{v2} results after simple calculations to the following system of equations for the coefficients $c_{1}$, $z_{1}$:
\begin{equation}
   \begin{array}{rl}
        z_{1}-c_{1}\gamma_{1}z_{2}(\delta_{0})=&c_{0}\gamma_{0}z_{2}(\delta_{0}),  \\
      z_{1}-c_{1}\delta_{0}z_{1}(\delta_{0})=&c_{0}\delta_{0}z_{1}(\delta_{0}),
    \end{array}\label{h11}
\end{equation}
where
\begin{equation}
\begin{array}{rl}
z_{1}(\delta)=&\frac{q_{1,0}+q_{1,-1}\delta}{\delta(q_{1,0}^{(v)}+q_{1,-1}^{(v)}\delta)+q_{1,1}^{(v)}}=\delta\frac{Q_{E}(\delta)}{V_{E}(\delta)},\vspace{2mm}\\
z_{2}(\delta)=&\frac{q_{-1,1}+q_{-1,0}\delta}{\delta(1-q_{0,0}^{(v)})-q_{0,-1}^{(v)}\delta^{2}-q_{0,1}^{(v)}}.
\end{array}\label{mkop}
\end{equation}
The solution to \eqref{h11} is 
\begin{displaymath}
\begin{array}{rl}
     z_{1}=&c_{0}\frac{z_{1}(\delta_{0})z_{2}(\delta_{0})(\gamma_{1}-\gamma_{0})}{z_{2}(\delta_{0})-z_{1}(\delta_{0})},  \vspace{2mm}\\
     c_{1}=&c_{0}\frac{\delta_{0}z_{1}(\delta_{0})-\gamma_{0}z_{2}(\delta_{0})}{\gamma_{1}(z_{2}(\delta_{0})-z_{1}(\delta_{0}))}. 
\end{array}
\end{displaymath}
Having in mind Condition C (which in turn implies that $\gamma_{0}=\frac{q_{1,0}+q_{1,-1}\delta_{0}}{q_{-1,1}+q_{-1,0}\delta_{0}}$), the expressions for the coefficients $c_{1}$, $z_{1}$ are written after some algebra in the form given in \eqref{coeg1}, \eqref{coeg2}, respectively. 
\end{proof}

The insertion of the new term (i.e., $c_{1}\gamma_{1}^{m}\delta_{0}^{n}$) will violate the horizontal boundary equations \eqref{h1}, \eqref{h2}. The next lemma summarizes the horizontal compensation step.
\begin{lemma}\label{hcs} (Horizontal compensation step) For $\gamma=\gamma_{1}=\rho_{1}\rho_{2}\in(0,1)$, let $\delta_{0}$, $\delta_{1}$ the roots of \eqref{ker} such that $\delta_{1}=\frac{\gamma_{1}}{\delta_{0}}f(\gamma_{1})=\rho_{2}$. Then there exists coefficients $d_{1}$ and $e_{1}$ such that 
\begin{equation}
    x(m,n)=\left\{\begin{array}{ll}
c_{1}\gamma_{1}^{m}\delta_{0}^{n}+d_{1}\gamma_{1}^{m}\delta_{1}^{n},&m>0,n>0,\\
e_{1}\gamma_{1}^{m},&m>0,n=0,
\end{array}\right.\label{g2}
\end{equation}
satisfies \eqref{int}, \eqref{h1}, \eqref{h2}. The factors $d_{1}$, $e_{1}$ are equal to:
\begin{eqnarray}
d_{1}=-c_{1}\frac{\frac{H_{N}(\gamma_{1})}{\delta_{1}}+W(\gamma_{1})}{\frac{H_{N}(\gamma_{1})}{\delta_{0}}+W(\gamma_{1})},\label{coefg1}\\
e_{1}=-c_{1}\frac{Q_{N}(\gamma_{1})\left(\frac{1}{\delta_{1}}-\frac{1}{\delta_{0}}\right)}{\frac{H_{N}(\gamma_{1})}{\delta_{0}}+W(\gamma_{1})},\label{coefg2}
\end{eqnarray}
where $c_{1}$ as given in Lemma \ref{lemma} and
\begin{equation}
\begin{array}{rl}
Q_{N}(\gamma_{1})=&q_{-1,1}\gamma_{1}^{2}+q_{0,1}\gamma_{1},\\
H_{N}(\gamma_{1})=&q_{-1,1}^{(h)}\gamma_{1}^{2}+q_{0,1}^{(h)}\gamma_{1}+q_{1,1}^{(h)},\\
W(\gamma_{1})=&q_{-1,0}^{(h)}\gamma_{1}^{2}-(1-q_{0,0}^{(h)})\gamma_{1}+q_{1,0}^{(h)}.
\end{array}\label{prege1}
\end{equation}
\end{lemma}
\begin{proof}
It is readily seen that for $\gamma=\gamma_{1}$, $K(\gamma_{1},\delta)=0$, has two roots $\delta_{0}$, $\delta_{1}$, with
\begin{displaymath}
\begin{array}{rl}
\delta_{1}=&\frac{\gamma_{1}}{\delta_{0}}f(\gamma_{1})=\frac{\phi(\delta_{0})}{\gamma_{0}}f(\gamma_{1})\xlongequal{\text{Cond. C}}f(\gamma_{1})\xlongequal{\text{Cond. C}}f(\gamma_{0})=\frac{q_{0,1}+q_{-1,1}}{q_{1,-1}+q_{0,-1}}=\rho_{2}.
\end{array}
\end{displaymath}
Thus, by substituting the solution to \eqref{h1}, \eqref{h2}, results in the following system of equations for the coefficients $d_{1}$, $e_{1}$:
\begin{equation}
\begin{array}{rl}
e_{1}-d_{1}\delta_{1}w_{1}(\gamma_{1})=&c_{1}\delta_{0}w_{1}(\gamma_{1}),\\
-e_{1}+d_{1}\gamma_{1}w_{2}(\gamma_{1})=&-c_{1}\gamma_{1}w_{2}(\gamma_{1}),
\end{array}\label{sys2}
\end{equation}
The solution to \eqref{sys2} is given by
\begin{displaymath}
\begin{array}{rl}
    d_{1}= &c_{1}\frac{\gamma_{1}w_{2}(\gamma_{1})-\delta_{0}w_{1}(\gamma_{1})}{\delta_{1}w_{1}(\gamma_{1})-\gamma_{1}w_{2}(\gamma_{1})}, \vspace{2mm} \\
     e_{1}=&c_{1}\frac{\gamma_{1}w_{2}(\gamma_{1})w_{1}(\gamma_{1})(\delta_{1}-\delta_{0})}{\delta_{1}w_{1}(\gamma_{1})-\gamma_{1}w_{2}(\gamma_{1})}. 
\end{array}
\end{displaymath}
Having in mind Condition C (which in turn implies that $\delta_{1}=\frac{q_{0,1}+q_{-1,1}\gamma_{1}}{q_{1,-1}+q_{0,-1}\gamma_{1}}$), the expressions for the coefficients $d_{1}$, $e_{1}$ are given, after some algebra, by the form given in \eqref{coefg1}, \eqref{coefg2}, respectively.
\end{proof}

Note that after the two compensation steps, the updated solution is 
\begin{equation}
x(m,n)=\left\{\begin{array}{ll}
c_{0}\gamma_{0}^{m}\delta_{0}^{n}+c_{1}\gamma_{1}^{m}\delta_{0}^{n}+d_{1}\gamma_{1}^{m}\delta_{1}^{n},&m>0,n>0,\\
z_{1}\delta_{0}^{n},&m=0,n>0,\\
e_{0}\gamma_{0}^{m}+e_{1}\gamma_{1}^{m},&m>0,n=0,
\end{array}\right.\label{solg}
\end{equation}

Solution \eqref{solg} violates the vertical boundary equations \eqref{v1}, \eqref{v2}, so we need to compensate by adding an additional term. Applying Lemma \ref{lemma} for $\delta=\delta_{1}=\rho_{2}$ (thus, from \eqref{ker} we have $\gamma_{1}=\rho_{1}\rho_{2}$, $\gamma_{2}=1$), i.e., seeking for $d_{2}$, $z_{2}$ so that 
\begin{equation}
x(m,n)=\left\{\begin{array}{ll}
d_{1}\gamma_{1}^{m}\delta_{1}^{n}+d_{2}\gamma_{2}^{m}\delta_{1}^{n},&m>0,n>0,\\
z_{2}\delta_{1}^{n},&m=0,n>0,
\end{array}\right.\label{soltg}
\end{equation}
satisfies \eqref{v1}, \eqref{v2} (and of course \eqref{int}), we come up with two equations that $d_{2}$, $z_{2}$ should satisfy:
In particular,
\begin{equation}
    \begin{array}{rl}
         z_{2}=&(d_{1}\gamma_{1}+d_{2})\frac{q_{-1,1}+q_{-1,0}\delta_{1}}{(1-q_{0,0}^{(v)})\delta_{1}-q_{0,1}^{(v)}-q_{0,-1}^{(v)}\delta_{1}^{2}},\vspace{2mm} \\ z_{2}=&(d_{1}+d_{2})\frac{q_{1,0}\delta_{1}+q_{1,-1}\delta_{1}^{2}}{q_{1,0}^{(v)}\delta_{1}+q_{1,-1}^{(v)}\delta_{1}^{2}+q_{1,1}^{(v)}}=(d_{1}+d_{2})\frac{Q_{E}(\delta_{1})}{V_{E}(\delta_{1})}.
    \end{array}\label{kl}
\end{equation}
Note from Part 2. of Lemma \ref{lem1},  that $\gamma_{1}(q_{-1,1}+q_{-1,0}\delta_{1})=q_{1,0}\delta_{1}+q_{1,-1}\delta_{1}^{2}$. Moreover,
\begin{equation}\begin{array}{rl}
     (1-q_{0,0}^{(v)})\delta-q_{0,1}^{(v)}-q_{0,-1}^{(v)}\delta^{2}=&q_{1,0}^{(v)}\delta+q_{1,-1}^{(v)}\delta^{2}+q_{1,1}^{(v)}\stackrel{\text{Cond. B.1}}{\Longleftrightarrow}\\
     (1-q_{0,0}^{(v)}-q_{1,0}^{(v)})\delta=&(q_{0,-1}+q_{1,-1}^{(v)})\delta^{2}+q_{0,1}+q_{-1,1}\stackrel{\text{Cond. B.1}}{\Longleftrightarrow}\\
     \delta^{2}-(1+\rho_{2})\delta+\rho_{2}=&0
\end{array}\label{hio}
\end{equation}
The last equation in \eqref{hio} has two roots, namely $\delta=\delta_{1}=\rho_{2}$, and $\delta=1$. Thus, 
\begin{displaymath}
    \frac{\gamma_{1}(q_{-1,1}+q_{-1,0}\delta_{1})}{(1-q_{0,0}^{(v)})\delta_{1}-q_{0,1}^{(v)}-q_{0,-1}^{(v)}\delta_{1}^{2}}=\frac{q_{1,0}\delta_{1}+q_{1,-1}\delta_{1}^{2}}{q_{1,0}^{(v)}\delta_{1}+q_{1,-1}^{(v)}\delta_{1}^{2}+q_{1,1}^{(v)}}.
\end{displaymath}
Therefore, it is readily seen from \eqref{kl} that $d_{2}=0$, so no other compensation term is needed (i.e., the coefficient of the additional product-form term vanishes), and $z_{2}=d_{1}\frac{Q_{E}(\delta_{1})}{V_{E}(\delta_{1})}$. 

Thus, the compensation approach is terminated for $m,n>0$. In particular, by using Conditions A, B.1, and C the following expression constitutes a formal solution to equilibrium equations \eqref{int}-\eqref{h2}:
\begin{equation}
x(m,n)=\left\{\begin{array}{ll}
c_{0}\gamma_{0}^{m}\delta_{0}^{n}+c_{1}\gamma_{1}^{m}\delta_{0}^{n}+d_{1}\gamma_{1}^{m}\delta_{1}^{n},&m>0,n>0,\\
z_{1}\delta_{0}^{n}+z_{2}\delta_{1}^{n},&m=0,n>0,\\
e_{0}\gamma_{0}^{m}+e_{1}\gamma_{1}^{m},&m>0,n=0.
\end{array}\right.\label{solgf}
\end{equation}

Define now, 
\begin{displaymath}
x(m,0)=e_{0}\gamma_{0}^{m}+e_{1}\gamma_{1}^{m},\,m\geq0,n=0.
\end{displaymath}
Since $e_{0}+e_{1}=z_{1}+z_{2}$ (i.e., thanks to Condition C), the term
\begin{equation}
x(m,n)=\left\{\begin{array}{ll}
c_{0}\gamma_{0}^{m}\delta_{0}^{n}+c_{1}\gamma_{1}^{m}\delta_{0}^{n}+d_{1}\gamma_{1}^{m}\delta_{1}^{n},&m>0,n>0,\\
z_{1}\delta_{0}^{n}+z_{2}\delta_{1}^{n},&m=0,n \geq 0,\\
e_{0}\gamma_{0}^{m}+e_{1}\gamma_{1}^{m},&m\geq 0,n=0,
\end{array}\right.\label{solgf1}
\end{equation}
satisfies \eqref{int}-\eqref{h2} except \eqref{00}-\eqref{11}.
\begin{remark}\label{rem8}
    Note that the procedure we followed to obtain the formal solution \eqref{solgf1} can also be made in a reverse order. More precisely, \eqref{solgf1} was derived by obtaining first an initial solution satisfying the inner and the horizontal boundary equations (Proposition \ref{prop1}). Then, we perform a vertical (Lemma \ref{lemma}), and a horizontal (Lemma \ref{hcs}) compensation step, and finish with a final vertical compensation step that results in the formal solution \eqref{solgf1}, since the coefficient of the additional product-form term vanishes. The solution \eqref{solgf1} can be also derived by starting from an initial solution satisfying the inner and the vertical boundary equations (a symmetrical version of Proposition \ref{prop1}), then performing a horizontal (Lemma \ref{hcs}) and a vertical (Lemma \ref{lemma}) compensation step, and finish with a final horizontal compensation step, in which the coefficient of the additional product-form (for $m,n>0$) term vanishes; see also Remark \ref{remark}.  
\end{remark}

\begin{remark}\label{rem9}
    Note that the initial solution may come natural if we think as follows: For fixed $\gamma$, we seek for a solution of the form:
    \begin{displaymath}
        x(m,n)=\left\{\begin{array}{ll}
             c_{0}\gamma^{m}\delta_{0}^{n}+c_{1}\gamma^{m}\delta_{1}^{n},&m,n>0,  \\
             e_{0}\gamma^{m},&m>0,n=0, 
        \end{array}\right.
    \end{displaymath}
    satisfying \eqref{int}, \eqref{h1}, \eqref{h2}. Substituting in \eqref{h1}, \eqref{h2} we come up with the following system of equations:
    \begin{equation}
        \begin{array}{rl}
             c_{1}\gamma(q_{0,1}+q_{-1,1}\gamma)-e_{0}H_{N}(\gamma)&=-c_{0}\gamma(q_{0,1}+q_{-1,1}\gamma),  \\
            c_{1}\delta_{1}(q_{1,-1}+q_{0,-1}\gamma)+e_{0}W(\gamma)&=-c_{0}\delta_{0}(q_{1,-1}+q_{0,-1}\gamma). 
        \end{array}\label{test1}
    \end{equation}
    Asking $c_{1}=0$ (since we now seek for a solution, which does not need compensation), \eqref{test1} has a unique solution when
    \begin{displaymath}
        W(\gamma)=-\frac{\delta_{0}(q_{1,-1}+q_{0,-1}\gamma)}{\gamma(q_{0,1}+q_{-1,1}\gamma)}H_{N}(\gamma)=-\frac{H_{N}(\gamma)}{\delta_{1}},
    \end{displaymath}
    by using also $K(\gamma,\delta)=0$. Using that result and by substituting back in \eqref{test1}:
    \begin{equation}
        \begin{array}{rl}
             \frac{e_{0}}{c_{0}}H_{N}(\gamma)=&\gamma(q_{0,1}+q_{-1,1}\gamma), \vspace{2mm} \\
             \frac{e_{0}}{c_{0}}W(\gamma)=&-\delta_{0}(q_{1,-1}+q_{0,-1}\gamma)\Leftrightarrow \frac{e_{0}}{c_{0}}H_{N}(\gamma)=\delta_{0}\delta_{1}(q_{1,-1}+q_{0,-1}\gamma),
        \end{array}\label{test2}
    \end{equation}
    so that
    \begin{eqnarray}
        \delta_{0}\delta_{1}(q_{1,-1}+q_{0,-1}\gamma)=\gamma(q_{0,1}+q_{-1,1}\gamma).\label{tre}
    \end{eqnarray}
    Now $H(\gamma,\delta_{0})=0$ implies
    \begin{displaymath}
    \begin{array}{rl}
      \gamma(q_{0,1}+q_{-1,1}\gamma)- \delta_{0}(q_{1,-1}+q_{0,-1}\gamma)-(H_{N}(\gamma)+W(\gamma))=  &0\Leftrightarrow \vspace{2mm}  \\
     \gamma(q_{0,1}+q_{-1,1}\gamma)- \delta_{0}(q_{1,-1}+q_{0,-1}\gamma)-\frac{c_{0}}{e_{0}}\gamma(q_{0,1}+q_{-1,1}\gamma)+\frac{c_{0}}{e_{0}} \delta_{0}(q_{1,-1}+q_{0,-1}\gamma)=   & 0\Leftrightarrow\vspace{2mm} \\
     (1-\frac{c_{0}}{e_{0}})[\gamma(q_{0,1}+q_{-1,1}\gamma)- \delta_{0}(q_{1,-1}+q_{0,-1}\gamma)]=&0,
    \end{array}
    \end{displaymath}
    which implies (if $e_{0}\neq c_{0}$) that $\delta_{0}=\gamma\frac{q_{0,1}+q_{-1,1}\gamma}{q_{1,-1}+q_{0,-1}\gamma}$, and thus, $\delta_{1}=1$. For $\delta=\delta_{0}=\gamma\frac{q_{0,1}+q_{-1,1}\gamma}{q_{1,-1}+q_{0,-1}\gamma}$, $K(\gamma,\delta_{0})=0$ implies:
    \begin{equation}
\begin{array}{c}
\delta_{0}(q_{1,-1}+q_{0,-1}\gamma)+(\gamma^{2}q_{-1,0}+q_{1,0}-\gamma(1-q_{0,0}))+\frac{\gamma}{\delta_{0}}(q_{0,1}+q_{-1,1}\gamma)=0\Leftrightarrow\vspace{2mm}\\
\gamma(q_{0,1}+q_{-1,1}\gamma)+\gamma^{2}q_{-1,0}+q_{1,0}-\gamma(1-q_{0,0})+\delta_{1}(q_{1,-1}+\gamma q_{-1,1})=0\Leftrightarrow\vspace{2mm}\\
\gamma^{2}-\gamma(1+\frac{q_{1,0}+q_{1,-1}}{q_{-1,0}+q_{-1,1}})+\frac{q_{1,0}+q_{1,-1}}{q_{-1,0}+q_{-1,1}}=0.
\end{array}\label{vzb}
\end{equation}
Thus, $\gamma_{0}=\frac{q_{1,0}+q_{1,-1}}{q_{0,-1}+q_{1,-1}}$, and $\gamma_{1}=1$, so the initial product form contains $\gamma_{0}$, $\delta_{0}$ as given above, and it happens to  satisfy $H(\gamma_{0},\delta_{0})=0$. Note that even if $e_{0}=c_{0}$, \eqref{tre} is still valid and $H(\gamma_{0},\delta_{0})=0$.
         \end{remark}

To constitute a formal solution to all equilibrium equations, we have to show that \eqref{solgf1} satisfies also \eqref{00}-\eqref{11}. Note that in the \textit{standard} compensation approach we do not pay attention to the equilibrium equations \eqref{00}-\eqref{11} (which is possible, because of the second fundamental requirement mentioned in Section \ref{intro}). However, in our case, where we have a finite number of geometric terms, it is essential for \eqref{solgf1} to satisfy \eqref{00}-\eqref{11}. To accomplish this task, we show in Proposition \ref{condd} that along with Condition B.2, we need to introduce an additional condition, named Condition D, that relates the transition probabilities $q_{1,0}^{(v)}$, $q_{0,1}^{(h)}$ with those at the interior:\\
\textbf{Condition D:} The transition probabilities $q_{0,1}^{(h)}$, $q_{1,0}^{(v)}$ satisfy
\begin{equation*}
    \begin{array}{rl}
         q_{1,0}^{(v)}=&q_{1,0}+\frac{q_{1,-1}q_{0,1}}{q_{1,0}},\\
         q_{0,1}^{(h)}=&q_{0,1}+\frac{q_{-1,1}q_{1,0}}{q_{0,1}}.
    \end{array}
\end{equation*}

\begin{proposition}\label{condd}
    Under the Condition D, the solution in \eqref{solgf1} satisfy the equilibrium equations \eqref{00}-\eqref{11}. 
\end{proposition}
\begin{proof}
We rewrite \eqref{10} (by using also Condition B.1) as
\begin{equation}
\begin{array}{rl}
\pi_{1,0}(1-q_{0,0}^{(h)})-\pi_{2,0}q_{-1,0}^{(h)}=&\pi_{0,0}q_{1,0}^{(0)}+\pi_{1,1}q_{0,-1}+\pi_{0,1}q_{1,-1}.
\end{array}\label{sv1}
\end{equation}
Substituting $\pi_{m,0}=x(m,0)$ in the left hand side of \eqref{sv1} with $x(m,0)$ given in \eqref{solgf1} yields
\begin{equation*}
\begin{array}{rl}
x(1,0)(1-q_{0,0}^{(h)})-x(2,0)q_{-1,0}^{(h)}=&e_{0}[(1-q_{0,0}^{(h)})\gamma_{0}-q_{-1,0}^{(h)}\gamma_{0}^{2}]+e_{1}[(1-q_{0,0}^{(h)})\gamma_{1}-q_{-1,0}^{(h)}\gamma_{1}^{2}].
\end{array}\label{bv1n1}
\end{equation*}
Since $e_{0}$, $e_{1}$ are derived when we substitute \eqref{solgf1} in \eqref{h2}, it is readily seen that
\begin{displaymath}
\begin{array}{l}
e_{0}[(1-q_{0,0}^{(h)})\gamma_{0}-q_{-1,0}^{(h)}\gamma_{0}^{2}]+e_{1}[(1-q_{0,0}^{(h)})\gamma_{1}-q_{-1,0}^{(h)}\gamma_{1}^{2}]\\
=q_{1,0}^{(h)}(e_{0}+e_{1})+c_{0}(q_{1,-1}+q_{0,-1}\gamma_{0})+(c_{1}\delta_{0}+d_{1}\delta_{1})(q_{1,-1}+q_{0,-1}\gamma_{1})
\end{array}
\end{displaymath}
Similarly, substituting $\pi_{m,n}=x(m,n)$ in the right hand side of \eqref{sv1} with $x(m,n)$ as given in \eqref{solgf1} (using also Condition B.1) yields
\begin{displaymath}
\begin{array}{r}
x(0,0)q_{1,0}^{(0)}+x(1,1)q_{0,-1}+x(0,1)q_{1,-1}=q_{1,0}^{(0)}(e_{0}+e_{1})+q_{0,-1}(c_{0}\gamma_{0}\delta_{0}+c_{1}\gamma_{1}\delta_{0}+d_{1}\gamma_{1}\delta_{1})\\+q_{1,-1}(z_{1}\delta_{0}+z_{2}\delta_{1}).
\end{array}
\end{displaymath}
Having in mind that $z_{1}+z_{2}=e_{1}+e_{2}$ (thanks to Condition C), for \eqref{solgf1} to satisfy \eqref{10}, it suffices to show that
\begin{equation}
\begin{array}{r}
q_{1,-1}((c_{0}+c_{1})\delta_{0}+d_{1}\delta_{1})=z_{1}(q_{1,0}^{(0)}-q_{1,0}^{(h)}+q_{1,-1}\delta_{0})+z_{2}(q_{1,0}^{(0)}-q_{1,0}^{(h)}+q_{1,-1}\delta_{1}).
\end{array}\label{uopz}
\end{equation}

Having in mind the way $z_{1}$, $z_{2}$ are obtained by using Lemma \ref{lemma}, we claim that:
\begin{displaymath}
\begin{array}{rl}
    d_{1}q_{1,-1}\delta_{1}=  & z_{2}(q_{1,0}^{(0)}-q_{1,0}^{(h)}+q_{1,-1}\delta_{1}),\\
     q_{1,-1}(c_{0}+c_{1})\delta_{0} = &  z_{1}(q_{1,0}^{(0)}-q_{1,0}^{(h)}+q_{1,-1}\delta_{0}). 
\end{array}
\end{displaymath}
Indeed,
\begin{equation}
\begin{array}{rl}
   d_{1}q_{1,-1}\delta_{1}=  & z_{2}(q_{1,0}^{(0)}-q_{1,0}^{(h)}+q_{1,-1}\delta_{1}) \text{        (due to Lemma \ref{lemma} in the final compensation step)}\\
     =&d_{1}\frac{\delta_{1}q_{1,0}+\delta_{1}^{2}q_{1,-1}}{q_{1,0}^{(v)}\delta_{1}+q_{1,-1}^{(v)}\delta_{1}^{2}+q_{1,1}^{(v)}}(q_{1,0}^{(0)}-q_{1,0}^{(h)}+q_{1,-1}\delta_{1})\stackrel{\text{Cond. B.1}}{\Longleftrightarrow}\\
     q_{1,-1}q_{1,0}^{(v)}\delta_{1}+q_{1,-1}q_{0,1}=&(q_{1,0}+q_{1,-1}\delta_{1})(q_{1,0}^{(0)}-q_{1,0}^{(h)})+q_{1,0}q_{1,-1}\delta_{1}\stackrel{\text{Cond. B.2}}{\Longleftrightarrow}\\
     q_{1,-1}q_{1,0}^{(v)}\delta_{1}+q_{1,-1}q_{0,1}=&(q_{1,0}+q_{1,-1}\delta_{1})(q_{1,0}^{(v)}-q_{1,0})+q_{1,0}q_{1,-1}\delta_{1}\Leftrightarrow\\
     q_{1,0}^{(v)}=&q_{1,0}+\frac{q_{1,-1}q_{0,1}}{q_{1,0}}.
\end{array}\label{u1}
\end{equation}
Similarly,
\begin{equation}
\begin{array}{rl}
   q_{1,-1}(c_{0}+c_{1})\delta_{0} = &  z_{1}(q_{1,0}^{(0)}-q_{1,0}^{(h)}+q_{1,-1}\delta_{0})\text{        (due to Lemma \ref{lemma})}\\
     =&(c_{0}+c_{1})\frac{\delta_{0}q_{1,0}+\delta_{0}^{2}q_{1,-1}}{q_{1,0}^{(v)}\delta_{0}+q_{1,-1}^{(v)}\delta_{0}^{2}+q_{1,1}^{(v)}}(q_{1,0}^{(0)}-q_{1,0}^{(h)}+q_{1,-1}\delta_{0})\stackrel{\text{Cond. B.1}}{\Longleftrightarrow}\\
     q_{1,-1}q_{1,0}^{(v)}\delta_{0}+q_{1,-1}q_{0,1}=&(q_{1,0}+q_{1,-1}\delta_{0})(q_{1,0}^{(0)}-q_{1,0}^{(h)})+q_{1,0}q_{1,-1}\delta_{0}\stackrel{\text{Cond. B.2}}{\Longleftrightarrow}\\
     q_{1,-1}q_{1,0}^{(v)}\delta_{0}+q_{1,-1}q_{0,1}=&(q_{1,0}+q_{1,-1}\delta_{0})(q_{1,0}^{(v)}-q_{1,0})+q_{1,0}q_{1,-1}\delta_{0}\Leftrightarrow\\
     q_{1,0}^{(v)}=&q_{1,0}+\frac{q_{1,-1}q_{0,1}}{q_{1,0}}.
\end{array}\label{u2}
\end{equation}
Therefore, when $q_{1,0}^{(v)}=q_{1,0}+\frac{q_{1,-1}q_{0,1}}{q_{1,0}}$, the solution \eqref{solgf1} satisfies \eqref{10}. By following similar arguments, we can show that \eqref{solgf1} satisfies \eqref{01} when $q_{0,1}^{(h)}=q_{1,0}+\frac{q_{-1,1}q_{1,0}}{q_{0,1}}$, so further details are omitted.

Let us proceed with the balance equation \eqref{11}. Substituting $\pi_{m,n}=x(m,n)$ in \eqref{11}, where $x(m,n)$ as given in \eqref{solgf1}, using Condition B.2 (i.e., $q_{1,1}^{(0)}=q_{1,1}^{(h)}+q_{1,1}^{(v)}$)  and \eqref{ker}, yields after some algebra
\begin{equation}
\begin{array}{l}
c_{0}[\gamma_{0}(q_{0,1}+q_{-1,1}\gamma_{0})+\delta_{0}(q_{1,0}+q_{1,-1}\delta_{0})]+c_{1}[\gamma_{1}(q_{0,1}+q_{-1,1}\gamma_{1})+\delta_{0}(q_{1,0}+q_{1,-1}\delta_{0})]\vspace{2mm}\\
+d_{1}[\gamma_{1}(q_{0,1}+q_{-1,1}\gamma_{1})+\delta_{1}(q_{1,0}+q_{1,-1}\delta_{1})]=z_{1}(q_{1,0}^{(v)}\delta_{0}+q_{1,-1}^{(v)}\delta_{0}^{2})+z_{2}(q_{1,0}^{(v)}\delta_{1}+q_{1,-1}^{(v)}\delta_{1}^{2})\vspace{2mm}\\
+e_{0}(q_{1,1}^{(h)}+q_{1,1}^{(v)}+q_{0,1}^{(h)}\gamma_{0}+q_{-1,1}^{(h)}\gamma_{0}^{2})+e_{1}(q_{1,1}^{(h)}+q_{1,1}^{(v)}+q_{0,1}^{(h)}\gamma_{1}+q_{-1,1}^{(h)}\gamma_{1}^{2}).
\end{array}\label{vbgxen}
\end{equation}
Having in mind that $z_{1}+z_{2}=e_{0}+e_{1}$ (thanks to Condition C), \eqref{vbgxen} is finally rewritten as
\begin{equation}
\begin{array}{l}
c_{0}[\gamma_{0}(q_{0,1}+q_{-1,1}\gamma_{0})+\delta_{0}(q_{1,0}+q_{1,-1}\delta_{0})]+c_{1}[\gamma_{1}(q_{0,1}+q_{-1,1}\gamma_{1})+\delta_{0}(q_{1,0}+q_{1,-1}\delta_{0})]\vspace{2mm}\\
+d_{1}[\gamma_{1}(q_{0,1}+q_{-1,1}\gamma_{1})+\delta_{1}(q_{1,0}+q_{1,-1}\delta_{1})]=z_{1}(q_{1,1}^{(v)}+q_{1,0}^{(v)}\delta_{0}+q_{1,-1}^{(v)}\delta_{0}^{2})\vspace{2mm}\\+z_{2}(q_{1,1}^{(v)}+q_{1,0}^{(v)}\delta_{1}+q_{1,-1}^{(v)}\delta_{1}^{2})
+e_{0}(q_{1,1}^{(h)}+q_{0,1}^{(h)}\gamma_{0}+q_{-1,1}^{(h)}\gamma_{0}^{2})+e_{1}(q_{1,1}^{(h)}+q_{0,1}^{(h)}\gamma_{1}+q_{-1,1}^{(h)}\gamma_{1}^{2}).
\end{array}\label{vbgxen1}
\end{equation}
Now note that from the derivation of $e_{0}$, $e_{1}$, \eqref{h11}, and Lemma \ref{lemma} (in the final compensation step) we have, respectively 
\begin{eqnarray}
e_{0}(q_{1,1}^{(h)}+q_{0,1}^{(h)}\gamma_{0}+q_{-1,1}^{(h)}\gamma_{0}^{2})=c_{0}\gamma_{0}(q_{0,1}+q_{-1,1}\gamma_{0}),\label{a1en}\\
e_{1}(q_{1,1}^{(h)}+q_{0,1}^{(h)}\gamma_{1}+q_{-1,1}^{(h)}\gamma_{1}^{2})=\gamma_{1}(c_{1}+d_{1})(q_{0,1}+q_{-1,1}\gamma_{1}),\label{a2en}\\
z_{1}(q_{1,1}^{(v)}+q_{1,0}^{(v)}\delta_{0}+q_{1,-1}^{(v)}\delta_{0}^{2})=(c_{0}+c_{1})\delta_{0}(q_{1,0}+q_{1,-1}\delta_{0}),\label{a3en}\\
z_{2}(q_{1,1}^{(v)}+q_{1,0}^{(v)}\delta_{1}+q_{1,-1}^{(v)}\delta_{1}^{2})=d_{1}\delta_{1}(q_{1,0}+q_{1,-1}\delta_{1}).
\label{a4en}
\end{eqnarray}
Combining \eqref{a1en}-\eqref{a4en}, we realize that \eqref{vbgxen1} holds, so that \eqref{solgf1} satisfies also \eqref{11}. The final equilibrium equation \eqref{00} is also satisfied due to
the dependence of the equilibrium equations. 
\end{proof}

Hence we can now conclude that using Conditions A, B (i.e., Conditions B.1, B.2), C, D, $\{x(m,n);(m,n)\in S\}$ as given in \eqref{solgf1} is a solution to all equilibrium equations. The next theorem summarizes our main result.
\begin{theorem}\label{mainr} (Main result)
Consider a stable two-dimensional nearest neighbour random walk satisfying Conditions A, B, C, D. Then its invariant measure is given by,
\begin{equation}
x_{m,n}=\left\{\begin{array}{ll}
c_{0}\rho_{1}^{m}(\rho_{1}\rho_{2})^{n}+c_{1}(\rho_{1}\rho_{2})^{m+n}+d_{1}(\rho_{1}\rho_{2})^{m}\rho_{2}^{n},&m>0,n>0,\\
z_{1}(\rho_{1}\rho_{2})^{n}+z_{2}\rho_{2}^{n},&m=0,n \geq 0,\\
e_{0}\rho_{1}^{m}+e_{1}(\rho_{1}\rho_{2})^{m},&m\geq 0,n=0,
\end{array}\right.\label{solfinal}
\end{equation}
where the coefficients $c_{1}$, $d_{1}$, $z_{1}$, $z_{2}$, $e_{0}$, $e_{1}$ are as given in Proposition \ref{prop1}, Lemmas \ref{lemma}, \ref{hcs}, in terms of $c_{0}$, while  $c_{0}$ is obtained by the normalization equation.
\end{theorem}
\begin{remark}\label{r10}
Remind that condition C is crucial in completing the compensation method with exactly three geometric terms. It further provides insights into the properties that the transition probabilities in the interior of the state space should satisfy:
\begin{enumerate}
\item 
\begin{displaymath}
\begin{array}{rl}
     q_{0,1}q_{0,-1}=&q_{-1,1}q_{1,-1} \Leftrightarrow \\
    q_{0,1}q_{0,-1}+ q_{0,1}q_{1,-1}=&q_{0,1}q_{1,-1}+q_{-1,1}q_{1,-1}\Leftrightarrow\\
    q_{0,1}(q_{0,-1}+q_{1,-1})=&q_{1,-1}(q_{0,1}+q_{-1,1})\Leftrightarrow\\
    \frac{q_{0,1}}{q_{1,-1}}=&\frac{q_{0,1}+q_{-1,1}}{q_{0,-1}+q_{1,-1}}=\rho_{2}=\frac{q_{-1,1}}{q_{0,-1}}.
\end{array}
\end{displaymath}
\item 
\begin{displaymath}
\begin{array}{rl}
     q_{1,0}q_{-1,0}=&q_{-1,1}q_{1,-1} \Leftrightarrow \\
    q_{1,0}q_{-1,0}+ q_{1,0}q_{-1,1}=&q_{1,0}q_{-1,1}+q_{-1,1}q_{1,-1}\Leftrightarrow\\
    q_{1,0}(q_{-1,0}+q_{-1,1})=&q_{-1,1}(q_{1,0}+q_{1,-1})\Leftrightarrow\\
    \frac{q_{1,0}}{q_{-1,1}}=&\frac{q_{1,0}+q_{1,-1}}{q_{-1,0}+q_{-1,1}}=\rho_{1}=\frac{q_{1,-1}}{q_{-1,0}}.
\end{array}
\end{displaymath}
\item $q_{1,0}q_{-1,0}=q_{0,1}q_{0,-1}\Leftrightarrow  \frac{q_{0,1}}{q_{-1,0}}=\frac{q_{1,0}}{q_{0,-1}}$.
Note that
\begin{displaymath}
\begin{array}{rl}
     \frac{q_{0,1}}{q_{-1,0}}=&\frac{q_{0,1}}{q_{1,-1}}\frac{q_{1,-1}}{q_{-1,0}}=\rho_{1}\rho_{2}=\frac{q_{1,0}}{q_{0,-1}}=\frac{q_{1,0}}{q_{-1,1}}\frac{q_{-1,1}}{q_{0,-1}}.
\end{array}
\end{displaymath}
Remind that by definition, $\rho_{1}\rho_{2}=\left(\frac{q_{1,0}+q_{1,-1}}{q_{-1,0}+q_{-1,1}}\right)\times\left(\frac{q_{0,1}+q_{-1,1}}{q_{0,-1}+q_{1,-1}}\right)$.
\end{enumerate}
\end{remark}

\begin{remark}\label{remark}
Note that, contrary to the general case \cite{adanaplprob}, Conditions A, B, C, D allow to start the compensation approach with a product-form satisfying either $K(\gamma,\delta)=0$, $H(\gamma,\delta)=0$, or $K(\gamma,\delta)=0$, $V(\gamma,\delta)=0$. Both procedures result in the same formal solution. It is seen that Conditions A, B, C, D preserve symmetry in the course of deriving the basic product-form terms, as shown in Figure \ref{p2}. 
In particular, starting with an initial solution satisfying the interior and the horizontal equilibrium equations, the compensation approach chooses the following terms in the following order: $(\rho_{1},\rho_{1}\rho_{2})\rightarrow(\rho_{1}\rho_{2},\rho_{1}\rho_{2})\rightarrow(\rho_{1}\rho_{2},\rho_{2})$. Starting with an initial solution satisfying the interior and the vertical equilibrium equations, the compensation approach chooses the same terms, but with the opposite direction/order, i.e., $(\rho_{1}\rho_{2},\rho_{2})\rightarrow(\rho_{1}\rho_{2},\rho_{1}\rho_{2})\rightarrow(\rho_{1},\rho_{1}\rho_{2})$. Clearly, the resulting solution is the same, independent of which boundary equations (i.e., either vertical or horizontal) the initial solution should  satisfy. A similar behaviour is also observed in \cite{diek}, where the invariant measure of a reflected Brownian motion with specific constraints on the boundary transition probabilities, is written as a finite sum of exponential terms.
\end{remark}

\subsection{The case where $q_{1,-1}=q_{-1,1}=0$: A single product-form solution}\label{special}
In the following, we consider the random walk discussed above, assuming now $q_{1,-1}=q_{-1,1}=0$. The transition diagram of the resulting random walk is given in Figure \ref{relax1}.
\begin{figure}[ht!]
\centering
\includegraphics[scale=1]{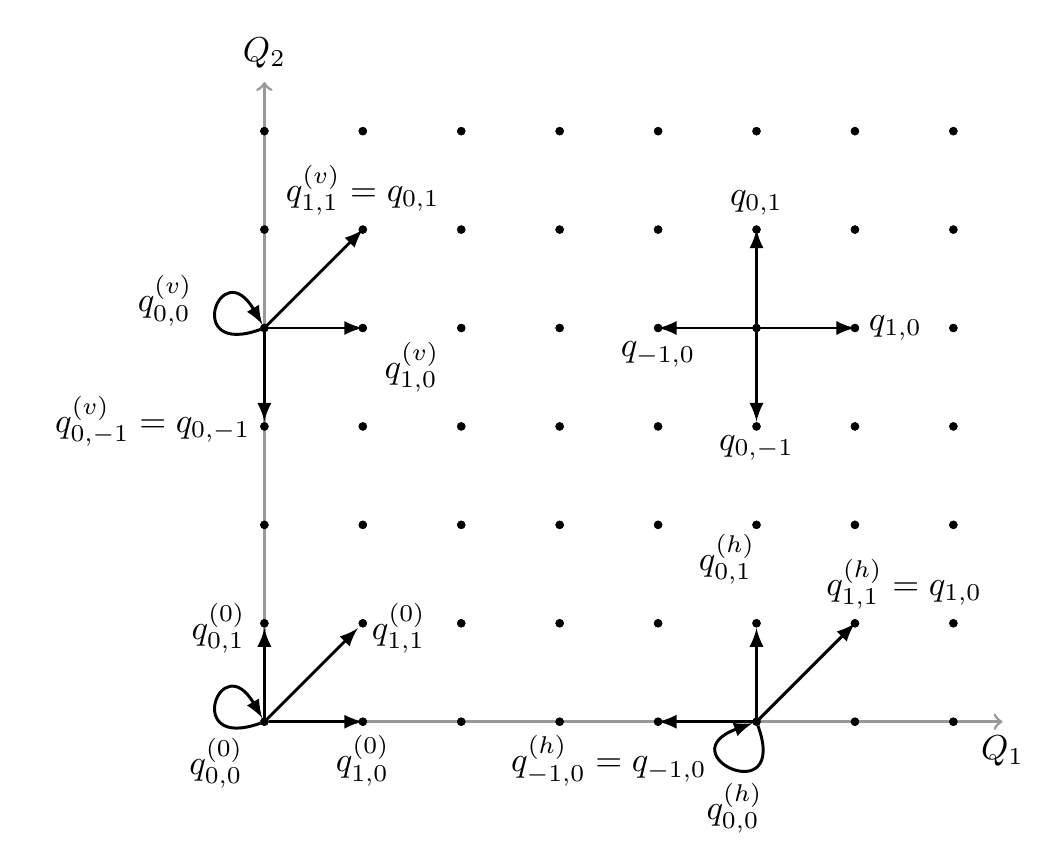}
\caption{The transition diagram.}\label{relax1}
\end{figure}

For such a case, we show that the formal solution (i.e., the solution that satisfies the interior and the boundary equilibrium equations, and obtained through the compensation approach) is of a single product-form term and in order to satisfy the equilibrium equations at points $(0,1)$, $(1,0)$, $(1,1)$, $(0,0)$, we must set $q_{0,1}^{(0)}=q_{1,0}^{(0)}=0$. Thanks to Condition B.2, this result implies that $q_{0,1}^{(h)}=q_{0,1}$ and $q_{1,0}^{(v)}=q_{1,0}$ (i.e., the revised Condition D).

Thus, by assuming, $q_{1,-1}=q_{-1,1}=0$, let Conditions A, B be satisfied. Condition C is partly satisfied in the sense that we have $q_{0,1}q_{0,-1}=q_{1,0}q_{-1,0}$. A balance principle allows to show that the marginal distributions are geometric as in \eqref{mar1}, \eqref{mar2}, where now 
\begin{displaymath}
\rho_{1}=\frac{q_{1,0}}{q_{-1,0}}<1,\text{ and }\rho_{2}=\frac{q_{0,1}}{q_{0,-1}}<1.
\end{displaymath}

We show that, by not allowing the interior transitions to the North-West, and South-East, a single product-form term is sufficient to satisfy the equilibrium equations \eqref{int}-\eqref{h2}. Following the lines of the previous subsection, the initial solution satisfying \eqref{int} and \eqref{h1}, \eqref{h2} is
\begin{equation}
    x(m,n)=\left\{\begin{array}{ll}
     c_{0}\gamma_{0}^{m}\delta_{0}^{n},&m,n\geq 1,  \\
     e_{0}\gamma_{0}^{m},&m\geq 1,n=0, \end{array}\right.\label{o1}
\end{equation}
where now Lemma \ref{lem1} implies $\gamma_{0}:=\rho_{1}$, $\delta_{0}:=\rho_{2}$. Coefficient $e_{0}$ is such that \eqref{h1}, \eqref{h2} are satisfied by $x(m,n)$, i.e., upon substitution in \eqref{h1}, \eqref{h2}, we have:
\begin{displaymath}
\begin{array}{rl}
    e_{0}= &c_{0}\frac{q_{0,-1}\delta_{0}}{1-q_{0,0}^{(h)}-q_{-1,0}^{(h)}\gamma_{0}}\xlongequal{\text{Cond. B.1 \& substituting }\gamma_{0},\ \delta_{0}}c_{0}\frac{q_{0,1}}{q_{-1,0}+q^{(h)}_{0,1}},  \\
     e_{0}=&c_{0}\frac{q_{0,1}\gamma_{0}}{q_{1,1}^{(h)}+q_{0,1}^{(h)}\gamma_{0}}\xlongequal{\text{Cond. B.1 \& substituting }\gamma_{0}} c_{0}\frac{q_{0,1}}{q_{-1,0}+q^{(h)}_{0,1}}.
\end{array}
\end{displaymath}
Following the lines in subsection \ref{compe} one should except that $x(m,n)$ needs compensation so that to satisfy \eqref{v1}, \eqref{v2}. However, by applying Lemma \ref{lemma} (note that for $\delta=\delta_{0}=\rho_{2}$, \eqref{ker} becomes a quadratic polynomial with respect to $\gamma$, having two zeros, namely $\gamma_{0}=\rho_{1}$, and $\gamma_{1}=1$), we seek for $c_{1}$, $z_{1}$ so that 
\begin{equation}
  x(m,n)=\left\{\begin{array}{ll}
c_{0}\gamma_{0}^{m}\delta_{0}^{n}+c_{1}\gamma_{1}^{m}\delta_{0}^{n},&m>0,n>0,\\
z_{1}\delta_{0}^{n},&m=0,n>0,
\end{array}\right. \label{gs1} 
\end{equation}
satisfy \eqref{v1}, \eqref{v2} (and of course \eqref{int}), or equivalently $c_{1}$, $z_{1}$ satisfy:
\begin{displaymath}
\begin{array}{rl}
    z_{1}= &(c_{0}\gamma_{0}+c_{1})\frac{q_{-1,0}}{1-q_{0,0}^{(v)}-q_{0,-1}^{(v)}\delta_{0}}\xlongequal{\text{Cond. B.1 \& substituting }\delta_{0}}(c_{0}\gamma_{0}+c_{1})\frac{q_{-1,0}}{q_{0,-1}+q^{(v)}_{1,0}},  \\
     z_{1}=&(c_{0}+c_{1})\frac{q_{1,0}\delta_{0}}{q_{1,1}^{(v)}+q_{1,0}^{(v)}\delta_{0}}\xlongequal{\text{Cond. B.1 \& substituting }\delta_{0}} (c_{0}+c_{1})\frac{q_{1,0}}{q_{0,-1}+q^{(v)}_{1,0}}.
\end{array}
\end{displaymath}
From the above, and the definition of $\gamma_{0}$, it is readily seen that $c_{1}=0$, so no compensation term is needed (i.e., the coefficient
of the additional product-form term vanishes) and $z_{1}=c_{0}\frac{q_{1,0}}{q_{0,-1}+q^{(v)}_{1,0}}$. Therefore, the formal solution to \eqref{int}-\eqref{h2} is
\begin{equation}
x(m,n)=\left\{\begin{array}{ll}
     c_{0}\gamma_{0}^{m}\delta_{0}^{n},&m,n\geq 1,  \\
     e_{0}\gamma_{0}^{m},&m\geq 1,n=0, \\
     z_{1}\delta_{0}^{n},&m=0,n\geq 1,
\end{array}\right.\label{o2}
\end{equation}
Define $x(m,0)=e_{0}\gamma_{0}^{m}$, $m\geq 0$, $n=0$, so the formal solution becomes
\begin{equation}
x(m,n)=\left\{\begin{array}{ll}
     c_{0}\gamma_{0}^{m}\delta_{0}^{n},&m,n\geq 1,  \\
     e_{0}\gamma_{0}^{m},&m\geq 0,n=0, \\
     z_{1}\delta_{0}^{n},&m=0,n\geq 1,
\end{array}\right.\label{o2m}
\end{equation}
We show that in order \eqref{o2m} to satisfy the remaining equilibrium equations \eqref{00}-\eqref{11}, we must set $q_{1,0}^{(0)}=q_{0,1}^{(0)}=0$, and thanks to Condition B.2, this result yields $q_{0,1}^{(h)}=q_{0,1}$, $q_{1,0}^{(v)}=q_{1,0}$, which further implies that $e_{0}=z_{1}$.

Substituting \eqref{o2m} in \eqref{01} (i.e., setting $\pi_{m,n}=x(m,n)$, where $x(m,n)$ as given in \eqref{o2m}), we have:
\begin{displaymath}
\begin{array}{rl}
    x(0,1)(1-q_{0,0}^{(v)})-x(0,2)q_{0,-1}^{(v)} =&x(0,0)q_{0,1}^{(0)}+x(1,1)q_{-1,0}\Longleftrightarrow  \\
    \delta_{0}[z_{1}(1-q_{0,0}^{(v)}-q_{0,-1}^{(v)}\delta_{0}) -c_{0}\gamma_{0}q_{-1,0}]=&e_{0}q_{0,1}^{(0)}. 
\end{array}
\end{displaymath}
Having in mind how $z_{1}$ is derived, we must set $q_{0,1}^{(0)}=0$, and thus, $q_{0,1}^{(h)}=q_{0,1}$, due to condition B.2.

Similarly, substituting \eqref{o2m} in \eqref{10} we obtain
\begin{displaymath}
\begin{array}{rl}
    x(1,0)(1-q_{0,0}^{(h)})-x(2,0)q_{-1,0}^{(h)} =&x(0,0)q_{1,0}^{(0)}+x(1,1)q_{0,-1}\Longleftrightarrow  \\
    \gamma_{0}[e_{0}(1-q_{0,0}^{(h)}-q_{-1,0}^{(h)}\gamma_{0})-c_{0}\delta_{0}q_{0,-1}] =&e_{0}q_{1,0}^{(0)}. 
\end{array}
\end{displaymath}
Having in mind how $e_{0}$ is derived, in order \eqref{o2m} to satisfy \eqref{10}, we must set $q_{1,0}^{(0)}=0$, and thus, $q_{1,0}^{(v)}=q_{1,0}$, due to condition B.2. It is easily realized that since $q_{1,0}^{(v)}=q_{1,0}$, $q_{0,1}^{(h)}=q_{0,1}$, Condition C also implies that $e_{0}=z_{1}$. Thus, the following expression is a solution to \eqref{int}-\eqref{h2} (and of \eqref{01}, \eqref{10}):
\begin{equation}
    x(m,n)=\left\{\begin{array}{ll}
     c_{0}\gamma_{0}^{m}\delta_{0}^{n},&m,n\geq 1,  \\
     e_{0}\gamma_{0}^{m},&m\geq 0,n=0, \\
     z_{1}\delta_{0}^{n},&m=0,n\geq 0.
\end{array}\right.\label{soll}
\end{equation}
The corresponding result is also illustrated in Figure \ref{pon10}.
\begin{figure}[ht!]
\centering
\includegraphics[scale=0.5]{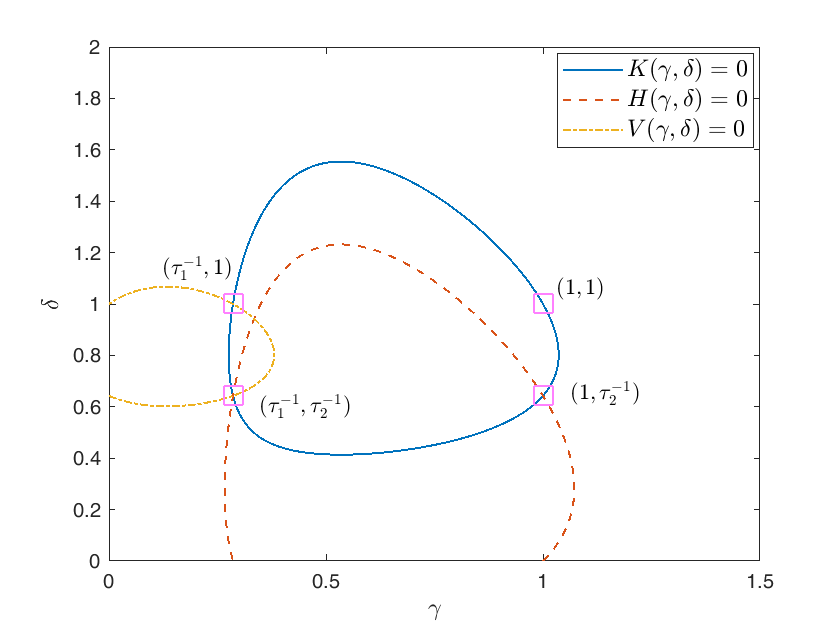}
\caption{The curves $K(\gamma,\delta)=0$, $H(\gamma,\delta)=0$, $V(\gamma,\delta)=0$, where $\tau_{1}^{-1}:=\rho_{1}=6/21$, $\tau_{2}^{-1}:=9/14$.}\label{pon10}
\end{figure}

\begin{remark}
Note that, for \eqref{soll} to satisfy \eqref{10}, \eqref{01}, we have shown that we have to set $q_{0,1}^{(h)}=q_{0,1}$, $q_{1,0}^{(v)}=q_{1,0}$. Clearly, these assumptions are direct consequences of Conditions B, and D, discussed in the general model, when we assume $q_{1,-1}=q_{-1,1}=0$, and $q_{0,1}^{(0)}=q_{1,0}^{(0)}=0$. 
\end{remark}

We now focus on \eqref{11}, having in mind that $q_{0,1}^{(h)}=q_{0,1}$, $q_{1,0}^{(v)}=q_{1,0}$, and $q_{0,1}^{(0)}=0$, $q_{1,0}^{(0)}=0$. Substituting \eqref{soll} in \eqref{11} we obtain
\begin{displaymath}
\begin{array}{rl}
     x(1,1)=&x(0,0)q_{1,1}^{(0)}+x(0,1)q_{1,0}^{(v)}+x(1,2)q_{0,-1}+x(2,1)q_{-1,0}+x(1,0)q_{0,1}^{(h)}\Longleftrightarrow  \\
     c_{0}\gamma_{0}\delta_{0}=&e_{0}q_{1,1}^{(0)}+q_{1,0}^{(v)}z_{1}\delta_{0}+q_{0,1}^{(h)}\gamma_{0}e_{0}+c_{0}(q_{0,-1}\gamma_{0}\delta_{0}^{2}+q_{-1,0}\gamma_{0}^{2}\delta_{0})\stackrel{Cond.\ D}{\Longleftrightarrow}\\
c_{0}[\gamma_{0}\delta_{0}-q_{0,-1}\gamma_{0}\delta_{0}^{2}-q_{-1,0}\gamma_{0}^{2}\delta_{0}]=&e_{0}(q_{1,1}^{(h)}+q_{1,1}^{(v)})+q_{0,1}^{(h)}\gamma_{0}e_{0}+q_{1,0}^{(v)}\delta_{0}z_{1}\stackrel{K(\gamma_{0},\delta_{0})=0}{\Longleftrightarrow}\\
c_{0}(q_{0,1}\gamma_{0}+q_{1,0}\delta_{0})=&e_{0}(q_{1,1}^{(h)}+q_{0,1}^{(h)}\gamma_{0})+z_{1}(q_{1,1}^{(v)}+q_{1,0}^{(v)}\delta_{0}).
\end{array}
\end{displaymath}
The last equation holds, since the derivation of $e_{0}$, $z_{1}$, implies that
\begin{displaymath}
\begin{array}{rl}
     c_{0}q_{0,1}\gamma_{0}=& e_{0}(q_{1,1}^{(h)}+q_{0,1}^{(h)}\gamma_{0}), \\
     c_{0}q_{1,0}\delta_{0}=&z_{1}(q_{1,1}^{(v)}+q_{1,0}^{(v)}\delta_{0}). 
\end{array}
\end{displaymath}
Finally, equation \eqref{00} is also satisfied due to the dependence of the equilibrium equations, and thus we can conclude that \eqref{soll} is a  solution to all equilibrium equations.

Therefore, when we set $q_{1,-1}=q_{-1,1}=0$, considering condition A, and the new versions of conditions B, C and D, to be satisfied, the solution to the interior and the boundary equations, which is derived through the finite compensation procedure (which now requires only a single product-form term), is a solution to all equilibrium equations. The following theorem summarizes our main result in this subsection:
\begin{theorem}\label{mainr1} (Main result)
Consider a stable two-dimensional nearest neighbour random walk satisfying Conditions A, B, C, D, as they are evolved, by assuming $q_{1,-1}=q_{-1,1}=0$. Then, its equilibrium distribution is given by,
\begin{equation}
\pi_{m,n}=c_{0}\times \left\{\begin{array}{ll}
\rho_{1}^{m}\rho_{2}^{n},&m>0,n>0,\\
\frac{q_{0,1}}{q_{-1,0}+q_{0,1}}\rho_{2}^{n},&m=0,n \geq 0,\\
\frac{q_{0,1}}{q_{-1,0}+q_{0,1}}\rho_{1}^{m},&m\geq 0,n=0,
\end{array}\right.\label{solfinalsp}
\end{equation}
where, thanks to the normalization equation, we have
\begin{displaymath}
\begin{array}{rl}
     c_{0}=&\frac{(1-\rho_{1})(1-\rho_{2})(q_{0,1}+q_{-1,0})}{q_{0,1}+q_{-1,0}\rho_{1}\rho_{2}}. 
\end{array}
\end{displaymath}
\end{theorem}

\section{Three motivating queueing examples
}\label{motivation}
In this section we present two queueing models that obey the theoretical framework in Section \ref{general}. Both models were analyzed in \cite{devosphd}, \cite{devos}, using the generating function approach and complex analytic arguments; see subsections \ref{prel}, \ref{prod}. In subsection \ref{geoprod}, we cope with another queueing model described by a non-nearest neighbor two dimensional random walk considered in \cite[subsection 2.5]{devosphd}. For this model the author provided a product-form solution by using complex analytic arguments.    
Although the authors provided an elegant mathematical approach, their method does not reveal the qualitative characteristics that are responsible for such an elegant solution. On the contrary, our approach is purely probabilistic and reveals these characteristics, resulting in the characterization formulated in Theorems \ref{mainr}, and \ref{mainr1} in Section \ref{general}. For the model in subsection \ref{geoprod}, although it does not entirely fit on the theoretical framework of Section 2, we provided the product-form solution by solving directly the equilibrium equations in the spirit of compensation approach; see also \cite{viss1}. 

\subsection{A discrete time two-class randomly alternating service model with independent Bernoulli arrivals}\label{prel}
In the following we treat the model analyzed in \cite{devos}. To make this section self-contained, we briefly describe the model in \cite{devos}. Consider a discrete-time single server  queueing model with two infinite capacity queues. The time axis is divided
into fixed-length intervals, referred to as (time) slots. New customers
may enter the system at any given (continuous) point on the time axis, but services are
synchronized to (i.e., can only start and end at) slot boundaries. The service time of a customer is exactly one slot. At the beginning of each time slot, the single server randomly selects either
queue to serve. This selection occurs independently of the system state, i.e., we assume a non-work conserving policy, and the server chooses (without knowing its state) queue I (resp. II) with probability $a$ (resp. $\bar{a}:=1-a$); $0<a<1$. Thus, the server may choose an empty queue, and then, in this slot, no service will be provided. It is further assumed that the allocation of the server to a queue 
is independent from slot to slot. The two input streams of customers into the queueing system are described by means of two independent Bernoulli processes, i.e., the number of class $j$ arrivals (i.e., the arrivals to queue $j$) during the consecutive slots is Bernoulli-distributed with parameter $\lambda_{k}$, $k = 1, 2$. For convenience, let $\bar{\lambda}_{k}=1-\lambda_{k}$, $k=1,2$. Note that the non-work conserving property is responsible for the special transition structure of the corresponding two dimensional random walk, which belongs to the general class described in Section \ref{general}.

Denote by $Q_{k}(n)$, $k=1,2,$ the system contents of class $k$ jobs at the beginning of the $n$th slot. Then, $Q(n)=(Q_{1}(n),Q_{2}(n))$ is a DTMC with state space $S=\{(i,j);i,j\geq0\}$. Denote by $q_{m,n}$ the one-step transition probabilities from state $(i,j)$ to $(i+m,j+n)$, where $(i,j)\in S$, $m,n,=-1,0,1$, where
\begin{enumerate}
\item $m,n>0$,
\begin{displaymath}
q_{0,1}=a\lambda_{1}\lambda_{2},\,q_{-1,1}=a\bar{\lambda}_{1}\lambda_{2},\,q_{0,0}=a\lambda_{1}\bar{\lambda}_{2}+\bar{a}\bar{\lambda}_{1}\lambda_{2},\,q_{1,0}=\bar{a}\lambda_{2}\lambda_{1},\,q_{1,-1}=\bar{a}\lambda_{1}\bar{\lambda}_{2},\,q_{-1,0}=a\bar{\lambda}_{1}\bar{\lambda}_{2},\,\,q_{0,-1}=\bar{a}\bar{\lambda}_{1}\bar{\lambda}_{2}.
\end{displaymath}
\item $m>0,n=0$,
\begin{displaymath}
q_{0,1}^{(h)}=a\lambda_{1}\lambda_{2}+\bar{a}\bar{\lambda}_{1}\lambda_{2},\,q_{-1,1}^{(h)}=a\bar{\lambda}_{1}\lambda_{2},\,q_{0,0}^{(h)}=a\lambda_{1}\bar{\lambda}_{2}+\bar{a}\bar{\lambda}_{1}\bar{\lambda}_{2},\,q_{1,0}^{(h)}=\bar{a}\bar{\lambda}_{2}\lambda_{1},\,q_{1,1}^{(h)}=\bar{a}\lambda_{1}\lambda_{2},\,q_{-1,0}^{(h)}=a\bar{\lambda}_{1}\bar{\lambda}_{2}.
\end{displaymath}
\item $m=0,n>0$,
\begin{displaymath}
q_{0,1}^{(v)}=a\bar{\lambda}_{1}\lambda_{2},\,q_{1,-1}^{(v)}=\bar{a}\lambda_{1}\bar{\lambda}_{2},\,q_{0,0}^{(v)}=a\bar{\lambda}_{1}\bar{\lambda}_{2}+\bar{a}\bar{\lambda}_{1}\lambda_{2},\,q_{1,0}^{(v)}=\bar{a}\lambda_{1}\lambda_{2}+a\lambda_{1}\bar{\lambda}_{2},\,q_{1,1}^{(v)}=a\lambda_{1}\lambda_{2},\,q_{0,-1}^{(v)}=\bar{a}\bar{\lambda}_{1}\bar{\lambda}_{2}.
\end{displaymath}
\item $m=n=0$,
\begin{displaymath}
\begin{array}{c}
q_{0,0}^{(0)}=\bar{a}\bar{\lambda}_{1}\bar{\lambda}_{2}+a\bar{\lambda}_{1}\bar{\lambda}_{2}=\bar{\lambda}_{1}\bar{\lambda}_{2},\,q_{1,0}^{(0)}=\bar{a}\lambda_{1}\bar{\lambda}_{2}+a\lambda_{1}\bar{\lambda}_{2}=\lambda_{1}\bar{\lambda}_{2},\\
q_{0,1}^{(0)}=\bar{a}\lambda_{2}\bar{\lambda}_{1}+a\lambda_{2}\bar{\lambda}_{1}=\lambda_{2}\bar{\lambda}_{1},\,q_{1,1}^{(0)}=\bar{a}\lambda_{2}\lambda_{1}+a\lambda_{2}\lambda_{1}=\lambda_{2}\lambda_{1}.
\end{array}
\end{displaymath}
\end{enumerate}

The transition diagram of the corresponding two-dimensional random walk is  given in Figure \ref{p1}. To be consistent with the notation in \cite{devos}, note that 
\begin{displaymath}
    \tau_{1}^{-1}:=\rho_{1}=\frac{\bar{a}\lambda_{1}}{a\bar{\lambda}_{1}},\,\,\tau_{2}^{-1}:=\rho_{2}=\frac{a\lambda_{2}}{\bar{a}\bar{\lambda}_{2}},\,\,\tau_{T}^{-1}:=\rho_{1}\rho_{2}=\frac{\lambda_{1}\lambda_{2}}{\bar{\lambda}_{1}\bar{\lambda}_{2}}.
\end{displaymath}
Theorem \ref{stabi} is now reduced to the following lemma:
\begin{lemma}
The system is stable if and only if $\lambda_{1}<a$, $\lambda_{2}<\bar{a}$.
\end{lemma}
\begin{proof}
Following Theorem \ref{stabi}, $\rho_{1}<1$, $\rho_{2}<1$ are necessary and sufficient conditions for the ergodicity of our model. In particular, $\rho_{1}<1$ (resp. $\rho_{2}<1$) is equivalent to $\bar{a}\lambda_{1}<a\bar{\lambda}_{1}$ (resp. $a\lambda_{2}<\bar{a}\bar{\lambda}_{2}$), i.e., $\lambda_{1}<a$ (resp. $\lambda_{2}<\bar{a}$). 
\end{proof}

%
Note that:
\begin{enumerate}
\item \begin{displaymath}
\begin{array}{rl}
q_{1,0}+q_{1,-1}=q_{1,0}^{(h)}+q_{1,1}^{(h)}=&\bar{a}\lambda_{1},\vspace{2mm}\\
q_{-1,0}+q_{-1,1}=q_{-1,0}^{(h)}+q_{-1,1}^{(h)}=&a\bar{\lambda}_{1},
\end{array}
\end{displaymath}
\item \begin{displaymath}
\begin{array}{rl}
q_{1,1}^{(h)}+q_{0,1}^{(h)}+q_{-1,1}^{(h)}=\lambda_{2}>a\lambda_{2}=q_{0,1}+q_{-1,1}.
\end{array}
\end{displaymath}
\item \begin{displaymath}
\begin{array}{rl}
q_{0,1}+q_{-1,1}=q_{0,1}^{(v)}+q_{1,1}^{(v)}=&a\lambda_{2},\vspace{2mm}\\
q_{0,-1}+q_{1,-1}=q_{0,-1}^{(v)}+q_{1,-1}^{(v)}=&\bar{a}\bar{\lambda}_{2},
\end{array}
\end{displaymath}
\item \begin{displaymath}
\begin{array}{rl}
q_{1,1}^{(v)}+q_{1,0}^{(v)}+q_{1,-1}^{(v)}=\lambda_{1}>\bar{a}\lambda_{1}=q_{1,0}+q_{1,-1}.
\end{array}
\end{displaymath}
\end{enumerate}
It is straightforward to see that the transition probabilities of the model at hand satisfy Conditions A, B, C, D, given in Section \ref{general}. Moreover, condition C implies that:  $q_{1,0}q_{-1,0}=q_{0,1}q_{0,-1}=q_{-1,1}q_{1,-1}=a\bar{a}\lambda_{1}\lambda_{2}\bar{\lambda}_{1}\bar{\lambda}_{2}$.

It is readily seen, by using the balance principle, as in Section \ref{general}, that we can  derive explicit formulae for the marginal distributions. In particular, by substituting the one step transition probabilities for the model at hand in \eqref{mar1}, \eqref{mar2}, we obtain after some algebra:
\begin{displaymath}
\pi_{m}^{(1)}=\left\{\begin{array}{ll}
     \frac{a-\lambda_{1}}{a\bar{a}}(\tau_{1}^{-1})^{m},&m\geq 1,  \\
    1-\frac{\lambda_{1}}{a}, &m=0, 
\end{array}\right.,\,\,\pi_{n}^{(2)}=\left\{\begin{array}{ll}
     \frac{\bar{a}-\lambda_{2}}{a\bar{a}}(\tau_{2}^{-1})^{n},&n\geq 1,  \\
    1-\frac{\lambda_{2}}{\bar{a}}, &n=0. 
\end{array}\right.
\end{displaymath}

The above expressions are exactly the same as those derived in \cite[Section 3]{devos}, where the authors used the generating function approach. In the following, we briefly apply the approach described in Subsection \ref{compe}. In Figure \ref{p2} we illustrate the curves $(\gamma,\delta)$ satisfying $K(\gamma,\delta)=0$, $H(\gamma,\delta)=0$, $V(\gamma,\delta)=0$, for $\lambda_{1}=0.4$, $\lambda_{2}=0.15$, $a=0.6$.
\paragraph{Step 1 (Initial solution):} We proceed with the constructing the solution starting from a solution satisfying the inner and the horizontal boundary equations. Substituting $\gamma_{0}=\tau_{1}^{-1}<1$ in \eqref{ker} yields a quadratic polynomial with respect to $\delta$, which has two roots, namely $\tau_{T}^{-1}$, and $1$. Clearly, the latter one is rejected. Thus, setting $\delta_{0}=\tau_{T}^{-1}$, the term $c_{0}\gamma_{0}^{m}\delta_{0}^{n}$ satisfies the inner equilibrium equations for $m,n\geq 1$. Thus an initial formal solution satisfying \eqref{ker}, \eqref{hn} (or equivalently, \eqref{int}, \eqref{h1}, \eqref{h2}) is given by \eqref{inig} with $e_{0}=c_{0}\lambda_{1}$.
\begin{figure}[ht!]
\centering
\includegraphics[scale=0.5]{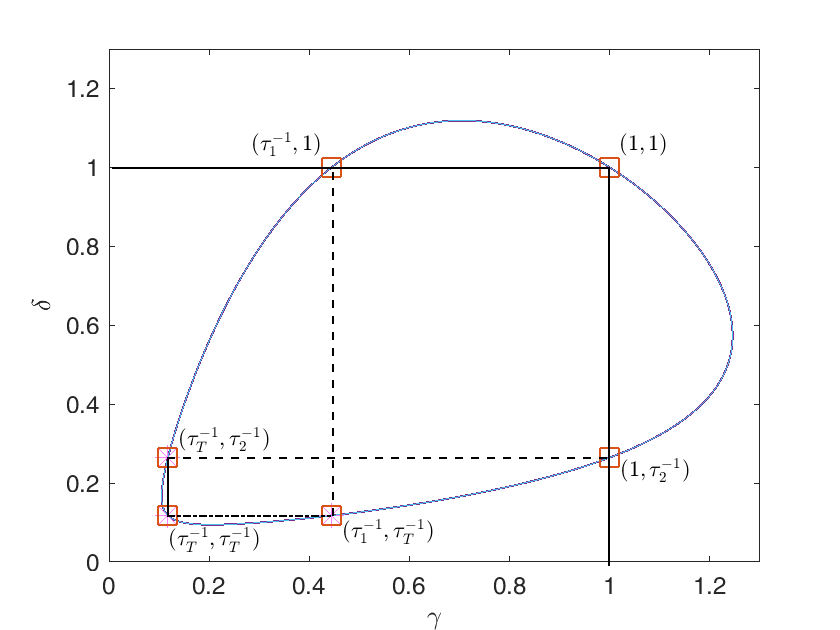}
\caption{The kernel equation \eqref{ker} for $\lambda_{1}=0.4$, $\lambda_{2}=0.15$, $a=0.6$, along with the terms $(\gamma_{0},\delta_{0})=(\tau_{1}^{-1},\tau_{T}^{-1})$, $(\gamma_{1},\delta_{0})=(\tau_{T}^{-1},\tau_{T}^{-1})$, $(\gamma_{1},\delta_{1})=(\tau_{T}^{-1},\tau_{2}^{-1})$.}\label{fw1}
\end{figure}
\paragraph{Step 2 (Vertical compensation step):} The initial solution \eqref{inig} (with $e_{0}=c_{0}\lambda_{1}$) violates \eqref{v1}, \eqref{v2}, 
 and to compensate this error we add a term $c_{1}\gamma\delta_{0}^{n}$. Note that setting $\delta=\delta_{0}=\tau_{T}^{-1}$ in \eqref{ker}, and following the details in Section \ref{general} we have that $\gamma:=\gamma_{1}=\tau_{T}^{-1}$. So, we consider the updated solution \eqref{g1} where the coefficients $c_{1}$, $z_{1}$ are obtained by using Lemma \ref{lemma}, and given by 
\begin{displaymath}
\begin{array}{rl}
c_{1}=&-c_{0}\frac{a\bar{a}(1-\lambda_{1}-\lambda_{2})}{\lambda_{2}\bar{\lambda}_{2}(a-\lambda_{1})},\vspace{2mm}\\
z_{1}=&-c_{0}\frac{\bar{a}\lambda_{1}(\bar{a}-\lambda_{2})}{(a-\lambda_{1})\bar{\lambda}_{2}}.
\end{array}
\end{displaymath}
\paragraph{Step 3 (Horizontal compensation step):} The updated solution \eqref{g1} (with $e_{0}$, $c_{1}$, $z_{1}$ as given above), violates \eqref{h1}, \eqref{h2}. Thus, we need to add another product-form term $d_{1}\gamma_{1}^{m}\delta_{1}^{n}$, where by  following the details in Lemma \ref{hcs}, $(\gamma_{1},\delta_{1})=(\tau_{T}^{-1}, \tau_{2}^{-1})$ and the updated solution after that step is given by \eqref{solg} where the coefficients $d_{1}$, $e_{1}$ satisfy \eqref{h1}, \eqref{h2}, and are given by
\begin{displaymath}
\begin{array}{rl}
e_{1}=&-c_{0}\lambda_{1}\frac{a}{\bar{\lambda}_{2}},\vspace{2mm}\\
d_{1}=&c_{0}\frac{\lambda_{1}\bar{\lambda}_{1}(\bar{a}-\lambda_{2})}{\lambda_{2}\bar{\lambda}_{2}(a-\lambda_{1})}.
\end{array}
\end{displaymath}
\paragraph{Step 4 (Termination of the compensation procedure):}
Again, the updated solution \eqref{solg} (with $e_{0}$, $c_{1}$, $z_{1}$, $e_{1}$, $d_{1}$ as given above) violates \eqref{v1}, \eqref{v2}. By applying Lemma \ref{lemma} for $\delta=\delta_{1}=\tau_{2}^{-1}$ (setting $\delta_{1}=\tau_{2}^{-1}$ in \eqref{ker}, we result in two candidates for $\gamma$, namely $\gamma=\tau_{T}^{-1}$, and $\gamma=1$), and following the lines of the previous section (see \eqref{soltg}-\eqref{solgf}), the compensation approach is terminated since the coefficient of the additional product-form term (for $m>0$, $n\geq 1$) equals zero. For $m=0,n\geq 1$, the new term is $z_{2}\delta_{1}^{n}$, where $z_{2}$ equals:
\begin{equation}
    z_{2}=\lambda_{2}d_{1}=c_{0}\frac{\lambda_{1}\bar{\lambda}_{1}(\bar{a}-\lambda_{2})}{\bar{\lambda}_{2}(a-\lambda_{1})}.\label{cov}
\end{equation}
The final formal solution to \eqref{int}-\eqref{h2} would be
\begin{equation}
x(m,n)=\left\{\begin{array}{ll}
c_{0}\gamma_{0}^{m}\delta_{0}^{n}+c_{1}\gamma_{1}^{m}\delta_{0}^{n}+d_{1}\gamma_{1}^{m}\delta_{1}^{n},&m\geq 1,n\geq 1,\\
z_{1}\delta_{0}^{n}+z_{2}\delta_{1}^{n},&m=0,n\geq 1,\\
e_{0}\gamma_{0}^{m}+e_{1}\gamma_{1}^{m},&m>0,n=0,
\end{array}\right.\label{fofin}
\end{equation}
Define for all $m\geq0$, $n=0$,
\begin{equation}
x(m,0)=e_{0}\gamma_{0}^{m}+e_{1}\gamma_{1}^{m}.\label{001}
\end{equation}
Note that $e_{0}+e_{1}=c_{0}\frac{\lambda_{1}(\bar{a}-\lambda_{2})}{\bar{\lambda}_{2}}=z_{1}+z_{2}$,
thus, we can also set
\begin{equation}
x(0,n)=z_{1}\delta_{0}^{n}+z_{2}\delta_{1}^{n},\,m=0,n\geq 0.\label{002}
\end{equation}
Therefore, $x(0,0)$ can be defined either from \eqref{001} or from \eqref{002}.

Finally, we  proved that
\begin{equation}
x(m,n)=\left\{\begin{array}{ll}
c_{0}\gamma_{0}^{m}\delta_{0}^{n}+c_{11}\gamma_{1}^{m}\delta_{0}^{n}+d_{1}\gamma_{1}^{m}\delta_{1}^{n},&m\geq 1,n\geq 1,\\
z_{1}\delta_{0}^{n}+z_{2}\delta_{1}^{n},&m=0,n\geq 0,\\
e_{0}\gamma_{0}^{m}+e_{1}\gamma_{1}^{m},&m\geq 0,n=0,
\end{array}\right.\label{fofinv}
\end{equation}
satisfies \eqref{int}-\eqref{h2}, except \eqref{00}-\eqref{11}. 
\paragraph{Step 5 (Closing step):} Using the same arguments as given in Subsection \ref{compe} (i.e., the corresponding random walk satisfies Conditions A, B, C, D discussed in Section \ref{general}), we can show that \eqref{fofinv} satisfies also \eqref{00}-\eqref{11}.

The following theorem is our main result and summarizes the analysis in that subsection. It is a direct consequence of Theorem \ref{mainr}, when we substitute the coefficients $e_{0}$, $e_{1}$, $c_{1}$, $d_{1}$, $z_{1}$, $z_{2}$, derived above. It remains to obtain the constant $c_{0}$, by using the normalization equation. Asking, $\sum_{m=0}^{\infty}\sum_{n=0}^{\infty}\pi_{m,n}=1$, we obtain after lengthy but straightforward calculations that
\begin{equation*}
c_{0}=\frac{(a-\lambda_{1})(1-\lambda_{1}-\lambda_{2})}{a\bar{a}\lambda_{1}\bar{\lambda}_{1}}.\label{cons}
\end{equation*} 
\begin{theorem}
Under stability condition, the solution of balance equations \eqref{int}-\eqref{11} is given as follows:
\begin{equation}
\pi_{m,n}=\left\{\begin{array}{ll}
\frac{(a-\lambda_{1})(1-\lambda_{1}-\lambda_{2})}{a\bar{a}\lambda_{1}\bar{\lambda}_{1}}\frac{1}{\tau_{1}^{m}}\frac{1}{\tau_{T}^{n}}+\frac{(\bar{a}-\lambda_{2})(1-\lambda_{1}-\lambda_{2})}{a\bar{a}\lambda_{2}\bar{\lambda}_{2}}\frac{1}{\tau_{T}^{m}}\frac{1}{\tau_{2}^{n}}-\frac{(1-\lambda_{1}-\lambda_{2})^{2}}{\lambda_{1}\bar{\lambda}_{1}\lambda_{2}\bar{\lambda}_{2}}\frac{1}{\tau_{T}^{m+n}},&m,n\geq 1,\vspace{2mm}\\
\frac{(\bar{a}-\lambda_{2})(1-\lambda_{1}-\lambda_{2})}{a\bar{\lambda}_{2}}\left(\frac{1}{\bar{a}}\frac{1}{\tau_{2}^{n}}-\frac{1}{\bar{\lambda}_{1}}\frac{1}{\tau_{T}^{n}}\right),&m=0,n\geq 0,\vspace{2mm}\\
\frac{(a-\lambda_{1})(1-\lambda_{1}-\lambda_{2})}{\bar{a}\bar{\lambda}_{1}}\left(\frac{1}{a}\frac{1}{\tau_{1}^{m}}-\frac{1}{\bar{\lambda}_{2}}\frac{1}{\tau_{T}^{m}}\right),&m\geq 0,n=0,\vspace{2mm}\\
\frac{(a-\lambda_{1})(\bar{a}-\lambda_{2})(1-\lambda_{1}-\lambda_{2})}{a\bar{a}\bar{\lambda}_{1}\bar{\lambda}_{2}},&m=n=0.
\end{array}\right.
\label{mainc}
\end{equation}
\end{theorem}
It is readily seen that \eqref{mainc} coincides with the joint probability distribution given in \cite[Theorem 3]{devos}. We can also arrive at the same result, by using the above steps, starting now with an initial solution that satisfies the interior balance equations, and the vertical boundary equations. This is due to the symmetry that conditions A, B, C, D inherit in our model; see Remark \ref{remark}.
\begin{remark}\label{r2}
Note that the initial terms $\gamma_{0}=\tau_{1}^{-1}$, (resp. $\widehat{\delta}_{0}=\tau_{2}^{-1}$, when we start compensating with an initial solution satisfying \eqref{int}-\eqref{v2}) have a clear probabilistic interpretation. In particular, $\tau_{1}^{-1}=\frac{\bar{a}\lambda_{1}}{a\bar{\lambda}_{1}}=\frac{q_{1,0}+q_{1,-1}}{q_{-1,0}+q_{-1,1}}$ (resp. $\tau_{2}^{-1}=\frac{a\lambda_{2}}{\bar{a}\bar{\lambda}_{2}}=\frac{q_{0,1}+q_{-1,1}}{q_{0,-1}+q_{1,-1}}$), which is the quotient of the probability of increasing by one the number of jobs in queue 1 (resp. queue 2) in a slot, divided by the probability of decreasing by one the number of jobs in queue 1 (resp. queue 2) in a slot. Note also that $\tau_{T}=\tau_{1}\tau_{2}$.
\end{remark}

\begin{remark}\label{rem18}
Note that the random walk discussed in this subsection (i.e., the one in  \cite{devos}), satisfies a modified version of the Detection Algorithm in \cite{chen2}; see also \cite{chen}. Remind that this algorithm provides necessary conditions for the invariant measure of a random walk to be written as a sum of geometric terms. However, in \cite{chen2}, the authors proved this result for a class of random walks that has a different structure from ours; see the transition diagram in \cite[p. 24, Fig. 1]{chen2}. Our intuition indicates that the Detection algorithm (or a slight modification) can be applied for a broader class of random walks that include ours. In particular, we can show that our random walk has an invariant measure that can be written as a sum of three product-form terms, by following the lines in \cite[subsection 3.2]{chen2} and applying a slight modification of the Detection algorithm: Let $H_{set}=\{(\gamma,\delta)\in(0,1)^{2}:(\gamma,\delta)\in K\cap H\}=\{(\tau_{1}^{-1},\tau_{T}^{-1})\}$, where $H(\gamma,\delta)=H_{0}(\gamma,\delta)+H_{1}(\gamma,\delta)$, and $V_{set}=\{(\widehat{\gamma},\widehat{\delta})\in(0,1)^{2}:(\widehat{\gamma},\widehat{\delta})\in K\cap H\}=\{(\tau_{T}^{-1},\tau_{2}^{-1})\}$, where $V(\gamma,\delta)=V_{0}(\gamma,\delta)+V_{1}(\gamma,\delta)$ (note that in the model in \cite{chen2}, the boundary behaviour is simpler, i.e., there is only a single horizontal and vertical boundary equation, while in our case there are two horizontal and two vertical boundary equations). 

The algorithm indicates that we first start from $(\rho_{1},\sigma_{1})\in H_{set}$, and construct a set $\Gamma^{H}$ as follows: $\rho_{2k}\neq \rho_{2k-1}$, $\sigma_{2k}= \sigma_{2k-1}$ and $\rho_{2k+1}= \rho_{2k}$, $\sigma_{2k+1}\neq \sigma_{2k}$, $k=1,2,\ldots$, and continue until we have $(\rho_{n},\sigma_{n})\in(0,1)^{2}$ and $(\rho_{n+1},\sigma_{n+1})\notin(0,1)^{2}$. In our case $\Gamma^{H}=\{(\tau_{1}^{-1},\tau_{T}^{-1}),(\tau_{T}^{-1},\tau_{T}^{-1}),(\tau_{T}^{-1},\tau_{2}^{-1})\}$. Similarly, starting from $(\rho_{1},\sigma_{1})\in V_{set}$, and construct a set $\Gamma^{V}$ as follows: $\rho_{2k}= \rho_{2k-1}$, $\sigma_{2k}\neq \sigma_{2k-1}$ and $\rho_{2k+1}\neq \rho_{2k}$, $\sigma_{2k+1} =\sigma_{2k}$, $k=1,2,\ldots$, and continue until we have $(\rho_{n},\sigma_{n})\in(0,1)^{2}$ and $(\rho_{n+1},\sigma_{n+1})\notin(0,1)^{2}$. In our case $\Gamma^{V}=\{(\tau_{T}^{-1},\tau_{2}^{-1}),(\tau_{T}^{-1},\tau_{T}^{-1}),(\tau_{1}^{-1},\tau_{T}^{-1})\}$.

The next step is to check whether the geometric terms are properly coupled \cite[pp. 29-30]{chen2}. According to the established conditions, for $\Gamma^{H}$, $(\rho_{3},\sigma_{3})=(\tau_{T}^{-1},\tau_{2}^{-1})\in V_{set}$, thus the invariant measure is induced by $\{(\rho_{1},\sigma_{1}),(\rho_{2},\sigma_{2}),(\rho_{3},\sigma_{3})\}=\{(\tau_{1}^{-1},\tau_{T}^{-1}),(\tau_{T}^{-1},\tau_{T}^{-1}),(\tau_{T}^{-1},\tau_{2}^{-1})\}$ and for $\Gamma^{V}$, $(\rho_{3},\sigma_{3})=(\tau_{1}^{-1},\tau_{T}^{-1})\in H_{set}$, thus the invariant measure is induced by $\{(\rho_{1},\sigma_{1}),(\rho_{2},\sigma_{2}),(\rho_{3},\sigma_{3})\}=\{(\tau_{T}^{-1},\tau_{2}^{-1}),(\tau_{T}^{-1},\tau_{T}^{-1}),(\tau_{1}^{-1},\tau_{T}^{-1})\}$. Thus, the invariant measure is written as a sum of geometric terms, as we also shown. 

Therefore, although we do not give a rigorous proof, we believe that the Detection algorithm can be applied in a broader class of random walks (so that to contains ours) than the one investigated in \cite{chen2}. 
\end{remark}
\begin{remark}
As already noted, Condition C provides insight into the properties that should be satisfied by the transition probabilities in the interior of the state space. In particular, for the model at hand we have:
\begin{enumerate}
\item $q_{0,1}q_{0,-1}=q_{-1,1}q_{1,-1}\Rightarrow \frac{q_{0,1}}{q_{-1,1}}=\frac{\lambda_{1}}{\bar{\lambda}_{1}}=\frac{q_{0,-1}}{q_{1,-1}}$.
\item $q_{0,1}q_{0,-1}=q_{-1,1}q_{1,-1}\Rightarrow \frac{q_{0,1}}{q_{1,-1}}=\frac{a\lambda_{2}}{\bar{a}\bar{\lambda}_{2}}=\tau_{2}^{-1}=\frac{q_{-1,1}}{q_{0,-1}}$. Note that $\tau_{2}^{-1}=\frac{q_{0,1}+q_{-1,1}}{q_{0,-1}+q_{1,-1}}$.
\item $q_{1,0}q_{-1,0}=q_{-1,1}q_{1,-1}\Rightarrow \frac{q_{1,0}}{q_{1,-1}}=\frac{\lambda_{2}}{\bar{\lambda}_{2}}=\frac{q_{-1,1}}{q_{-1,0}}$.
\item $q_{1,0}q_{-1,0}=q_{-1,1}q_{1,-1}\Rightarrow \frac{q_{1,0}}{q_{-1,1}}=\frac{\bar{a}\lambda_{1}}{a\bar{\lambda}_{1}}=\tau_{1}^{-1}=\frac{q_{1,-1}}{q_{-1,0}}$. Note that $\tau_{1}^{-1}=\frac{q_{1,0}+q_{1,-1}}{q_{-1,0}+q_{-1,1}}$.
\item $q_{1,0}q_{-1,0}=q_{0,1}q_{0,-1}\Rightarrow \frac{q_{0,1}}{q_{1,0}}=\frac{a}{\bar{a}}=\frac{q_{-1,0}}{q_{0,-1}}$.
\item $q_{1,0}q_{-1,0}=q_{0,1}q_{0,-1}\Rightarrow \frac{q_{0,1}}{q_{-1,0}}=\frac{\lambda_{1}\lambda_{2}}{\bar{\lambda}_{1}\bar{\lambda}_{2}}=\tau_{T}^{-1}=\frac{q_{1,0}}{q_{0,-1}}$. Note that $\tau_{T}^{-1}=\tau_{1}^{-1}\tau_{2}^{-1}$.
\end{enumerate}
\end{remark}

\subsection{The model with a single stream: Simultaneous arrivals}\label{prod}
Here we consider a modification of the model in Subsection \ref{prel}. In particular, we consider only a single stream of jobs that add a single job at each queue; see \cite[subsection 2.4]{devosphd}. More precisely, during any slot, the
number of type 1 arrivals is the same as the number of type 2 arrivals.

The author in \cite[subsection 2.4]{devosphd}, by following the lines in \cite{konheim} and using complex analytic arguments, derives in an elegant solution for the joint stationary queue length distribution; see \cite[Theorem 2.5]{devosphd}. In particular, he shows that it is of product-form (i.e., only a single product-form term is needed to satisfy all equilibrium equations). Our approach is simpler and probabilistic, and thanks to conditions A, B, C, D, the queueing model at hand belongs to the class of two-dimensional random walks presented in Subsection \ref{special}. Moreover, it is surprising that although the queue contents are correlated, the joint stationary distribution is given as a single product of two geometric terms. 

This queueing model is characterized as follows: 
\begin{itemize}
\item In each time slot, the server chooses queue $1$ (resp. 2) with probability $a$ (resp. $\bar{a}$). 
\item A job arrives at each queue with probability $\lambda$ ($\bar{\lambda}=1-\lambda$). Note that in this case the probability generating function of the arrivals is $A(z_{1},z_{2})=\bar{\lambda}+\lambda z_{1}z_{2}$.
\item Note that, contrary to the model in Subsection \ref{prel}, the number of type 1 and type 2 arrivals are correlated. Indeed, if $a_{j,k}$ denotes the number of type-$j$ arrivals during slot $k$, $j=1,2,$ then, $cov[a_{1,k},a_{2,k}]=\lambda\bar{\lambda}>0$.
\end{itemize} 
In \cite[Subsection 2.4, Theorem 2.5]{devosphd}, by using complex analytic arguments, it was shown that the invariant measure has a product-form solution, i.e., $c(\tau_{1}^{-1})^{m}(\tau_{2}^{-1})^{n}$, $m,n>0$, where now 
\begin{displaymath}
    \rho_{1}=\tau_{1}^{-1}=\frac{\bar{a}\lambda}{a\bar{\lambda}},\,\,\rho_{2}=\tau_{2}^{-1}=\frac{a\lambda}{\bar{a}\bar{\lambda}}.
\end{displaymath}
The one step transition probabilities are now (see also Figure \ref{relax1}, with $q_{0,1}^{(0)}=q_{1,0}^{(0)}=0$):
\begin{itemize}
\item $m,n>0$,
\begin{displaymath}
q_{0,1}=a\lambda,\,q_{1,0}=\bar{a}\lambda,\,q_{-1,0}=a\bar{\lambda},\,q_{0,-1}=\bar{a}\bar{\lambda}.
\end{displaymath}
\item $m>0$, $n=0$,
\begin{displaymath}
q_{0,1}^{(h)}=a\lambda,\,q_{1,1}^{(h)}=\bar{a}\lambda,\,q_{-1,0}^{(h)}=a\bar{\lambda},\,q_{0,0}^{(h)}=\bar{a}\bar{\lambda}.
\end{displaymath}
\item $m=0$, $n>0$,
\begin{displaymath}
q_{1,1}^{(v)}=a\lambda,\,q_{1,0}^{(v)}=\bar{a}\lambda,\,q_{0,0}^{(v)}=a\bar{\lambda},\,q_{0,-1}^{(v)}=\bar{a}\bar{\lambda}.
\end{displaymath}
\item $m=n=0$,
\begin{displaymath}
q_{1,1}^{(0)}=\lambda,\,q_{0,0}^{(0)}=\bar{\lambda}.
\end{displaymath}
\end{itemize}
This queueing model is stable if and only if $\lambda<\min\{a,1-a\}$; see Theorem \ref{stabi}.
The one-step transition probabilities satisfy the conditions A, and the updated versions of Conditions B, C, D; see Subsection \ref{special}. In particular:
\begin{enumerate}
\item $q_{1,1}^{(h)}=q_{1,0}$, $q_{-1,0}^{(h)}=q_{-1,0}$, $q_{1,0}^{(h)}=0=q_{1,-1}$, $q_{-1,1}^{(h)}=0=q_{-1,1}$.
\item $q_{1,1}^{(v)}=q_{0,1}$, $q_{0,-1}^{(v)}=q_{0,-1}$, $q_{0,1}^{(v)}=0=q_{-1,1}$, $q_{1,-1}^{(v)}=0=q_{1,-1}$.
\item $q_{0,1}^{(h)}=q_{0,1}$, $q_{1,0}^{(v)}=q_{1,0}$.
\item \begin{itemize}
\item $q_{0,0}^{(0)}+q_{0,0}=\bar{\lambda}=a\bar{\lambda}+\bar{a}\bar{\lambda}=q_{0,0}^{(h)}+q_{0,0}^{(v)}$,
\item $q_{0,1}^{(0)}+q_{0,1}=a\lambda=q_{0,1}^{(h)}+q_{0,1}^{(v)}$,
\item $q_{1,0}^{(0)}+q_{1,0}=\bar{a}\lambda=q_{1,0}^{(h)}+q_{1,0}^{(v)}$,
\item $q_{1,1}^{(0)}+q_{1,1}=\lambda=a\lambda+\bar{a}\lambda=q_{1,1}^{(h)}+q_{1,1}^{(v)}$.
\end{itemize}
\item $q_{0,1}q_{0,-1}=a\bar{a}\lambda\bar{\lambda}=q_{1,0}q_{-1,0}$.
\end{enumerate}
Using similar arguments as in Section \ref{general} we can  derive explicit formulae for the marginal distributions. By substituting the one-step transition probabilities for the model at hand in \eqref{mar1}, \eqref{mar2}, we obtain after some algebra:
\begin{displaymath}
\pi_{m}^{(1)}=\left\{\begin{array}{ll}
     \frac{a-\lambda}{a\bar{a}}(\tau_{1}^{-1})^{m},&m\geq 1,  \\
    1-\frac{\lambda}{a}, &m=0, 
\end{array}\right.,\,\,\pi_{n}^{(2)}=\left\{\begin{array}{ll}
     \frac{\bar{a}-\lambda}{a\bar{a}}(\tau_{2}^{-1})^{n},&n\geq 1,  \\
    1-\frac{\lambda}{\bar{a}}, &n=0. 
\end{array}\right.
\end{displaymath}
The above expressions are exactly the same as those derived in \cite[subsection 2.4]{devosphd}, by using the generating function approach.

Let $\pi_{m,n}$ the stationary joint queue length distribution. Then, the equilibrium equations are given by
\begin{equation}
\begin{array}{rl}
\pi_{m,n}=&\pi_{m,n-1}q_{0,1}+\pi_{m-1,n}q_{1,0}+\pi_{m,n+1}q_{0,-1}+\pi_{m+1,n}q_{-1,0},\,m>1,n>1,
\end{array}\label{ints}
\end{equation}
\begin{equation}
\begin{array}{rl}
\pi_{1,n}=&\pi_{1,n-1}q_{0,1}+\pi_{0,n}q_{1,0}^{(v)}+\pi_{1,n+1}q_{0,-1}+\pi_{2,n}q_{-1,0}+\pi_{0,n-1}q_{1,1}^{(v)},\,n>1,
\end{array}\label{v1s}
\end{equation}
\begin{equation}
\begin{array}{rl}
\pi_{0,n}(1-q_{0,0}^{(v)})=&\pi_{0,n+1}q_{0,-1}^{(v)}+\pi_{1,n}q_{-1,0},\,n>1,
\end{array}\label{v2s}
\end{equation}
\begin{equation}
\begin{array}{rl}
\pi_{m,1}=&\pi_{m,0}q_{0,1}^{(h)}+\pi_{m-1,1}q_{1,0}+\pi_{m,2}q_{0,-1}+\pi_{m+1,1}q_{-1,0}+\pi_{m-1,0}q_{1,1}^{(h)},\,m>1,
\end{array}\label{h1s}
\end{equation}
\begin{equation}
\begin{array}{rl}
\pi_{m,0}(1-q_{0,0}^{(h)})=&\pi_{m+1,0}q_{-1,0}^{(h)}+\pi_{m,1}q_{0,-1},\,m>1.
\end{array}\label{h2s}
\end{equation}
\begin{equation}
\pi_{0,0}(1-q_{0,0}^{(0)})=\pi_{1,0}q_{-1,0}^{(h)}+\pi_{0,1}q_{0,-1}^{(v)},
\label{00s}
\end{equation}
\begin{equation}
\pi_{0,1}(1-q_{0,0}^{(v)})=\pi_{1,1}q_{-1,0}+\pi_{0,2}q_{0,-1}^{(v)},\label{01s}
\end{equation}
\begin{equation}
\pi_{1,0}(1-q_{0,0}^{(h)})=\pi_{1,1}q_{0,-1}+\pi_{2,0}q_{-1,0}^{(h)},\label{10s}
\end{equation}
\begin{equation}
\pi_{1,1}=\pi_{0,0}q_{1,1}^{(0)}+\pi_{0,1}q_{1,0}^{(v)}+\pi_{1,2}q_{0,-1}+\pi_{2,1}q_{-1,0}+\pi_{1,0}q_{0,1}^{(h)}.\label{11s}
\end{equation}
Now, we have
\begin{displaymath}
K(\gamma,\delta)=\gamma\delta-q_{0,1}\gamma-q_{1,0}\delta-q_{0,-1}\gamma\delta^{2}-q_{-1,0}\gamma^{2}\delta.
\end{displaymath}
Moreover, substitute the product $\gamma^{m}\delta^{n}$ in the corresponding boundary equations  to obtain the following equations:
\begin{equation}
\begin{array}{rl}
V_{1}(\gamma,\delta)=&\gamma\delta-q_{0,1}\gamma-q_{1,0}^{(v)}\delta-q_{0,-1}\gamma\delta^{2}-q_{-1,0}\gamma^{2}\delta-q_{1,1}^{(v)}=0,\vspace{2mm}\\
V_{0}(\gamma,\delta)=&\delta(1-q_{0,0}^{(v)})-q_{0,-1}^{(v)}\delta^{2}-q_{-1,0}\gamma\delta=0,
\end{array}\label{vxsn}
\end{equation}
\begin{equation}
\begin{array}{rl}
H_{1}(\gamma,\delta)=&\gamma\delta-q_{0,1}^{(h)}\gamma-q_{1,0}\delta-q_{0,-1}\gamma\delta^{2}-q_{-1,0}\gamma^{2}\delta-q_{1,1}^{(h)}=0,\vspace{2mm}\\
H_{0}(\gamma,\delta)=&\gamma(1-q_{0,0}^{(h)})-q_{-1,0}^{(h)}\gamma^{2}-q_{0,-1}\gamma\delta=0.
\end{array}\label{hxsn}
\end{equation}
Set $H(\gamma,\delta):=H_{1}(\gamma,\delta)+H_{0}(\gamma,\delta)$, and $V(\gamma,\delta):=V_{1}(\gamma,\delta)+V_{0}(\gamma,\delta)$. These curves are plotted in Figure \ref{pon10}.

We now briefly describe the procedure. Observe that starting with $\gamma_{0}=\tau_{1}^{-1}$, $K(\gamma_{0},\delta)=0$ is a quadratic polynomial with two roots $\delta=\tau_{2}^{-1}$, and $\delta=1$. Clearly, the latter is rejected. Following the lines in Subsection \ref{special} the solution in \eqref{o1}  
satisfies the interior and the horizontal boundary equations where now $e_{0}=c_{0}\lambda$. By applying Lemma \ref{lemma} (i.e., a vertical compensation step with $\delta_{0}=\tau_{2}^{-1}$, for which $K(\gamma,\delta_{0})=0$ has two roots, namely $\gamma=\tau_{1}^{-1}$ and $\gamma=1$), we realize that the coefficient of the additional product-form term vanishes, so compensation is terminated by using the initial solution; see details in Section \ref{special}. Therefore, the formal solution is given by \eqref{o2}, where now, simple algebraic calculations leads to $z_{1}=e_{0}=c_{0}\lambda$.
%
Thus, \eqref{o2} (with $z_{1}=e_{0}=c_{0}\lambda_{1}$) satisfies \eqref{ints}-\eqref{h2s}. Setting,
\begin{displaymath}
x(m,0)=e_{0}\gamma_{0}^{m},\,m\geq 0,\,n=0,
\end{displaymath}
the solution in \eqref{soll} (with $z_{1}=e_{0}=c_{0}\lambda_{1}$) satisfies \eqref{ints}-\eqref{h2s}.
Following the lines as in Subsection \ref{special}, we can show that \eqref{soll} (with $z_{1}=e_{0}=c_{0}\lambda_{1}$) satisfies also the remaining equations \eqref{00s}-\eqref{11s}, and thus, it constitutes a  solution to all equilibrium equations.

Note that \eqref{soll} (with $z_{1}=e_{0}=c_{0}\lambda_{1}$) is rewritten as
\begin{equation}
x(m,n)=c_{0}\times\left\{\begin{array}{ll}
(\tau_{1}^{-1})^{m}(\tau_{2}^{-1})^{n},&m,n\geq 1,\\
\lambda(\tau_{2}^{-1})^{n},&m=0,n\geq 0,\\
\lambda(\tau_{1}^{-1})^{m},&m\geq 0,n=0.
\end{array}\right.\label{fosu12}
\end{equation}
By using the normalization condition and \eqref{fosu12}, we obtain after straightforward computations that
\begin{displaymath}
c_{0}=\frac{(a-\lambda)(\bar{a}-\lambda)}{a\bar{a}\bar{\lambda}\lambda}.
\end{displaymath}
The next theorem summarizes the main result:
\begin{theorem}
For $\lambda<\min\{a,\bar{a}\}$, and assuming that conditions A, B, C, D are satisfied, the solution to equilibrium equations \eqref{ints}-\eqref{11s} is given by 
\begin{equation}
\pi_{m,n}=\left\{\begin{array}{ll}
\frac{(a-\lambda)(\bar{a}-\lambda)}{a\bar{a}\bar{\lambda}\lambda}(\tau_{1}^{-1})^{m}(\tau_{2}^{-1})^{n},&m,n>0,\\
\frac{(a-\lambda)(\bar{a}-\lambda)}{a\bar{a}\bar{\lambda}}(\tau_{2}^{-1})^{n},&m=0,n\geq 0,\\
\frac{(a-\lambda)(\bar{a}-\lambda)}{a\bar{a}\bar{\lambda}}(\tau_{1}^{-1})^{m},&m\geq 0,n=0,\\
\frac{(a-\lambda)(\bar{a}-\lambda)}{a\bar{a}\bar{\lambda}},&m=n=0.
\end{array}\right.\label{finsu}
\end{equation}
\end{theorem}
Note that \eqref{finsu} coincides with the expression given in \cite[Theorem 2.5]{devosphd}, where the author obtained the same solution by using complex analytic arguments, by following the lines in \cite{konheim}.

We can also arrive at the same result by starting with an initial solution that satisfies the interior and the vertical boundary conditions, i.e., by compensating with initial term $\widehat{\delta}_{0}=\tau_{2}^{-1}$. 
\subsubsection{Some comments on the conditions B and D}
We now investigate  the importance of Conditions B and D (which are immediate consequences of the non-work conserving policy) for the derivation of solution \eqref{finsu}. In particular, these conditions considerably deactivate the effect of the transition probabilities from boundaries on the balance equations \eqref{00s}-\eqref{11s}, and result in a detailed balance principle among the states. This result is expected due to the product-form solution \eqref{finsu}. 

Having in mind that $1-q_{0,0}^{(0)}=q_{1,1}^{(0)}=q_{1,1}^{(h)}+q_{1,1}^{(v)}$, due to the Condition B.2, and that $q_{1,1}^{(h)}=q_{1,0}$, $q_{1,1}^{(v)}=q_{0,1}$, $q_{-1,0}^{(h)}=q_{-1,0}$, $q_{0,-1}^{(v)}=q_{0,-1}$ due to the Condition B.1, \eqref{00s} can be rewritten as
\begin{equation}
\pi_{0,0}(q_{0,1}+q_{1,0})=\pi_{1,0}q_{-1,0}+\pi_{0,1}q_{0,-1}.\label{bnm}
\end{equation}
The resulting equation states that, under Condition B, the behaviour of the process in the origin is the same as the behaviour in the interior state-space, and more importantly, we can observe that the detailed balance principle applies:
\begin{displaymath}
\pi_{0,0}q_{0,1}=\pi_{0,1}q_{0,-1}\text{ and }\pi_{0,0}q_{1,0}=\pi_{1,0}q_{-1,0}.
\end{displaymath}
Indeed, using \eqref{fosu12},
\begin{displaymath}
\pi_{0,0}q_{0,1}=\pi_{0,1}q_{0,-1}\Leftrightarrow a\lambda c_{0}\frac{\lambda}{\tau_{2}}=c_{0}\frac{\lambda}{\tau_{2}^{2}}\bar{a}\bar{\lambda}\Leftrightarrow a\lambda=\bar{a}\bar{\lambda}\frac{a\lambda}{\bar{a}\bar{\lambda}}\Leftrightarrow a\lambda=a\lambda,
\end{displaymath}
which is true. This result is expected due to the form of \eqref{finsu}, but it is surprising, since the number of arriving jobs of type 1 and 2 are correlated. Similarly, \eqref{01s}, due to  Condition B and Condition D (i.e., $q_{1,0}^{(v)}=q_{1,0}$) becomes
\begin{displaymath}
\pi_{0,1}(q_{0,1}+q_{1,0}+q_{0,-1})=\pi_{1,1}q_{-1,0}+\pi_{0,2}q_{0,-1}.
\end{displaymath}
Note that $\pi_{0,1}q_{0,1}=\pi_{0,2}q_{0,-1}$, thus,
\begin{displaymath}
\pi_{0,1}(q_{1,0}+q_{0,-1})=\pi_{1,1}q_{-1,0}.
\end{displaymath}
Similarly, \eqref{10s} is written as
\begin{displaymath}
\pi_{1,0}(q_{0,1}+q_{-1,0})=\pi_{1,1}q_{0,-1},
\end{displaymath}
since $\pi_{1,0}q_{1,0}=\pi_{2,0}q_{-1,0}$ (note that $q_{0,1}^{(h)}=q_{0,1}$ due to Condition D). 

Due to the Conditions B, D, and \eqref{bnm}, equation \eqref{11s} can be written as
 \begin{equation}
\pi_{1,1}(q_{0,1}+q_{1,0})=\pi_{1,2}q_{0,-1}+\pi_{2,1}q_{-1,0}.\label{bnm1}
\end{equation}
and observe, by using \eqref{finsu}, that
\begin{displaymath}
\pi_{1,1}q_{0,1}=\pi_{1,2}q_{0,-1}\text{ and }\pi_{1,1}q_{1,0}=\pi_{2,1}q_{-1,0}.
\end{displaymath}
Therefore, Conditions B, D, inherit the behaviour of the transition probabilities in the interior state space to the boundaries and to the origin, and for the specific model, creates a time-reversibility framework. To conclude,
\begin{displaymath}
\begin{array}{rl}
\pi_{0,n}(q_{1,0}+q_{0,-1})=&\pi_{1,n}q_{-1,0},\,n\geq 1,\\
\pi_{m,0}(q_{0,1}+q_{-1,0})=&\pi_{m,1}q_{0,-1},\,m\geq 1,\\
\pi_{1,0}=&\frac{q_{1,0}}{q_{-1,0}}\pi_{0,0}=(\tau_{1}^{-1})\pi_{0,0},\\
\pi_{0,1}=&\frac{q_{0,1}}{q_{0,-1}}\pi_{0,0}=(\tau_{2}^{-1})\pi_{0,0},\\
\pi_{m+1,0}=&\frac{q_{1,0}}{q_{-1,0}}\pi_{m,0}\Leftrightarrow\pi_{m,0}=(\tau_{1}^{-1})^{m}\pi_{0,0},\,m\geq 0,\\
\pi_{0,n+1}=&\frac{q_{0,1}}{q_{0,-1}}\pi_{n,0}\Leftrightarrow\pi_{0,n}=(\tau_{2}^{-1})^{n}\pi_{0,0},\,n\geq 0.
\end{array}
\end{displaymath}
Note that
\begin{displaymath}
\pi_{2,n}=\tau_{1}^{-1}\pi_{1,n}=\tau_{1}^{-1}\frac{q_{1,0}+q_{0,-1}}{q_{-1,0}}(\tau_{2}^{-1})^{n}\pi_{0,0}=\frac{\pi_{0,0}}{\lambda}(\tau_{1}^{-1})^{2}(\tau_{2}^{-1})^{n}.
\end{displaymath}
Similarly, we can show that
\begin{displaymath}
\pi_{m,n}=\frac{\pi_{0,0}}{\lambda}(\tau_{1}^{-1})^{m}(\tau_{2}^{-1})^{n},\,m,n\geq 1.
\end{displaymath}
Using the normalization condition, we obtain
\begin{displaymath}
\pi_{0,0}=\frac{(\bar{a}-\lambda)(a-\lambda)}{a\bar{a}\bar{\lambda}}.
\end{displaymath}

Note that the resulting solution is the same as the one derived above by using the compensation method. We revealed in this subsection the importance of conditions B, D which allow for independence of the balance equations in $(0,0)$, $(0,1)$, $(1,0)$, $(1,1)$ from the effect of boundaries. Equivalently, the transition probabilities on the boundaries, and the origin are the same as those in the interior of the state space. 
\subsection{Geometric arrivals that are probabilistically routed to the queues}\label{geoprod}
Here we consider the following modification of the model in subsection \ref{prel} (see \cite[subsection 2.5]{devosphd}): Assume that the total numbers of newly arriving jobs during consecutive slots are independent and geometrically distributed with mean $\lambda_{T}$, i.e., $a_{T}(n):=P(n\text{ job arrivals in a slot})=\frac{1}{1+\lambda_{T}}(\frac{\lambda_{T}}{1+\lambda_{T}})^{n}$. Such arrival process corresponds to a batch arrival process with geometrically distributed inter-arrival times and the batch sizes are
geometrically distributed. We further assume that an arriving job is routed to queue $k$ with probability $\frac{\lambda_{k}}{\lambda_{T}}$, $k=1,2$ ($\lambda_{1}+\lambda_{2}=\lambda_{T}$). Simple calculations show that
\begin{equation}
\begin{array}{rl}
     a(i,j)=&P(i\text{ class 1 jobs and }j\text{ class 2 jobs arrive in a slot}) =\frac{1}{1+\lambda_{T}}\binom{i+j}{i}(\frac{\lambda_{1}}{1+\lambda_{T}})^{i}(\frac{\lambda_{2}}{1+\lambda_{T}})^{j}.
\end{array}\label{vghb}
\end{equation}
The service times of jobs equal one slot. At the beginning of every time
slot, the single server randomly selects either queue 1 (with probability $a$) or queue 2 (with probability $\bar{a}=1-a$) to serve. For such a model the author in \cite[subsection 2.5]{devosphd}, using complex analytic arguments shown that the stationary joint occupancy measure is of product-form, and in particular the joint probability mass function $\pi_{m,n}=\pi_{m}^{(1)}\pi_{n}^{(2)}$, i.e., the product of the marginal distributions, with $\pi_{m}^{(1)}=(1-\frac{\lambda_{1}}{a})(\frac{\lambda_{1}}{a})^{m}$, $\pi_{n}^{(2)}=(1-\frac{\lambda_{2}}{\bar{a}})(\frac{\lambda_{2}}{\bar{a}})^{n}$. Note that this result is quite impressive since the corresponding two-dimensional random walk is no longer nearest neighbour, but instead, it has arbitrary large jumps to the East, North and North-East. The balance equations are:
\begin{equation}
    \begin{array}{rl}
         \pi_{m,n}=&a[\sum_{j=0}^{n}a(m,n-j)\pi_{0,j}+\sum_{i=0}^{m}\sum_{j=0}^{n}a(m-i,n-j)\pi_{i+1,j}]\vspace{2mm}\\&+\bar{a}[\sum_{i=0}^{m}a(m-i,n)\pi_{i,0}+\sum_{i=0}^{m}\sum_{j=0}^{n}a(m-i,n-j)\pi_{i,j+1}],\,m,n\geq 0.
    \end{array}\label{be}
\end{equation}

In \cite{viss1,viss2}, the authors studied two-dimensional random walks with geometric jumps to the North-West, and used a method, closely related to the compensation approach, to solve directly the equilibrium equations. The joint equilibrium distribution was of product-form. Contrary to our case, in \cite{viss1,viss2}, the equilibrium equations contain a single finite sum, and this helped them upon substituting $\pi_{m,n}=\gamma^{m}\delta^{n}$, to reduce the resulting equation to an equation (called the $\Delta-$equation) that was independent of the state of the random walk (as was $K(\gamma,\delta)=0$ in subsection \ref{prel}). In our case, we have geometric jumps to the East, North and North-East, and thus, the interior equations when we substitute $\pi_{m,n}=\gamma^{m}\delta^{n}$ contain multiple finite sums.

We prove the same result as in \cite[subsection 2.5]{devosphd}, by solving directly the equilibrium equations using a similar framework as in \cite{viss1} (i.e., without complex analytic arguments, or partial balance principle). We summarize the steps below:
\begin{enumerate}
    \item We obtain the marginal distributions using the concept of the $\Delta-$equation; see \cite{viss1}. The marginal distributions will be of geometric form.
    \item The knowledge of the marginal distributions provides candidate solutions for the joint distribution. Substituting $\pi_{m,n}=\gamma^{m}\delta^{n}$, with $\gamma$ (resp. $\delta$) be the geometric term of the marginal distribution of the queue 1 (resp. queue 2), and using the concept of the $\Delta-$equation, we come up with a linear equation with respect to $\delta$ (resp. $\gamma$), which "seems" to depend on the state of the random walk. However, we show that for any state of the random walk, this $\delta$ is unique. Thus, the form of the arrival process, along with the non-work conserving service policy, provides a single $\delta$ (resp. $\gamma$), irrespective of the state of the random walk. 
\end{enumerate}

Let $\pi_{m}^{(1)}:=\sum_{n=0}^{\infty}\pi_{m,n}$, and note that $a_{1}(i):=\sum_{j=0}^{\infty}a(i,j)=\frac{1}{1+\lambda_{1}}(\frac{\lambda_{1}}{1+\lambda_{1}})^{i}$. Sum \eqref{be} for all $n\geq 0$, to obtain
\begin{equation}
    \begin{array}{c}
         \pi_{m}^{(1)}=a a_{1}(m)\pi_{0}^{(1)}+\sum_{i=0}^{m}a_{1}(m-i)[a\pi_{i+1}^{(1)}+\bar{a}\pi_{i}^{(1)}],\,m\geq 0.
    \end{array}\label{be1}
\end{equation}
Set in \eqref{be1}, $\pi_{m}^{(1)}=c\gamma^{m}$, $m\geq 0$, where $c$ a constant. Then, we arrive at
\begin{equation}
    \gamma^{m}=aa_{1}(m)+\sum_{i=0}^{m}a_{1}(m-i)[a\gamma^{i+1}+\bar{a}\gamma^{i}].\label{be2}
\end{equation}
Note that \eqref{be2} depends on the state at queue 1. Consider first \eqref{be2} for $m=l+1$. Then, multiply \eqref{be2} (for $m=l$) with $\gamma$. Now subtract the resulting equations to obtain the so called $\Delta$-equation, which in our case will give
\begin{displaymath}
\gamma=\frac{a_{1}(l+1)}{a(a_{1}(l)-a_{1}(l+1))}=\frac{\lambda_{1}}{a}.
\end{displaymath}
Note that under such an operation all the equilibrium equations \eqref{be1} are satisfied for such a $\gamma$. Thus, $\pi_{m}^{(1)}=c\gamma^{m}$, $m\geq 0$, so that $c=1-\gamma$, $\gamma=\frac{\lambda_{1}}{a}$. Similarly, we can show that $\pi_{n}^{(2)}=(1-\delta)\delta^{n}$, $n\geq 0$, $\delta=\frac{\lambda_{2}}{\bar{a}}$. This terminates step 1. 

Having in mind the form of the marginal distributions, we consider the following candidate solution for \eqref{be}: $\pi_{m,n}=d\gamma^{m}\delta^{n}$, with $\gamma=\lambda_{1}/a$. Substituting in \eqref{be} yields 
\begin{equation}
    \begin{array}{rl}         \gamma^{m}\delta^{n}=&a[\sum_{j=0}^{n}a(m,n-j)\delta^{j}+\sum_{i=0}^{m}\sum_{j=0}^{n}a(m-i,n-j)\gamma^{i+1}\delta^{j}]\vspace{2mm}\\&+\bar{a}[\sum_{i=0}^{m}a(m-i,n)\gamma^{i}+\sum_{i=0}^{m}\sum_{j=0}^{n}a(m-i,n-j)\gamma^{i}\delta^{j+1}],\,m,n\geq 0.
    \end{array}\label{be3}
\end{equation}
Contrary to \eqref{be2}, in \eqref{be3} we have a double finite sum, and for the given arrival process it seems impossible to derive an equation for obtaining $\delta$ that is independent of $m$ and $n$. We now apply a similar procedure as in step 1. Consider first \eqref{be3} for $n=k+1$, and then multiply \eqref{be3} (where $n=k$) with $\delta$. Now subtract the resulting equations to obtain after some algebra that for any $m\geq 0$,
\begin{equation}
    \delta\bar{a}\sum_{i=0}^{m}\gamma^{i}(a(m-i,k)-a(m-i,k+1))=aa(m,k+1)+(\bar{a}+\lambda_{1})\sum_{i=0}^{m}\gamma^{i}a(m-i,k+1).\label{be4}
\end{equation}
As before, the candidate solution with $\gamma=\lambda_{1}/a$ and $\delta$ the one that satisfies \eqref{be4}, satisfies \eqref{be3}. Consider \eqref{be4} for $m=l+1$, then multiply \eqref{be4} (where $m=l$) with $\gamma$, and subtract the resulting equations, to obtain the following $\Delta-$equation, which seems to be dependent of $l$ and $k$:
\begin{equation}
    \delta=\frac{a(l+1,k+1)(1+\lambda_{1})-\lambda_{1}a(l,k+1)}{\bar{a}(a(l+1,k)-a(l+1,k+1))}.\label{be5}
\end{equation}
Substituting $a(i,j)$ (as given in \eqref{vghb}) in \eqref{be5} yields after some algebra
\begin{displaymath}
    \delta=\frac{\lambda_{2}}{\bar{a}}\times\left(\frac{\frac{1+\lambda_{1}}{1+\lambda_{T}}\binom{l+k+2}{l+1}-\binom{l+k+1}{l}}{\binom{l+k+1}{l+1}-\frac{\lambda_{2}}{1+\lambda_{T}}\binom{l+k+2}{l+1}}\right). 
\end{displaymath}
It is easy to realize that the term in parentheses equals to 1, for any values of $l$, $k$. Indeed,
\begin{displaymath}
\begin{array}{rl}
    \frac{\frac{1+\lambda_{1}}{1+\lambda_{T}}\binom{l+k+2}{l+1}-\binom{l+k+1}{l}}{\binom{l+k+1}{l+1}-\frac{\lambda_{2}}{1+\lambda_{T}}\binom{l+k+2}{l+1}}=&1 \Leftrightarrow \vspace{2mm}\\
    \frac{1+\lambda_{1}}{1+\lambda_{T}}\binom{l+k+2}{l+1}-\binom{l+k+1}{l}=& \binom{l+k+1}{l+1}-\frac{\lambda_{2}}{1+\lambda_{T}}\binom{l+k+2}{l+1}\Leftrightarrow \vspace{2mm}\\
    \binom{l+k+2}{l+1}=&\binom{l+k+1}{l}+\binom{l+k+1}{l+1},
\end{array}
\end{displaymath}
which holds due to the Pascal's triangle.

Thus, for $\gamma=\lambda_{1}/a$ there is a unique $\delta=\lambda_{2}/\bar{a}$, so that $\pi_{m,n}=d\gamma^{m}\delta^{n}$, $m,n\geq 0$ satisfies \eqref{be3}. Normalization condition implies that $d=(1-\gamma)(1-\delta)$, and thus, $\pi_{m,n}=\pi_{1}^{(m)}\pi_{2}^{(n)}$, is the joint stationary distribution. 
\section{Discussion}\label{discussion}
Note that the characterization result in Theorem \ref{mainr} requires that from any point of the interior state space, we cannot have transitions to the North-East and South-West; i.e., Condition A. This is because, in case we allow such transitions, Condition B is no longer valid. We show in the following that by violating this requirement the resulting invariant measure cannot be written as a sum of three geometric terms; see Subsections \ref{pairedn}, \ref{ne}. In Subsection \ref{genn}, by completely violating Condition A, we briefly show how the conditions B, C, D, have to be adapted, so that the resulting random walk has a single product-form invariant measure.

Let us briefly comment on the importance of Conditions B, C, D, when our aim is to construct two-dimensional random walks with an invariant measure that can be obtained by a finite compensation procedure. We have to mention that the whole procedure can be divided in two parts:
\begin{enumerate}
    \item \textit{Part 1:} This part is common for all the two-dimensional random walks that are analyzed through the compensation approach: to obtain a formal solution to the interior, and the boundary equilibrium equations (i.e., the equations \eqref{int}-\eqref{h2}). In our case, although we allow transitions from the interior to the North and East, i.e., by violating a fundamental requirement in applying the compensation method \cite{adanaplprob}, we are able to apply the compensation method thanks to conditions A, B, C. More precisely, Condition B helps to show that the marginal distributions are of geometric form, and provides us with a strong intuition regarding the values of the initial terms that we need to start the compensation method, i.e., the $\rho_{1}$, or $\rho_{2}$. These initial terms are of high importance when we are about to apply the compensation method. It also helps in obtaining the formal solution (e.g., in deriving  the compensation coefficients) to \eqref{int}-\eqref{h2}. 
    More importantly, by using condition C, the compensation method is terminated in exactly three steps. Thus, by using conditions A, B, C, we are able to find a formal solution satisfying \eqref{int}-\eqref{h2}.
    \item \textit{Part 2:} The second part is not used in the standard compensation approach. In particular, in the standard compensation method \cite{adanaplprob}, we do not pay attention to the equilibrium equations \eqref{00}-\eqref{11} when we construct the stationary distribution as an infinite sum of product-form terms. In such a case, the solution most likely leads to a divergent series. Thus, the fundamental requirement of no transitions to the North, East, and North-East allows to find a convergent solution for states away from the origin. However, it is most likely that this solution diverges for states close to the origin of the state space. Note that this solution satisfies the inner and the boundary equilibrium equations except from the equations that refer to states close to the origin. In such a case, these equilibrium equations have to be solved numerically from the convergent solution that was previously derived by the (standard) compensation method. In our work, where a finite number of geometric terms is needed, there is no need to worry about the convergence of the sequences of the terms, as well as of the solution of the inner and the boundary equilibrium equations. However, the corresponding solution must also satisfies the remaining equilibrium equations at the states close to the origin. This is why it is of highest importance to show that the formal solution of \eqref{int}-\eqref{h2} (obtained in Part 1) also satisfies \eqref{00}-\eqref{11}. At that point, Condition D (as well as Condition B.2) is crucial. More precisely, condition D provides the proper relation among $q_{0,1}^{(h)}$, $q_{1,0}^{(v)}$ with the transition probabilities in the interior, so that \eqref{solgf1} to satisfy \eqref{01}, \eqref{10}. Condition B.2 relates the transition probabilities at the origin with those at the boundaries, and it is crucial to show that \eqref{solgf1} satisfies \eqref{11}.
\end{enumerate} 

We realize that in constructing two-dimensional random walks with an invariant measure written as a linear combination of a finite number of geometric terms, the locations of the intersections of $K$, $H$ and $V$ are crucial. Remind that, by starting from the intersection of $K$ with $H$, we end up to the intersection of $K$ with $V$, and vice-versa, we finally construct an invariant measure as a unique mixture of a finite number of geometric terms. In particular, there exists a \textit{pairwise-coupled} set \cite{chen} connecting the intersection of $K$ with $H$ to the intersection of $K$ with $V$, and vise versa; see also Remark \ref{remark}.

Several interesting tasks are open for future research. An interesting task is to investigate the case where we can have a finite number (but more than three) of geometric terms if we relax Condition C. Having in mind that Conditions A, B, D should be satisfied, we expect to have again two starting initial terms ($\rho_{1}$ or $\rho_{2}$). So, there are three cases:
\begin{enumerate}
\item the path starting from either initial couple, must end to the other one and is exactly the same when we traverse in the opposite direction (i.e., if we start from $(\rho_{1},\rho_{1}\rho_{2})$ we must end to $(\rho_{1}\rho_{2},\rho_{2})$ and vice-versa). This case would be ideal, and the queueing model in \cite{devos} belongs to that framework. In that case both paths (and as a result the generated couples) are identical.
\item There are different paths when we start from either initial couple, and the solution will be derived by the linear combination of the common couples. In that case, when we start from either initial couple, we must visit the other initial couple, i.e., it is important that both initial couples belong to the intersection of the couples generated by either path. Our intuition is that, to apply the compensation method, we must generate the same couples when we start from either initial couple. So we expect that when Conditions A, B, D  are satisfied, compensation method will result in the same path.
\item Finally, there is also an option to have two different solutions (different paths with finite number of product-forms), generated by each starting couple, and thus, the sum of these solutions to provide the invariant measure.
\end{enumerate}
In this direction, it would be of great interest to further investigate the connection of our work with the one in \cite{diek}. We remind that in \cite{diek}, the authors provide conditions under which the stationary distribution of a reflected Brownian motion with specific constraints on the covariance and reflection matrices can be written as a finite sum of exponential terms. We believe that further understanding of the machinery of the transitions in \cite{diek} 
will be crucial in understanding how to construct discrete state space random walks with invariant measures that are represented as a sum of more than three (but finitely many) product-form terms.

Another interesting task is to see whether the finite compensation procedure applies when condition B is partly satisfied, so that we can have only one initial term, e.g., only $\gamma_{0}$, and an initial solution satisfying the interior and the horizontal boundary.




In the future, we plan to extend the results obtained in Section \ref{general}, to multidimensional random walks. The results derived in Subsections \ref{special}, \ref{geoprod}, where a single product-form is sufficient, are quite promising. The form of the marginal distribution seems to be crucial in proving whether the compensation procedure can be applied. Another direction is to consider Markov-modulated two-dimensional random walks, by assuming that the arrivals form a discrete-time PH distribution, e.g., when the interarrival time follows a negative binomial distribution (at least for the model in Subsection \ref{prod}). Our experience \cite{adanerl}, shows that the compensation method can be applied in a similar setting. A first step in this direction might be to consider a modulation that allows a completely tractable analysis, e.g., when a change in the phase does not immediately trigger a transition of the level process but changes its dynamics (indirect interaction).

Another interesting direction is to consider bounded transitions (greater than one step, but always bounded) to the South-East and North-West. Our motivation stems from the fact that compensation approach is shown to be not affected by such kind of transitions \cite{adanerl,sax}. Such an assumption will result in a random walk with a more complicated boundary behaviour.

In the following, we discuss three queueing examples (see Subsections \ref{wcon}, \ref{pairedn}, \ref{ne}), related to those in Section \ref{motivation}, for which conditions A, B, C, D (all or some of them) do not hold. We will see that their joint stationary queue-length distribution cannot be derived by using the finite compensation procedure. We end this section, by considering a two-dimensional random walk for which condition A is relaxed, i.e., $q_{1,1}>0$, $q_{-1,-1}>0$. We show (see Subsection \ref{genn}) how we can adapt the rest of the conditions so that the resulting random walk has a single product-form invariant measure.     

\subsection{The model under the work-conserving policy}\label{wcon}
Consider now the companion model of the one discussed in Subsection \ref{prel}. It is easy to see that the model at hand under the work conserving policy is a special case of the model considered in \cite{walra}, when $A(z_{1},z_{2})=A_{1}(z_{1})A_{2}(z_{2})$, with $A_{k}(z_{k})=\bar{\lambda}_{k}+\lambda_{k}z_{k}$, $k=1,2$, and $\mu_{1}=\mu_{2}=1$. 
Note that the stability condition of the work conserving case ($\lambda_{1}+\lambda_{2}<1$) is different from the one in the non-work conserving case considered in subsection \ref{prel}. In the work-conserving case, if one queue is empty, then the server will serve with probability 1, the non-empty queue.

The stationary behaviour of the work-conserving model can only be treated by using the theory of boundary value problems \cite{cohbox,fay1}, or by applying approximation methods as in \cite{walra}, and definitely, its invariant measure cannot be written as a finite sum of product-form terms; see also Fig. \ref{pw1} for the curves $K(\gamma,\delta) = 0$, $H(\gamma,\delta) = 0$, $V(\gamma,\delta) = 0$.  

Note that under the work-conserving policy, the one-step transition probabilities from any point at the interior and from the origin, coincide with those under the non-work conserving policy given in Section \ref{prel}. However, this is not the case for the transition probabilities at the boundaries. Moreover, due to the work-conserving policy, from any point on the horizontal (resp. vertical) boundary is not allowed to have one-step displacements to the East (resp. North). Specifically, the one step transition probabilities from the boundaries are now given by:
\begin{enumerate}
\item For $m>0,n=0$,
\begin{displaymath}
q_{-1,0}^{(h)}=\bar{\lambda}_{1}\bar{\lambda}_{2},\,q_{0,0}^{(h)}=\lambda_{1}\bar{\lambda}_{2},\,q_{-1,1}^{(h)}=\bar{\lambda}_{1}\lambda_{2},\,q_{0,1}^{(h)}=\lambda_{1}\lambda_{2}.
\end{displaymath}
\item For $m=0,n>0$,
\begin{displaymath}
q_{0,-1}^{(v)}=q_{-1,0}^{(h)},\,q_{1,-1}^{(h)}=q_{0,0}^{(h)},\,q_{0,0}^{(v)}=q_{-1,1}^{(h)},\,q_{1,0}^{(v)}=q_{0,1}^{(h)}.
\end{displaymath}
\end{enumerate}
\begin{figure}[ht!]
\centering
\includegraphics[scale=0.5]{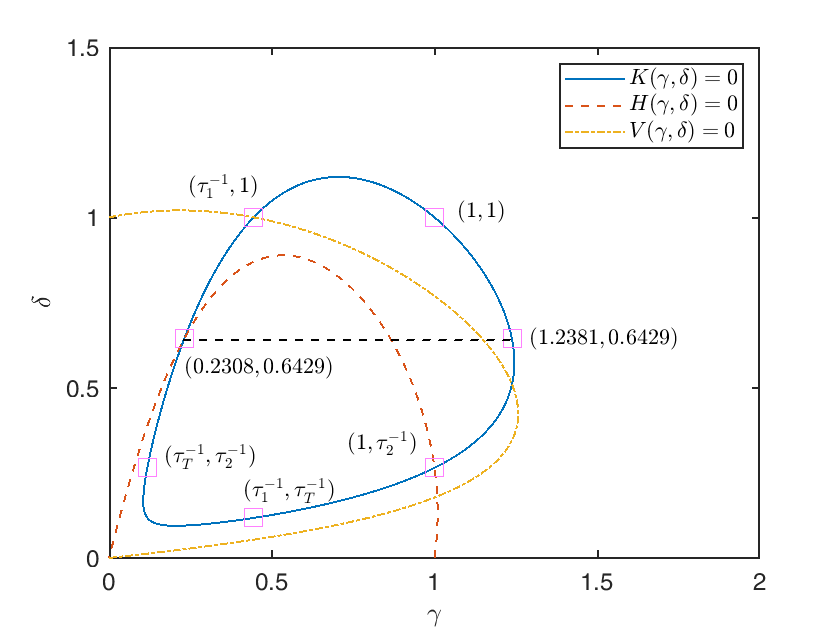}
\caption{The curves $K(\gamma,\delta)=0$, $H(\gamma,\delta)=0$, $V(\gamma,\delta)=0$, for $\lambda_{1}=0.4$, $\lambda_{2}=0.15$, $a=0.6$.}\label{pw1}
\end{figure}

It is easy to realize that \textbf{none} of the conditions B, and D are satisfied. In particular:
\begin{enumerate}
\item Regarding Condition B.1:
\begin{itemize}
\item $q_{0,1}> q_{1,1}^{(v)}=0$, $q_{-1,1}> q_{0,1}^{(v)}=0$, $q_{0,-1}\neq q_{0,-1}^{(v)}$, $q_{1,-1}\neq q_{1,-1}^{(v)}$.
\item $q_{1,0}> q_{1,1}^{(h)}=0$, $q_{1,-1}> q_{1,0}^{(h)}=0$, $
q_{-1,0}\neq q_{-1,0}^{(h)}$, $q_{-1,1}\neq q_{-1,1}^{(h)}.$
\end{itemize}
\item Regarding Condition B.2:
\begin{displaymath}
\begin{array}{rl}
q_{0,0}^{(0)}+q_{0,0}=&q_{0,0}^{(h)}+q_{0,0}^{(v)}\Leftrightarrow\bar{\lambda}_{1}\bar{\lambda}_{2}+\alpha\lambda_{1}\bar{\lambda}_{2}+\bar{a}\bar{\lambda}_{1}\lambda_{2}= \lambda_{1}\bar{\lambda}_{2}+\bar{\lambda}_{1}\lambda_{2}\Leftrightarrow\vspace{2mm}\\
\bar{\lambda}_{1}\bar{\lambda}_{2}=&\lambda_{1}\bar{\lambda}_{2}\bar{a}+\bar{\lambda}_{1}\lambda_{2}a\Leftrightarrow\vspace{2mm}\\
\frac{\lambda_{1}}{\bar{\lambda}_{1}}\bar{a}+\frac{\lambda_{2}}{\bar{\lambda}_{2}}a=&1.
\end{array}
\end{displaymath}
Moreover,
\begin{displaymath}
\begin{array}{rl}
q_{0,1}^{(0)}+q_{0,1}=&q_{0,1}^{(h)}+q_{0,1}^{(v)}\Leftrightarrow \bar{\lambda}_{1}\lambda_{2}+a\lambda_{1}\lambda_{2}=\lambda_{1}\lambda_{2}\Leftrightarrow
\bar{a}=\frac{\bar{\lambda}_{1}}{\lambda_{1}},\vspace{2mm}\\
q_{1,0}^{(0)}+q_{1,0}=&q_{1,0}^{(h)}+q_{1,0}^{(v)}\Leftrightarrow \lambda_{1}\bar{\lambda}_{2}+\bar{a}\lambda_{1}\lambda_{2}= \lambda_{1}\lambda_{2}\Leftrightarrow\vspace{2mm}
a=\frac{\bar{\lambda}_{2}}{\lambda_{2}},\vspace{2mm}\\
q_{1,1}^{(0)}=\lambda_{1}\lambda_{2}>&q_{1,1}^{(h)}+q_{1,1}^{(v)}=0.
\end{array}
\end{displaymath}
Note that the first three requirements cannot be satisfied simultaneously, since using the second and the third one, the first one cannot be satisfied. Moreover, the forth one obviously cannot be satisfied.
\item Regarding Condition D:
Note that
\begin{displaymath}
\begin{array}{rl}
     q_{0,1}^{(h)}=\lambda_{1}\lambda_{2}\neq &q_{0,1}+\frac{q_{-1,1}q_{1,0}}{q_{0,1}}=a\lambda_{1}\lambda_{2}+\bar{a}\bar{\lambda}_{1}\lambda_{2},  \\
    q_{1,0}^{(v)}=\lambda_{1}\lambda_{2}\neq &q_{1,0}+\frac{q_{1,-1}q_{0,1}}{q_{1,0}}=\bar{a}\lambda_{1}\lambda_{2}+a\bar{\lambda}_{2}\lambda_{1}. 
\end{array}
\end{displaymath}
\end{enumerate}
%
\subsection{The model with an option for paired service: Allowing transitions to the South-West}\label{pairedn}
We consider a modification of the model in \cite{devos}, which offers an option for paired services, i.e., at the beginning of a slot with probability $a_{0}$, the server chooses to serve a job from each queue (paired service), and with probability $a_{k}$, $k=1,2,$ chooses to serve a job from queue $k$. The rest is as in the main model in Subsection \ref{prel}. Note that under the option of paired services, we allow transitions from an interior point of the state space to the South-West. 

The one-step transition probabilities are:
\begin{itemize}
\item $m,n>0$
\begin{displaymath}
\begin{array}{c}
q_{0,0}=a_{0}\lambda_{1}\lambda_{2}+a_{1}\lambda_{1}\bar{\lambda}_{2}+a_{2}\lambda_{2}\bar{\lambda}_{1},\,q_{-1,0}=a_{0}\bar{\lambda}_{1}\lambda_{2}+a_{1}\bar{\lambda}_{1}\bar{\lambda}_{2},\,q_{0,-1}=a_{0}\bar{\lambda}_{2}\lambda_{1}+a_{2}\bar{\lambda}_{1}\bar{\lambda}_{2},\,q_{-1,-1}=a_{0}\bar{\lambda}_{1}\bar{\lambda}_{2},\\
q_{0,1}=a_{1}\lambda_{1}\lambda_{2},\,q_{-1,1}=a_{1}\bar{\lambda}_{1}\lambda_{2},\,q_{1,0}=a_{2}\lambda_{1}\lambda_{2},\,q_{1,-1}=a_{2}\lambda_{1}\bar{\lambda}_{2}.
\end{array}
\end{displaymath}
\item $m>0,n=0$
\begin{displaymath}
\begin{array}{c}
q_{0,0}^{(h)}=(a_{0}+a_{1})\lambda_{1}\bar{\lambda}_{2}+a_{2}\bar{\lambda}_{2}\bar{\lambda}_{1},\,q_{-1,0}^{(h)}=(a_{0}+a_{1})\bar{\lambda}_{1}\bar{\lambda}_{2},\\
q_{0,1}^{(h)}=(a_{0}+a_{1})\lambda_{1}\lambda_{2}+a_{2}\bar{\lambda}_{1}\lambda_{2},\,q_{-1,1}^{(h)}=(a_{0}+a_{1})\bar{\lambda}_{1}\lambda_{2},\,q_{1,0}^{(h)}=a_{2}\lambda_{1}\bar{\lambda}_{2},\,q_{1,1}^{(h)}=a_{2}\lambda_{1}\lambda_{2}.
\end{array}
\end{displaymath}
\item $m=0,n>0$
\begin{displaymath}
\begin{array}{c}
q_{0,0}^{(v)}=(a_{0}+a_{2})\lambda_{2}\bar{\lambda}_{1}+a_{1}\bar{\lambda}_{2}\bar{\lambda}_{1},\,q_{0,-1}^{(v)}=(a_{0}+a_{2})\bar{\lambda}_{1}\bar{\lambda}_{2},\\
q_{1,0}^{(v)}=(a_{0}+a_{2})\lambda_{1}\lambda_{2}+a_{1}\bar{\lambda}_{2}\lambda_{1},\,q_{1,-1}^{(v)}=(a_{0}+a_{2})\bar{\lambda}_{2}\lambda_{1},\,q_{0,1}^{(v)}=a_{1}\lambda_{2}\bar{\lambda}_{1},\,q_{1,1}^{(v)}=a_{1}\lambda_{1}\lambda_{2}.
\end{array}
\end{displaymath}
\item $m=n=0$, as in the original model.
\end{itemize}

Following \cite{fayo}, the stability condition is $\lambda_{1}<a_{0}+a_{1}$, $\lambda_{2}<a_{0}+a_{2}$, while the balance principle ensures that the marginal probabilities are of geometric form, i.e., $\pi_{m}^{(1)}=C_{1}(\tau_{1}^{-1})^{m}$, $m\geq 1$, $\pi_{n}^{(2)}=C_{2}(\tau_{2}^{-1})^{n}$, $n\geq 1$, for
\begin{displaymath}
    \tau_{1}^{-1}:=\frac{a_{2}\lambda_{1}}{(a_{0}+a_{1})\bar{\lambda}_{1}},\,\, \tau_{2}^{-1}:=\frac{a_{1}\lambda_{2}}{(a_{0}+a_{2})\bar{\lambda}_{2}}.
\end{displaymath}


Note that in general, the model with two queues and paired services (a simpler version of ours without transitions from an interior point to the North and East) is known to be analysed by Cohen \cite{coh1} with the uniformization technique (for a nearest-neighbour random walk model), and in \cite{blanc} with the aid of a Riemann boundary value problem, for the case of arbitrarily distributed services times. So its invariant measure cannot be written as a sum of finite geometric terms. It is readily seen that this random walk does not satisfy the Conditions A, B, C, D. Condition A is not satisfied, since $q_{-1,-1}>0$. Moreover, Condition B.1 is not satisfied, since
\begin{itemize}
\item although $q_{1,1}^{(h)}=q_{1,0}$, $q_{1,0}^{(h)}=q_{1,-1}$, we have $q_{-1,0}^{(h)}\neq q_{-1,0}$, $q_{-1,1}^{(h)}\neq q_{-1,1}$,
\item although $q_{1,1}^{(v)}=q_{0,1}$, $q_{0,1}^{(v)}=q_{-1,1}$, we have $q_{0,-1}^{(v)}\neq q_{0,-1}$, $q_{1,-1}^{(v)}\neq q_{1,-1}$.
\end{itemize}
Condition B.2 is not satisfied, since
\begin{itemize}
\item 
$q_{0,0}^{(0)}+q_{0,0}=q_{0,0}^{(h)}+q_{0,0}^{(v)} \, \Leftrightarrow \, \bar{\lambda}_{1}\bar{\lambda}_{2}+\lambda_{1}\lambda_{2}=\bar{\lambda}_{1}\lambda_{2}+\lambda_{1}\bar{\lambda}_{2}$,
\item $q_{0,1}^{(0)}+q_{0,1}=q_{0,1}^{(h)}+q_{0,1}^{(v)}\Leftrightarrow \lambda_{1}=\frac{1}{2}$,
\item $q_{1,0}^{(0)}+q_{1,0}=q_{1,0}^{(h)}+q_{1,0}^{(v)}\Leftrightarrow \lambda_{2}=\frac{1}{2}$.
Note that for $\lambda_{1}=\lambda_{2}=1/2$, the above conditions are satisfied. 
\item However, $q_{1,1}^{(0)}=q_{1,1}^{(h)}+q_{1,1}^{(v)}\Leftrightarrow 1=a_{1}+a_{2}$, which is not true. Thus, condition B.2 is not satisfied.
\end{itemize}
\begin{figure}[ht!]
\centering
\includegraphics[scale=0.5]{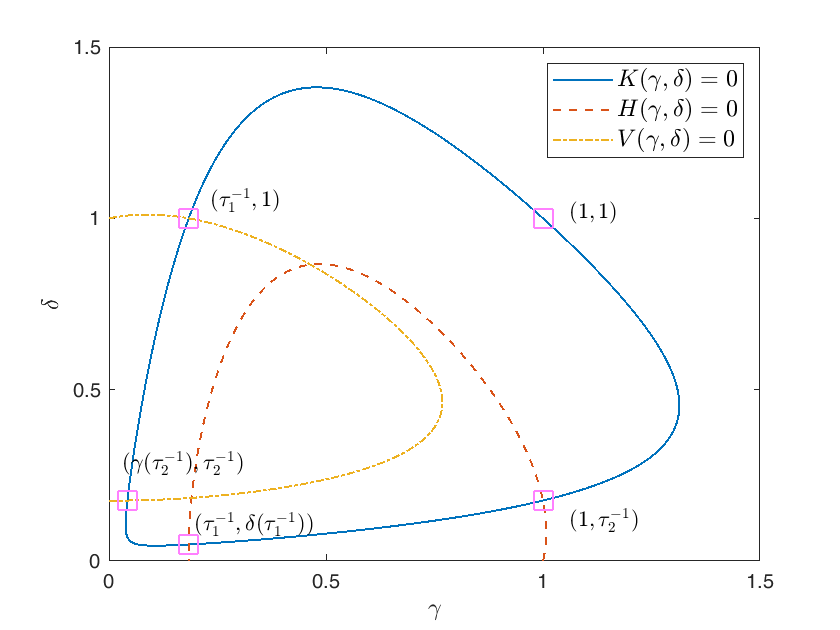}
\caption{The curves $K(\gamma,\delta)=0$, $H(\gamma,\delta)=0$, $V(\gamma,\delta)=0$, for $\lambda_{1}=0.3$, $\lambda_{2}=0.15$, $a_{0}=0.2$, $a_{1}=0.5$, $a_{2}=0.3$, for the model with an option for paired service, where $\delta(\tau_{1}^{-1})=\frac{q_{-1,1}+q_{0,1}\tau_{1}}{q_{-1,-1}+\tau_{1}q_{0,-1}+q_{1,-1}\tau_{1}^{2}}$, $\gamma(\tau_{2}^{-1})=\frac{q_{1,-1}+q_{1,0}\tau_{2}}{q_{-1,-1}+\tau_{2}q_{-1,0}+q_{-1,1}\tau_{2}^{2}}$.}\label{pon1}
\end{figure}
\subsection{The model with a false service initiation: Allowing transitions to the North-East}\label{ne}
We consider the following modification of the model in Subsection \ref{prel}:
\begin{itemize}
\item When the server chooses the next queue to serve,  the switch to this  queue is successful with probability, say $b$, whereas with probability $\bar{b}=1-b$, the server fails to switch to the preferred queue. Thus, there is the possibility of no service in a slot.
\item Note that under the assumption of false service initiations, we allow from the interior transitions to the North-East. 
\item Note that when $b=1$, we have the model in \cite{devos}.
\end{itemize}
The transition probabilities are now
\begin{enumerate}
\item $m,n>0$,
\begin{displaymath}
\begin{array}{c}
q_{0,1}=ab\lambda_{1}\lambda_{2}+\bar{b}\bar{\lambda}_{1}\lambda_{2},\,q_{-1,1}=ab\bar{\lambda}_{1}\lambda_{2},\,q_{0,0}=ab\lambda_{1}\bar{\lambda}_{2}+\bar{a}b\bar{\lambda}_{1}\lambda_{2}+\bar{b}\bar{\lambda}_{1}\bar{\lambda}_{2},\\q_{1,0}=\bar{a}b\lambda_{2}\lambda_{1}+\bar{b}\lambda_{1}\bar{\lambda}_{2},\,q_{1,-1}=\bar{a}b\lambda_{1}\bar{\lambda}_{2},\,q_{-1,0}=ab\bar{\lambda}_{1}\bar{\lambda}_{2},\,\,q_{0,-1}=\bar{a}b\bar{\lambda}_{1}\bar{\lambda}_{2},\,q_{1,1}=\bar{b}\lambda_{1}\lambda_{2},
\end{array}
\end{displaymath}
\item $m>0,n=0$,
\begin{displaymath}
\begin{array}{c}
q_{0,1}^{(h)}=ab\lambda_{1}\lambda_{2}+(\bar{a}b+a\bar{b}+\bar{a}\bar{b})\bar{\lambda}_{1}\lambda_{2},\,q_{-1,1}^{(h)}=ab\bar{\lambda}_{1}\lambda_{2},\,q_{0,0}^{(h)}=ab\lambda_{1}\bar{\lambda}_{2}+(\bar{a}b+a\bar{b}+\bar{a}\bar{b})\bar{\lambda}_{1}\bar{\lambda}_{2},\\q_{1,0}^{(h)}=(\bar{a}b+a\bar{b}+\bar{a}\bar{b})\bar{\lambda}_{2}\lambda_{1},\,q_{1,1}^{(h)}=(\bar{a}b+a\bar{b}+\bar{a}\bar{b})\lambda_{1}\lambda_{2},\,q_{-1,0}^{(h)}=ab\bar{\lambda}_{1}\bar{\lambda}_{2},
\end{array}
\end{displaymath}
\item $m=0,n>0$,
\begin{displaymath}
\begin{array}{c}
q_{0,1}^{(v)}=(ab+a\bar{b}+\bar{a}\bar{b})\bar{\lambda}_{1}\lambda_{2},\,q_{1,-1}^{(v)}=\bar{a}b\lambda_{1}\bar{\lambda}_{2},\,q_{0,0}^{(v)}=(ab+a\bar{b}+\bar{a}\bar{b})\bar{\lambda}_{1}\bar{\lambda}_{2}+\bar{a}b\bar{\lambda}_{1}\lambda_{2},\\q_{1,0}^{(v)}=\bar{a}b\lambda_{1}\lambda_{2}+(ab+a\bar{b}+\bar{a}\bar{b})\lambda_{1}\bar{\lambda}_{2},\,q_{1,1}^{(v)}=(ab+a\bar{b}+\bar{a}\bar{b})\lambda_{1}\lambda_{2},\,q_{0,-1}^{(v)}=\bar{a}b\bar{\lambda}_{1}\bar{\lambda}_{2},
\end{array}
\end{displaymath}
\item $m=n=0$,
\begin{displaymath}
\begin{array}{c}
q_{0,0}^{(0)}=\bar{a}(b+\bar{b})\bar{\lambda}_{1}\bar{\lambda}_{2}+a(b+\bar{b})\bar{\lambda}_{1}\bar{\lambda}_{2}=\bar{\lambda}_{1}\bar{\lambda}_{2},\,q_{1,0}^{(0)}=\bar{a}(b+\bar{b})\lambda_{1}\bar{\lambda}_{2}+a(b+\bar{b})\lambda_{1}\bar{\lambda}_{2}=\lambda_{1}\bar{\lambda}_{2},\\
q_{0,1}^{(0)}=\bar{a}(b+\bar{b})\lambda_{2}\bar{\lambda}_{1}+a(b+\bar{b})\lambda_{2}\bar{\lambda}_{1}=\lambda_{2}\bar{\lambda}_{1},\,q_{1,1}^{(0)}=\bar{a}(b+\bar{b})\lambda_{2}\lambda_{1}+a(b+\bar{b})\lambda_{2}\lambda_{1}=\lambda_{2}\lambda_{1}.
\end{array}
\end{displaymath}
\end{enumerate}
Following \cite{fayo}, the stability condition is $\lambda_{1}(\bar{a}+a\bar{b})<ab\bar{\lambda}_{1}$, $\lambda_{2}(a+\bar{a}\bar{b})<\bar{a}b\lambda_{2}$.
%

The balance principle enables us to show that $\pi_{m}^{(1)}=(\tau_{1}^{-1})^{m}C_{1}$, $m\geq 1$, $\pi_{n}^{(2)}=(\tau_{2}^{-1})^{n}C_{2}$, $n\geq 1$, where 
\begin{displaymath}
    \tau_{1}^{-1}=\frac{\lambda_{1}(\bar{a}+a\bar{b})}{ab\bar{\lambda}_{1}},\,\, \tau_{2}^{-1}=\frac{\lambda_{2}(a+\bar{a}\bar{b})}{\bar{a}b\lambda_{2}}.
\end{displaymath}
\begin{figure}[ht!]
\centering
\includegraphics[scale=0.5]{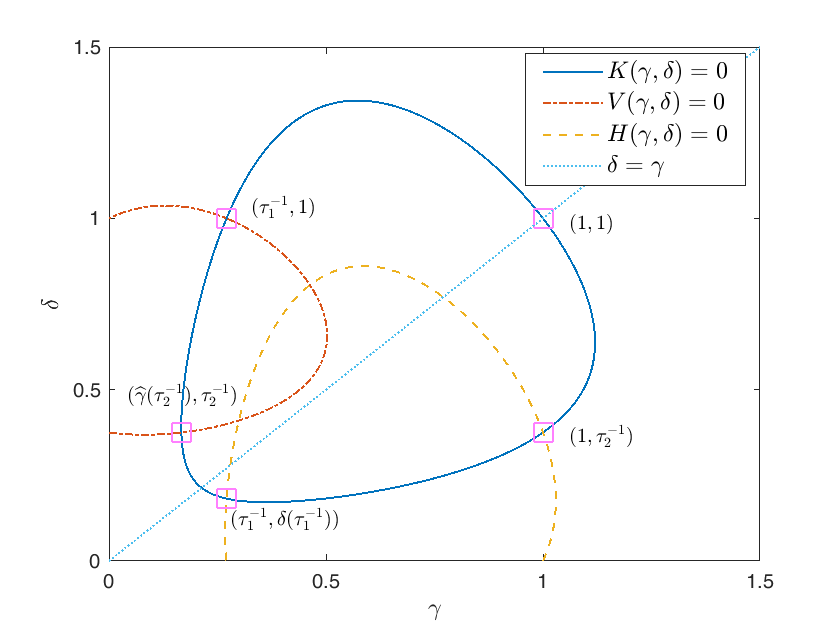}
\caption{The curves $K(\gamma,\delta)=0$, $H(\gamma,\delta)=0$, $V(\gamma,\delta)=0$, for $\lambda_{1}=0.2$, $\lambda_{2}=0.15$, $a=0.6$, $b=0.8$, for the model with a false service initiation, where $\delta(\tau_{1}^{-1})=\frac{q_{-1,1}+q_{0,1}\tau_{1}+q_{1,1}\tau_{1}^{2}}{\tau_{1}(q_{0,-1}+q_{1,-1}\tau_{1})}$, $\widehat{\gamma}(\tau_{2}^{-1})=\frac{q_{1,-1}+q_{1,0}\tau_{2}+q_{1,1}\tau_{2}^{2}}{\tau_{2}(q_{-1,0}+q_{-1,1}\tau_{2})}$.}\label{pv1}
\end{figure}


We  further note that this random walk does not satisfy  Conditions A, B, C, D:
\begin{itemize}
\item Condition A is not satisfied, since $q_{1,1}>0$.
\item Condition B.1 is not satisfied, since
\begin{enumerate}
\item from each point on the horizontal boundary, $q_{1,0}^{(h)}+q_{1,1}^{(h)}=q_{1,0}+q_{1,1}+q_{1,-1}=\lambda_{1}(\bar{a}+a\bar{b})$ (as in \cite{devos}), \textbf{but} now $q_{1,1}^{(h)}\neq q_{1,0}$, $q_{1,0}^{(h)}\neq q_{1,-1}$ (although $q_{-1,0}=q_{-1,0}^{(h)}$, $q_{-1,1}=q_{-1,1}^{(h)}$),
\item similarly, from each point on the vertical boundary, $q_{0,1}^{(v)}+q_{1,1}^{(v)}=q_{0,1}+q_{1,1}+q_{-1,1}=\lambda_{2}(a+\bar{a}\bar{b})$  (as in \cite{devos}), \textbf{but} now $q_{1,1}^{(v)}\neq q_{0,1}$, $q_{0,1}^{(v)}\neq q_{-1,1}$ (although $q_{0,-1}=q_{0,-1}^{(v)}$, $q_{1,-1}=q_{1,-1}^{(v)}$).
\end{enumerate}
\item Moreover, Condition B.2 is partly satisfied. Indeed, note that $q_{0,1}^{(0)}+q_{0,1}=q_{0,1}^{(h)}+q_{0,1}^{(v)}$, $q_{1,0}^{(0)}+q_{1,0}=q_{1,0}^{(h)}+q_{1,0}^{(v)}$, $q_{0,0}^{(0)}+q_{0,0}=q_{0,0}^{(h)}+q_{0,0}^{(v)}$, \textbf{but} $q_{1,1}^{(0)}\neq q_{1,1}^{(h)}+q_{1,1}^{(v)}$ (actually $q_{1,1}^{(0)}+q_{1,1}=q_{1,1}^{(h)}+q_{1,1}^{(v)}$).
\item Condition C is also not satisfied, e.g., $q_{1,-1}q_{-1,1}=ab\lambda_{1}\lambda_{2}\bar{a}b\bar{\lambda}_{1}\bar{\lambda}_{2}<(ab\lambda_{1}\lambda_{2}+\bar{b}\lambda_{1}\lambda_{2})\bar{a}b\bar{\lambda}_{1}\bar{\lambda}_{2}=q_{0,1}q_{0,-1}$.
\end{itemize}

So it is expected that the finite compensation approach does not work: Starting with the usual steps, the compensation procedure will fail. In particular, starting with the initial solution,
\begin{equation}
x_{1}(m,n)=\left\{\begin{array}{ll}
c_{0}\gamma_{0}^{m}\delta_{0}^{n},&m,n>0,\\
e_{0}\gamma_{0}^{m},&m>0,n=0.
\end{array}\right.
\label{inif}
\end{equation}
with $(\gamma_{0},\delta_{0})=(\tau_{1}^{-1},\delta(\tau_{1}^{-1}))$, and substituting \eqref{inif} in the horizontal boundary equations, results in two equations for $e_{0}$:
\begin{displaymath}
\begin{array}{rl}
e_{0}=&c_{0}\delta_{0}w_{1}(\gamma_{0}),\vspace{2mm}\\
e_{0}=&c_{0}\widehat{w}_{2}(\gamma_{0}),
\end{array}
\end{displaymath}
where $w_{1}(\gamma)$ is as in \eqref{mko}, and $\widehat{w}_{2}(\gamma)=\frac{q_{0,1}\gamma+q_{-1,1}\gamma^{2}+q_{1,1}}{\gamma(q_{0,1}^{(h)}+q_{-1,1}^{(h)}\gamma)+q_{1,1}^{(h)}}$. Numerical examples show that $\delta_{0}w_{1}(\gamma_{0})\neq \widehat{w}_{2}(\gamma_{0})$, so we cannot find a single $e_{0}$ satisfying the horizontal boundary equations.
Our intuition indicates that the violation of Condition A (i.e., $q_{1,1}>0=q_{-1,-1}$) is a major cause for not having a solution as the one given in Subsection \ref{prel}.  

\subsection{Removing Condition A: $q_{1,1}>0$, and $q_{-1,-1}>0$}\label{genn}
In Subsections \ref{pairedn} (i.e., $q_{-1,-1}>0$, $q_{1,1}=0$), \ref{ne} (i.e., $q_{-1,-1}=0$, $q_{1,1}>0$), we considered two queueing models, resulting in random walks where condition A is  relaxed. We now assume that Condition A is no longer valid at all, i.e., $q_{1,1}>0$, and $q_{-1,-1}>0$. We will apply the same steps as in Section \ref{general}, and focus on constructing a a single product-form invariant measure by applying the finite compensation procedure. Further development will be required for the case with more than one product-form term, but this attempt is postponed as future work. We summarize the following the basic steps.
\begin{itemize}
    \item As a first step, we have to ensure that the marginal distributions are geometric. In order to do so, we update Condition B (see Section \ref{general}) as follows:
    \begin{enumerate}
        \item Condition B.1: \begin{displaymath}
        \begin{array}{rlcrl}
             q_{1,1}^{(h)}=&q_{1,0}+q_{1,1},&&q_{1,1}^{(v)}=&q_{0,1}+q_{1,1},  \\
             q_{1,0}^{(h)}=&q_{1,-1},&&q_{0,1}^{(v)}=&q_{-1,1},\\
             q_{-1,1}^{(h)}=&q_{-1,1},&&q_{1,-1}^{(v)}=&q_{1,-1},\\
             q_{-1,0}^{(h)}=&q_{-1,0}+q_{-1,-1},&&q_{0,-1}^{(v)}=&q_{0,-1}+q_{-1,-1}.
        \end{array}
        \end{displaymath} 
        \item Condition B.2:\begin{displaymath}
        \begin{array}{rl}
             q_{1,0}^{(0)}+q_{1,0}=&q_{1,0}^{(h)}+q_{1,0}^{(v)},  \\
             q_{0,1}^{(0)}+q_{0,1}=&q_{0,1}^{(h)}+q_{0,1}^{(v)}, \\
             q_{1,1}^{(0)}+q_{1,1}=&q_{1,1}^{(h)}+q_{1,1}^{(v)}.
        \end{array}
        \end{displaymath}
    \end{enumerate}
    Under the updated Condition B, we have
    \begin{displaymath}
    \begin{array}{rl}
         q_{1,0}^{(0)}+q_{1,1}^{(0)}=&q_{1,-1}+q_{0,1}+q_{1,1}+q_{1,0}^{(v)},  \\
         q_{0,1}^{(0)}+q_{1,1}^{(0)}=&q_{-1,1}+q_{1,0}+q_{1,1}+q_{0,1}^{(h)}.
    \end{array}
    \end{displaymath}
    Thus, under condition B, a simple balance argument ensures that the marginal distributions are geometric:
    \begin{displaymath}
    \begin{array}{rl}
         \pi_{m}^{(1)}=&C_{1}\rho_{1}^{m},\,m\geq 1,  \\
        \pi_{n}^{(2)}=&C_{2}\rho_{2}^{n},\,n\geq 1, 
    \end{array}
    \end{displaymath}
    where  $\rho_{1}:=\frac{q_{1,0}+q_{1,-1}+q_{1,1}}{q_{-1,0}+q_{-1,1}+q_{-1,-1}}<1$, $\rho_{2}:=\frac{q_{0,1}+q_{-1,1}+q_{1,1}}{q_{0,-1}+q_{1,-1}+q_{-1,-1}}<1$. 
    \item Following the lines of Lemma \ref{lem1}, we can show that there exists a single product-form $\gamma^{m}\delta^{n}$, $0<|\gamma|,|\delta|<1$ that satisfies the updated $K(\gamma,\delta)=0$, $H(\gamma,\delta)=0$ (i.e., those with $q_{1,1}>0$, $q_{-1,-1}>0$). In particular, $\gamma_{0}=\rho_{1}$, and $\delta_{0}:=\frac{\gamma_{0}^{2}q_{-1,1}+\gamma_{0} q_{0,1}+q_{1,1}}{\gamma_{0}^{2}q_{-1,-1}+\gamma_{0} q_{0,-1}+q_{1,-1}}$. 
    \item Following the lines in Proposition \ref{prop1}, an initial solution satisfying the balance equations in the interior and the horizontal boundary is:
    \begin{equation}
    x(m,n)=\left\{\begin{array}{ll}
         c_{0}\gamma_{0}^{m}\delta_{0}^{n},&m,n>0,  \\
        e_{0}\gamma_{0}^{m}, &m>0,n=0,
    \end{array}\right.\label{ella}
    \end{equation}
    where now
    \begin{displaymath}\begin{array}{c}
    e_{0}=c_{0}\frac{\gamma_{0}^{2}q_{-1,1}+\gamma_{0} q_{0,1}+q_{1,1}}{\gamma_{0}^{2}q_{-1,1}^{(h)}+\gamma_{0} q_{0,1}^{(h)}+q_{1,1}^{(h)}}=c_{0}\delta_{0}\frac{\gamma_{0}^{2}q_{-1,-1}+\gamma_{0} q_{0,-1}+q_{1,-1}}{\gamma_{0}(1-q_{0,0}^{(h)})-q_{1,0}^{(h)}-\gamma_{0}^{2}q_{-1,0}^{(h)}},
    \end{array} \end{displaymath}
    thanks to Condition B and the definition of $\delta_{0}$ (as in Section \ref{general}).
    \item Since our aim is to construct a random walk with a single product-form invariant measure, we have to set a stopping criterion for the compensation approach  (i.e., a new Condition C).
    This stopping criterion is derived by requiring $\delta_{0}=\rho_{2}$. In such a case, setting $\delta=\rho_{2}$, the kernel equation $K(\gamma,\delta)=0$ gives  $\gamma=\rho_{1}$ or $\gamma=1$. So, starting with $\gamma_{0}=\rho_{1}$, we ensure that $\delta_{0}=\rho_{2}$, when the following condition among the transition probabilities in the interior is imposed (i.e., the new Condition C):
    \begin{displaymath}
    \begin{array}{c}
         q_{-1,1}q_{1,-1}(1-q_{0,0})+q_{-1,1}q_{1,0}q_{0,-1}+q_{1,-1}q_{-1,0}q_{0,1}\vspace{2mm}\\=q_{-1,-1}q_{1,1}(1-q_{0,0})+q_{-1,-1}q_{1,0}q_{0,1}+q_{1,1}q_{-1,0}q_{0,-1}.
    \end{array}
    \end{displaymath}
    The above expression is derived by  setting $\frac{q_{1,1}+q_{-1,1}+q_{0,1}}{q_{-1,-1}+q_{1,-1}+q_{0,-1}}=\frac{\gamma_{0}^{2}q_{-1,1}+\gamma_{0} q_{0,1}+q_{1,1}}{\gamma_{0}^{2}q_{-1,-1}+\gamma_{0} q_{0,-1}+q_{1,-1}}$, and proceeding with straightforward computations. 
    \item Now by applying a vertical boundary compensation step, the coefficient of the additional term vanishes, and thus, compensation stops. For $m=0$, $n>0$, $x(m,n)=z_{1}\delta_{0}^{n}$, with $\delta_{0}=\rho_{2}$, where
    \begin{displaymath}\begin{array}{c}
    z_{1}=c_{0}\frac{\delta_{0}^{2}q_{1,-1}+\delta_{0} q_{1,0}+q_{1,1}}{\delta_{0}^{2}q_{1,-1}^{(v)}+\delta_{0} q_{1,0}^{(v)}+q_{1,1}^{(v)}}=c_{0}\gamma_{0}\frac{\delta_{0}^{2}q_{-1,-1}+\delta_{0} q_{-1,0}+q_{-1,1}}{\delta_{0}(1-q_{0,0}^{(v)})-\delta_{0}^{2}q_{0,-1}^{(v)}+q_{0,1}^{(v)}},
   \end{array} \end{displaymath}
    thanks to the kernel equation $K(\gamma_0,\delta_0)=0$ and Condition B.
    \item Condition C ensures that $e_{0}=z_{1}$ (as in Section \ref{general}). Thus, the solution
    \begin{equation}
    x(m,n)=\left\{\begin{array}{ll}
         c_{0}\gamma_{0}^{m}\delta_{0}^{n},&m,n>0,  \\
        e_{0}\gamma_{0}^{m}, &m\geq 0,n=0,\\
        z_{1}\delta_{0}^{n},&m=0,n\geq 0,
    \end{array}\right.\label{ella1}
    \end{equation}
    satisfies the interior, the horizontal and the vertical boundary equations, except those at points $(1,0)$, $(0,1)$, $(1,1)$, $(0,0)$.
    \item We need an updated Condition D, so that \eqref{ella1} also satisfies the boundary equations at points $(1,0)$, $(0,1)$, $(1,1)$, $(0,0)$. By repeating exactly the same steps as those in Section \ref{general} (and omitting further details), updated Condition D should read:
    \begin{displaymath}
    \begin{array}{rl}
         q_{0,1}^{(h)}=&q_{0,1}+\frac{q_{-1,1}q_{1,0}\gamma_{0}}{q_{1,1}+q_{0,1}\gamma_{0}}, \\
          q_{1,0}^{(v)}=&q_{1,0}+\frac{q_{1,-1}q_{0,1}\delta_{0}}{q_{1,1}+q_{1,0}\delta_{0}}.
    \end{array}
    \end{displaymath}
    Finally, by using \eqref{ella1}, and the normalization equation, we can determine coefficient $c_{0}$. 
\end{itemize}

We conclude that the selection of proper boundary transition probabilities is essential for the construction of random walks with an invariant measure written as a linear combination of finitely many product-form terms. Moreover, it seems to be possible to violate certain conditions and apply a finite compensation procedure, although one need to appropriately adapt the other conditions in order to obtain an invariant measure written as a finite linear combination of product-form terms. 

In this subsection, we briefly showed how the conditions have to be adapted when condition A is violated (i.e., when $q_{1,1}>0$, $q_{-1,-1}>0$), with ultimate goal to construct a random walk with a single product-form invariant measure. Further work is required to understand how to adapt the other conditions (or to introduce new ones), in order to construct random walks with an invariant measure in the form of a finite sum of more than one product-form term. This attempt is postponed as a future work. 
\section*{Acknowledgement} The authors gratefully acknowledge stimulating discussions with Brian Fralix (School of Mathematical and Statistical Sciences, Clemson University, Clemson, SC, USA), who provided insight and expertise in the
course of this research. The authors would also like to thank the Guest Editors and the Reviewers for the insightful remarks, which helped to improve the
original exposition.

\bibliographystyle{abbrv}

\bibliography{draftv2arxiv}
\end{document}